\documentclass[11pt]{amsart}
\usepackage{amssymb
,amsthm
,amsmath
,amscd
,mathtools,
mathdots,
}
\usepackage[all]{xy}
\usepackage[top=35truemm,bottom=35truemm,left=30truemm,right=30truemm]{geometry}
\usepackage[usenames]{color}
\usepackage{float}
\usepackage{tikz}
\usepackage[utf8]{inputenc}

\newcommand{\Z}{\mathbb{Z}}

\newcommand{\R}{\mathbb{R}}
\newcommand{\C}{\mathbb{C}}

\newcommand{\K}{\mathbb{K}}
\newcommand{\Fbar}{\overline{F}}

\newcommand{\BB}{\mathcal{B}}
\newcommand{\CC}{\mathcal{C}}

\newcommand{\FF}{\mathcal{F}}

\newcommand{\HH}{\mathcal{H}}
\newcommand{\LL}{\mathcal{L}}
\newcommand{\TT}{\mathcal{T}}

\newcommand{\NN}{\mathcal{N}}
\newcommand{\WW}{\mathcal{W}}

\newcommand{\bn}{\mathbf{n}}

\newcommand{\GL}{\mathrm{GL}}
\newcommand{\SL}{\mathrm{SL}}
\newcommand{\Sp}{\mathrm{Sp}}

\newcommand{\Irr}{\mathrm{Irr}}
\newcommand{\unit}{\mathrm{unit}}
\newcommand{\temp}{\mathrm{temp}}

\newcommand{\Ind}{\mathrm{Ind}}
\newcommand{\ind}{\mathrm{ind}}
\newcommand{\Jac}{\mathrm{Jac}}

\newcommand{\Gal}{\mathrm{Gal}}

\newcommand{\Supp}{\mathrm{Supp}}
\newcommand{\Hom}{\mathrm{Hom}}

\newcommand{\diag}{\mathrm{diag}}
\renewcommand{\Re}{\mathrm{Re}}
\newcommand{\ess}{\mathrm{ess}}
\newcommand{\ze}{\mathrm{Ze}}
\newcommand{\sh}{\mathrm{Sh}}
\newcommand{\sgn}{\mathrm{sgn}}
\newcommand{\maxi}{\mathrm{max}}
\newcommand{\rest}{\mathrm{max}}
\newcommand{\Card}{\mathrm{Card}}
\newcommand{\ram}{\mathrm{ram}}
\newcommand{\unip}{\mathrm{unip}}
\newcommand{\length}{\mathrm{length}}
\newcommand{\End}{\mathrm{End}}

\newcommand{\path}{path}
\newcommand{\subq}{\dashv}
\newcommand{\Image}{\mathrm{Image}\,}
\newcommand{\VNpair}{$VN$-pair}
\newcommand{\WLpair}{$WL$-pair}

\newcommand{\Ker}{\mathrm{Ker}}
\newcommand{\seq}{\mathrm{seq}}
\newcommand{\Fil}{\mathrm{F}}
\newcommand{\Gr}{\mathrm{Gr}}

\newcommand{\st}{\mathrm{st}}

\newcommand{\ur}{\mathrm{ur}}
\newcommand{\ab}{\mathrm{ab}}

\newcommand{\af}{\mathfrak{a}}
\newcommand{\oo}{\mathfrak{o}}
\newcommand{\pp}{\mathfrak{p}}

\newcommand{\mm}{\mathfrak{m}}

\newcommand{\iif}{&\quad&\text{if }}

\newcommand{\resp}{resp.~}
\renewcommand{\1}{\mathbf{1}}
\newcommand{\ep}{\varepsilon}
\newcommand{\bs}{\backslash}

\newcommand{\half}[1]{\frac{#1}{2}}
\newcommand{\ub}[1]{\underline{#1}}
\newcommand{\wh}[1]{\widehat{#1}}
\newcommand{\wt}[1]{\widetilde{#1}}
\newcommand{\xto}[1]{\xrightarrow{#1}}

\newtheorem{thm}{Theorem}[section]
\newtheorem{lem}[thm]{Lemma}
\newtheorem{prop}[thm]{Proposition}
\newtheorem{cor}[thm]{Corollary}
\newtheorem{rem}[thm]{Remark}
\newtheorem{defi}[thm]{Definition}
\newtheorem{ex}[thm]{Example}

\numberwithin{equation}{section}
\allowdisplaybreaks
\title[
Local newforms for the general linear groups 
]{
Local newforms for the general linear groups 
\\
over a non-archimedean local field
}
\author{
Hiraku Atobe, 
Satoshi Kondo, \and 
Seidai Yasuda}
\date{\today}
\address{
Department of Mathematics, Hokkaido University,
Kita 10, Nishi 8, Kita-Ku, Sapporo, Hokkaido, 060-0810, Japan
}
\address{
Middle East Technical University \\
Northern Cyprus Campus, Kalkanli \\
Guzelyurt, Mersin 10, Turkey;
Kavli Institute for the Physics and Mathematics of the Universe\\
University of Tokyo\\
5-1-5 Kashiwanoha\\
Kashiwa 277-8583\\ Japan
}
\address{
Department of Mathematics, Hokkaido University,
Kita 10, Nishi 8, Kita-Ku, Sapporo, Hokkaido, 060-0810, Japan 
}
\email{
atobe@math.sci.hokudai.ac.jp
}
\email{
satoshi.kondo@gmail.com
}
\email{
sese@math.sci.hokudai.ac.jp
}

\begin{document}
\maketitle

\begin{abstract}
In \cite{JPSS2}, Jacquet--Piatetskii-Shapiro--Shalika defined a family of compact open subgroups 
of $p$-adic general linear groups indexed by non-negative integers,
and established the theory of local newforms for irreducible generic representations. 
In this paper, 
we extend their results to all irreducible representations. 
To do this, 
we define a new family of compact open subgroups indexed by certain tuples of non-negative integers.
For the proof, we introduce the Rankin--Selberg integrals for Speh representations. 
\end{abstract}

\tableofcontents

\section{Introduction}

\subsection{Background}
The theory of newforms is fascinating and plays an important role in the theory of automorphic forms. 
It was first studied in the early 70s by Atkin--Lehner \cite{AL} and Li \cite{Li} in terms of classical modular forms, 
and by Casselman \cite{C} in terms of local newforms on $\GL_2$. 
Their results become a bridge between classical modular forms and automorphic representations of $\GL_2$. 
Casselman's result was generalized to $\GL_n$ by Jacquet--Piatetskii-Shapiro--Shalika \cite{JPSS2} (see also Jacquet's erratum \cite{J}) in the 80s. 
Another proof was given by Matringe \cite{Mat} in 2013.
\par

After their works, the theory of local newforms was established  
\begin{itemize}
\item
for $\mathrm{PGSp}_4$ and for $\widetilde{\SL}_2$, which is the double cover of $\SL_2$, 
by Roberts--Schmidt \cite{RS1,RS2};
\item
for $\mathrm{GSp}_4$ by Okazaki \cite{Ok};
\item
for $\mathrm{U}(1,1)$ by Lansky--Raghuram \cite{LR}; 
\item
for unramified $\mathrm{U}(2,1)$ by Miyauchi \cite{M1,M2,M3,M4}.
\end{itemize}
In 2010, Gross gave a conjecture on the local newforms for $\mathrm{SO}_{2n+1}$ in a letter to Serre 
(see an expansion \cite{G} of this letter). 
It is a natural extension of the $\GL_2$ case \cite{C} and the $\mathrm{PGSp}_4$ case \cite{RS1}.
This conjecture was proven for generic supercuspidal representations by Tsai \cite{Ts}.
\vskip 10pt

One has to notice that 
in all previous works, representations are assumed to be \emph{generic}.
For $\GL_n$, this assumption might be reasonable 
since all local components of an arbitrary irreducible cuspidal automorphic representation of $\GL_n$ are generic. 
However, for other groups, this assumption is too strong 
because there are many irreducible cuspidal automorphic representations
whose local components are non-generic (and non-tempered), 
such as the Saito--Kurokawa lifting of $\mathrm{PGSp}_4$. 
\vskip 10pt

In this paper, we generalize the results in \cite{JPSS2} to all the irreducible representations. 
Namely, we extend the theory of local newforms to non-generic representations in the case of $\GL_n$. 
By considering the endoscopic classification, 
our results would be useful for the study of local newforms for classical groups in the future.  

\subsection{Main results}\label{sec.intro.main}
Let us describe our results.
Let $F$ be a non-archimedean local field of characteristic zero
with the ring of integers $\oo$ and the maximal ideal $\pp$. 
Fix a non-trivial additive character $\psi$ of $F$ which is trivial on $\oo$ but not on $\pp^{-1}$. 
We denote by $q$ the order of $\oo/\pp$. 
\par

For an integer $n \geq 1$, 
set 
$\Lambda_n 
= \{(\lambda_1, \dots, \lambda_n) \in \Z^n \;|\; 0 \leq \lambda_1 \leq \dots \leq \lambda_n\}$. 
We regard $\Lambda_n$ as a totally ordered monoid 
with respect to the lexicographic order. 
For $\lambda = (\lambda_1, \dots, \lambda_n)$, we set $|\lambda| = \lambda_1+\dots+\lambda_n$.
\par

We set $G_n = \GL_n(F)$. 
For $\lambda = (\lambda_1,\dots,\lambda_n) \in \Lambda_n$, 
we define a subgroup $\K_{n,\lambda}$ of $\GL_n(\oo)$ 
by 
\[
\K_{n,\lambda} = \{(k_{i,j}) \in \GL_n(\oo)\;|\; 
k_{i,j} \equiv \delta_{i,j} \bmod \pp^{\lambda_i},\; 
1 \leq i,j \leq n
\},
\]
where $\delta_{i,j}$ is the Kronecker delta.
\par

Let $\pi$ be an irreducible smooth complex representation of $G_n$.  
Godement--Jacquet \cite{GJ} associated  two local factors $L(s,\pi)$ and $\ep(s,\pi,\psi)$ with $\pi$. 
By \cite[(5.1) Th{\'e}or{\`e}me (i)]{JPSS2} and 
\cite[Corollary 3.6]{GJ}
(or by the local Langlands correspondence \cite{HT}, \cite{H}), 
we have $\ep(s,\pi,\psi) = \ep(0,\pi,\psi) q^{-c_\pi s}$
for some non-negative integer $c_\pi$. 
We call $c_\pi$ the \emph{conductor of $\pi$}.
\par

Set $\pi^{(0)} = \pi$, 
and $\pi^{(i)}$ to be the highest derivative of $\pi^{(i-1)}$ 
in the sense of Bernstein--Zelevinsky \cite{BZ} for $i = 1,\dots, n$. 
(Note that our notation is different from the original one in \cite{BZ}.)
It is known that $\pi^{(i)}$ is irreducible so that one can consider its conductor $c_{\pi^{(i)}}$.
We then define $\lambda_\pi = (\lambda_1, \dots, \lambda_n)$ by
\[
\lambda_k = c_{\pi^{(n-k)}} - c_{\pi^{(n-k+1)}}
\]
for $1 \leq k \leq n$.
In Section \ref{sec.def_lambda} (especially, in Proposition \ref{prop_lambda}) below, 
we will see that $\lambda_\pi \in \Lambda_n$.
We note that $|\lambda_\pi| = c_{\pi}$.
\par

We denote by $\pi^{\K_{n,\lambda}}$ the $\K_{n,\lambda}$-invariant subspace of $\pi$, 
which is finite dimensional. 
Our main theorem is stated as follows: 

\begin{thm}[Theorems \ref{main1}, \ref{conj1-2}]\label{main}
Let $\pi$ be an irreducible representation of $G_n$. 
\begin{enumerate}
\item
For $\lambda \in \Lambda_n$, we have
\[
\dim(\pi^{\K_{n,\lambda}}) = 
\left\{
\begin{aligned}
&1 \iif \lambda = \lambda_\pi, \\
&0 \iif \lambda < \lambda_\pi. 
\end{aligned}
\right. 
\]

\item
If $\lambda \in \Lambda_n$ satisfies that $|\lambda| < |\lambda_\pi|$, 
then $\pi^{\K_{n,\lambda}} = 0$.
\end{enumerate}
\end{thm}

We call any nonzero element in $\pi^{\K_{n,\lambda_\pi}}$ a \emph{local newform} of $\pi$.
Using Theorem \ref{main}, 
we can give a characterization of the conductor in terms of the dimensions of fixed parts, that is,
\[
c_\pi = \min\left\{ |\lambda| \;\middle|\; \pi^{\K_{n,\lambda}} \not= 0\right\}.
\]
\vskip 10pt

Note that when $\pi$ is generic, 
since $\pi^{(i)}$ is the trivial representation $\1_{G_0}$ for any $i \geq 1$, 
we have $\lambda_\pi = (0,\dots,0,c_\pi)$. 
In this case, $\K_{n,\lambda_\pi}$ is nothing but the compact group introduced 
by Jacquet--Piatetskii-Shapiro--Shalika \cite{JPSS2}. 
Hence Theorem \ref{main} (1) is an extension of a result in loc.~cit. 
\vskip 10pt

According to the Zelevinsky classification, 
the set of isomorphism classes of irreducible representations of $G_n$ 
is classified by multisegments. 
We recall it in Section \ref{sec.zel}. 
When $\pi = Z(\mm)$ is the irreducible representation associated with a multisegment $\mm$, 
we have another description of $\lambda_\pi$ in terms of $\mm$ (Proposition \ref{prop_lambda}),
which allows us to compute $\lambda_\pi$ in many important cases (Example \ref{ex_lambda}). 
Moreover, Corollary \ref{cor:fund4} tells us how to compute $\lambda_\pi$ inductively in general. 
\vskip 10pt

The proof of Theorem \ref{main} takes following three steps:
\begin{description}
\item[Step 1]
Reduce to two cases; 
the case where $\pi$ is of type $\chi$ with an unramified character $\chi$ of $F^\times$,  
and the case where $L(s,\pi) = 1$. 
Here, we say that an irreducible representation $\pi$ is \emph{of type $\chi$}
if $\pi = Z(\Delta_1+\dots+\Delta_r)$ such that 
for $i=1,\ldots,r$, the segment 
$\Delta_i$ is of the form $[a_i,b_i]_\chi$ for some integers $a_i,b_i$ satisfying $a_i \le b_i$.

\item[Step 2]
Prove Theorem \ref{main} for $\pi$ of type $\chi$ with an unramified character $\chi$ of $F^\times$. 

\item[Step 3]
Prove Theorem \ref{main} for $\pi$ such that $L(s,\pi) = 1$. 
\end{description}

Let us give the detail of each step.

\subsection{Reduction}
Using the Mackey theory, 
we study the $\K_{n,\lambda}$-invariant subspaces of 
parabolically induced representations in Section \ref{sec.inv}. 
To do this, in Section \ref{sec.o-modules}, 
we relate $\Lambda_n$ with 
the set $|\CC^n|$ of isomorphism classes $[M]$ of $\oo$-modules 
such that $M$ is generated by at most $n$ elements.
In Section \ref{sec.inv}, we associate 
a compact open subgroup $\K_{n,[M]}$ of $G_n$ with $[M] \in |\CC^n|$. 
If $\lambda = (\lambda_1,\dots,\lambda_n) \in \Lambda_n$ and 
$M = \oplus_{i=1}^n \oo/\pp^{\lambda_i}$, then $\K_{n,[M]} = \K_{n,\lambda}$. 
Proposition \ref{mackey} says that 
the $\K_{n,[M]}$-invariant subspace of 
a parabolically induced representation
decomposes into a direct sum indexed by certain filtrations on $M$ by $\oo$-modules. 
In particular, this proposition together with Corollary \ref{fil0} reduces Theorem \ref{main}
to the following two types of irreducible representations: 
\begin{itemize}
\item
$\pi \in \Irr(G_n)$ of type $\chi$ 
with a fixed unramified character $\chi$ of $F^\times$. 

\item
$\pi \in \Irr(G_n)$ such that $L(s,\pi) = 1$.
\end{itemize}

\subsection{The case where $\pi$ is of type $\chi$}
In Section \ref{sec.proofs1}, we prove Theorem \ref{main} for 
irreducible representations $\pi \in \Irr(G_n)$ of type $\chi$ 
with a fixed unramified character $\chi$ of $F^\times$. 
\vskip 10pt

In the proof of Theorem \ref{main} (1), 
we first consider the case where $\pi$ is a ladder representation. 
The main ingredient in this case is 
Tadi\'c's determinantal formula established by Lapid--M{\'i}nguez \cite{LMi}.
This formula describes $\pi$ explicitly as an alternating sum of standard modules. 
The key point is that 
the standard modules appearing here are 
parabolically induced representations from one-dimensional representations. 
In particular, for $[M] \in |\CC^n|$, 
the determinantal formula together with Proposition \ref{mackey}
expresses the dimension of $\pi^{\K_{n,[M]}}$ explicitly as an alternating sum 
of the numbers of certain filtrations on $M$ by $\oo$-modules (Proposition \ref{alt}). 
Surprisingly, there are many cancellations in this alternating sum (Lemma \ref{vanish}). 
From this lemma, we can deduce Theorem \ref{main} (1) 
for a ladder representation $\pi$ of type $\chi$.
For these miraculous cancellations, see Example \ref{ex_ladder}. 
\vskip 10pt

The proof of Theorem \ref{main} (1) for general $\pi$ of type $\chi$ 
is by induction on a certain totally ordered set. 
The key is Proposition \ref{lem:fund3}, whose proof
relies on 
a highly non-trivial result of Knight--Zelevinsky \cite{KZ}
which describes the Zelevinsky dual of $\pi$
(see also Proposition \ref{prop:KZ}). 
\vskip 10pt

We reduce the proof of Theorem \ref{main} (2) to the case 
where $\pi$ is a Steinberg representation. 
In this case, by Tadi\'c's determinantal formula
(or by the definition of the Steinberg representations in Harish-Chandra \cite{HC}), 
we can express $\pi^{\K_{n,\lambda}}$ 
explicitly as an alternating sum of the numbers of certain filtrations on 
the $\oo$-module corresponding to $\lambda$. 
We realize this alternating sum as a coefficient of 
certain formal power series in one variable whose coefficients are in  
a graded ring.
By giving another description of this formal power series, 
we deduce that $\pi^{\K_{n,\lambda}} = 0$.  

\subsection{The case where $L(s,\pi) = 1$}
In Section \ref{sec.pf_L=1}, we firstly prove Theorem \ref{main} (2) for $\pi$ with $L(s,\pi) = 1$. 
We reduce the proof to the case where $\pi$ is cuspidal. 
In this case, Lemma \ref{lem:nilp} says that 
certain Hecke operators depending on $\lambda \in \Lambda_n$ 
act on $\pi$ as nilpotent endomorphisms. 
We consider the Godement--Jacquet integral $Z(\Phi,s,f)$ defined in \cite{GJ}. 
From this lemma, it follows that if $\pi^{\K_{n,\lambda}} \not= 0$, 
then we can obtain data $\Phi$ and $f$ such that 
$Z(\Phi,s,f)$ is a nonzero constant whereas $Z(\hat{\Phi},s,\check{f}) \in q^{|\lambda| s} \C[[q^{-s}]]$.
Since $\ep(s,\pi,\psi) = \ep(0,\pi,\psi)q^{-|\lambda_\pi|s}$, 
by the functional equation, we conclude that $|\lambda| \geq |\lambda_\pi|$.
\vskip 10pt

By Proposition \ref{mackey}, 
we can reduce Theorem \ref{main} (1) for $\pi$ with $L(s,\pi) = 1$
to the case where $\pi = Z(\Delta)$ for a segment $\Delta$ (Lemma \ref{thm.reduction}). 
The key point here is that the matrices defined by the multiplicities of irreducible representations appearing in standard modules are ``triangular'' and unipotent
(\cite[7.1 Theorem]{Z}). 
\vskip 10pt

Finally, we prove Theorem \ref{main} (1) for $\pi = Z(\Delta)$ with $L(s,\pi) = 1$. 
Slightly generally, we do it in Section \ref{s.ess_speh}
for Speh representations $\Sp(\pi_\temp, m)$ 
with an irreducible tempered representation $\pi_\temp$ of $G_n$. 
For the notation of Speh representations, see Example \ref{ex_lambda} (4). 
The proof of this case is rather an analogue of the generic case in \cite{JPSS2}.
Namely, it is an application of the theory of Rankin--Selberg integrals. 
To carry out the proof, we establish this theory for Speh representations in Section \ref{s.RS}. 

\subsection{Rankin--Selberg integrals for Speh representations}
The theory of Rankin--Selberg integrals was developed by 
Jacquet--Piatetskii-Shapiro--Shalika \cite{JPSS}. 
These integrals are integrations of 
products of Whittaker functions of two irreducible representations of $G_n$ and $G_m$, 
and they represent the Rankin--Selberg $L$-functions.
Since representations are required to admit non-trivial Whittaker functions, they must be generic. 
As an application of Rankin--Selberg integrals for $G_n \times G_{n-1}$, 
the theory of local newforms for generic representations of $G_n$
was established in \cite{JPSS2}. 
\vskip 10pt

To prove Theorem \ref{main} (1) for Speh representations, 
we need to extend the theory of Rankin--Selberg integrals to the case of Speh representations. 
In the equal rank case, this extension was done by Lapid--Mao \cite{LM}.
In their paper, instead of Whittaker models, 
they used two models of a Speh representation that are called 
the \emph{Zelevinsky model} and the \emph{Shalika model}\footnote{
As mentioned in \cite{LM}, 
this terminology does not coincide with the standard notion of the Shalika model in the literature.
This model was also used in the theory of twisted doubling \cite{CFGK}
established by Cai--Friedberg--Ginzburg--Kaplan, 
in which it is called the \emph{$(k,c)$ model}.
}. 
\vskip 10pt

For our purpose, we need the Rankin--Selberg integrals in the ``almost equal rank case'', 
which are easier than the equal rank case.
The Zelevinsky model is a direct generalization of Whittaker model  
so that we can easily extend the theory of Rankin--Selberg integrals using this model 
(Theorem \ref{RS}). 
On the other hand, the Shalika model has
an important property of the Whittaker model (Theorem \ref{inner}), 
which we need for the proof of Theorem \ref{main} (1) for Speh representations.
To transfer the Rankin--Selberg integrals in the Zelevinsky models 
to the ones in the Shalika models, 
we use the model transition established by Lapid--Mao (see Proposition \ref{transition}).
\vskip 10pt

After establishing the Rankin--Selberg integrals in the Shalika models, 
the proof of Theorem \ref{main} (1) for Speh representations $\pi$ with $L(s,\pi) = 1$
is exactly the same as in the generic case \cite{JPSS2}. 
We do not compute the greatest common divisors of the Rankin--Selberg integrals in general
(see Proposition \ref{gcd}). 
This is a main reason why this method cannot be applied 
to Speh representations $\pi$ with $L(s,\pi) \not= 1$. 
However, as an application of Theorem \ref{main} (1) for all cases, 
we can specify the greatest common divisor
when the Speh representation of the group of smaller rank is unramified 
(see Theorem \ref{essential}). 

\subsection{Organization}
This paper is organized as follows.
In Section \ref{sec.statements}, we state the main results (Theorems \ref{main1} and \ref{conj1-2}). 
We give two definitions of $\lambda_\pi$ (Proposition \ref{prop_lambda}) 
and explain how to compute it (Corollary \ref{cor:fund4}). 
Some important examples of $\lambda_\pi$ are given in Example \ref{ex_lambda}.
Propositions \ref{prop_lambda} and \ref{lem:fund3} are proven in Section \ref{sec.lambda}. 
After preparing several facts on $\oo$-modules in Section \ref{sec.modules}, 
we prove the Mackey decomposition of the $\K_{n,[M]}$-invariant subspace of 
a parabolically induced representation (Proposition \ref{mackey}) in Section \ref{s.mackey}. 
It reduces the proofs of main results to two cases: 
$\pi \in \Irr(G_n)$ of type $\chi$ 
with a fixed unramified character $\chi$ of $F^\times$, 
and 
$\pi \in \Irr(G_n)$ such that $L(s,\pi) = 1$.
For the former case, Theorems \ref{main1} and \ref{conj1-2} are proven in Section \ref{sec.proofs1}. 
In Section \ref{sec.pf_L=1}, we treat the latter case. 
More precisely, for the latter case, 
we prove Theorem \ref{conj1-2}, 
but we reduce Theorem \ref{main1} to the case where $\pi$ is a Speh representation. 
Theorem \ref{main1} for Speh representations $\pi$ with $L(s,\pi) = 1$ is proven in Section \ref{s.ess_speh} 
after establishing the theory of Rankin--Selberg integrals for Speh representations in Section \ref{s.RS}. 

\subsection*{Acknowledgement}
The first author was supported by JSPS KAKENHI Grant Number 19K14494.
The second author was supported by WPI Initiative, MEXT, Japan.
The third author was supported by JSPS KAKENHI Grant Number 21H00969.
The authors thank the referee for careful reading of the  long manuscript and for giving valuable comments.

\subsection*{Notation}
Let $F$ be a non-archimedean local field of characteristic zero.
Denote the ring of integers and its maximal ideal by $\oo$ and $\pp$, respectively. 
Fix a uniformizer $\varpi$ of $\oo$, 
and normalize the absolute value $|\cdot|$ on $F$ so that $|\varpi| = q^{-1}$, 
where $q = \# (\oo/\pp)$.
We fix a non-trivial additive character $\psi$ of $F$ 
such that $\psi$ is trivial on $\oo$ but non-trivial on $\pp^{-1}$.
\par

For an integer $n \ge 1$ and for a commutative ring $R$, 
we let $M_n(R)$ denote the $R$-module of
$n$-by-$n$ matrices with entries in $R$.
\par

In this paper, all representations are assumed to be smooth. 
For a representation $\pi$ of $\GL_n(F)$, 
its contragredient representation is denoted by $\wt{\pi}$.

\section{Statements of the main results}\label{sec.statements}
In this section, we fix notations and state the main results. 

\subsection{The Zelevinsky classification}\label{sec.zel}
We recall the Zelevinsky classification \cite{Z} 
of irreducible representations of $G_n = \GL_n(F)$. 
For a smooth representation $\pi$ of $G_n$ and a character $\chi$ of $F^\times$, 
the twisted representation $g \mapsto \pi(g)\chi(\det g)$ is denoted by $\pi \chi$. 
The set of equivalence classes of irreducible representations of $G_n$ is denoted by $\Irr(G_n)$.
\par

When $\pi_1, \dots, \pi_r$ are smooth representations of $G_{n_1}, \dots, G_{n_r}$, respectively,
with $n_1+\dots+n_r = n$, 
we write $\pi_1 \times \dots \times \pi_r$ for the parabolically induced representation of $G_n$
from $\pi_1 \otimes \dots \otimes \pi_r$ 
via the standard parabolic subgroup whose Levi subgroup is 
$G_{n_1} \times \dots \times G_{n_r}$. 
\par

A \emph{segment} $\Delta$ is a finite set of representations of the form 
\[
[x,y]_\rho = \{\rho|\cdot|^x, \rho|\cdot|^{x+1}, \dots, \rho|\cdot|^y\}, 
\]
where $\rho$ is an irreducible cuspidal representation of $G_d$ for some $d \ge 1$, 
and $x,y \in \R$ with $x \equiv y \bmod \Z$ and $x \leq y$. 
We write $l(\Delta) = y-x+1$ and call it the \emph{length} of $\Delta$.
\par

Let $\Delta = [x,y]_\rho$ be a segment. 
Then the parabolically induced representation
\[
\rho|\cdot|^x \times \rho|\cdot|^{x+1} \times \dots \times \rho|\cdot|^y
\]
of $G_{dl(\Delta)}$ has a unique irreducible subrepresentation.
We denote it by $Z(\Delta)$. 
For example, if $\rho = \chi$ is a character of $F^\times$, 
then $Z([x,y]_\chi) = |\det|^{\half{x+y}}\chi(\det)$ is a 
one-dimensional representation of $G_{y-x+1}$.
\par

Let $r \geq 1$.
For $i=1,\ldots,r$,
let $\Delta_i = [x_i,y_i]_{\rho_i}$ be a segment,
$n_i \geq 1$ an integer such that
$\rho_i$ is a cuspidal representation of $G_{n_i}$. 
When $\rho_i$ is unitary and the inequalities
\[
x_1+y_1 \geq \dots \geq x_r+y_r
\]
hold, the parabolically induced representation
\[
Z(\Delta_1) \times \dots \times Z(\Delta_r)
\]
has a unique irreducible subrepresentation. 
We denote it by $Z(\mm)$, 
where $\mm$ denotes the multisegment
$\mm = \Delta_1 + \dots + \Delta_r$. 
The Zelevinsky classification says that 
for any irreducible representation $\pi$ of $G_n$, 
there exists a unique multisegment $\mm = \Delta_1 + \dots + \Delta_r$ 
such that $\pi \cong Z(\mm)$.
\par

When $x_1 > \dots > x_t$ and $y_1 > \dots > y_t$, 
the irreducible representation 
\[
Z([x_1,y_1]_\rho, \dots, [x_t,y_t]_\rho)
= Z([x_1,y_1]_\rho + \dots + [x_t,y_t]_\rho)
\]
is called a \emph{ladder representation}. 
A ladder representation of the form 
$Z([x,y]_\rho, [x-1,y-1]_\rho, \dots, [x-t+1,y-t+1]_\rho)$
for some positive integer $t$ 
is called a \emph{Speh representation}.

\subsection{Main results}
Fix $n \geq 1$.
Let $\Lambda_n$ be the subset of $\Z^n$ consisting of  
$\lambda = (\lambda_1, \dots, \lambda_n) \in \Z^n$ such that $0 \leq \lambda_1 \leq \dots \leq \lambda_n$. 
Note that $\Lambda_n$ is a submonoid of $\Z^n$. 
We endow $\Lambda_n$ with the total order induced by the lexicographic order, i.e., 
for $\lambda = (\lambda_1, \dots, \lambda_n), \lambda' = (\lambda'_1, \dots, \lambda'_n) \in \Lambda_n$, 
we write $\lambda < \lambda'$ if and only if 
there exists $1 \leq i \leq n$ such that $\lambda_j = \lambda'_j$ for $j < i$ and $\lambda_i < \lambda'_i$. 
\par

For $\lambda = (\lambda_1, \dots, \lambda_n) \in \Lambda_n$, 
define $\K_{n,\lambda}$ to be 
the subgroup of $G_n(\oo) = \GL_n(\oo)$ consisting of matrices 
$k = (k_{i,j})_{1 \le i,j \le n}$ 
such that 
\[
k_{i,j} \equiv \delta_{i,j} \bmod \pp^{\lambda_i}
\]
for any $1\leq i,j \leq n$. 
For example, if $n = 4$ and $\lambda = (0,0,1,2)$, then
\[
\K_{4, (0,0,1,2)} = \begin{pmatrix}
\oo & \oo & \oo & \oo \\
\oo & \oo & \oo & \oo \\
\pp & \pp & 1+\pp & \pp \\
\pp^2 & \pp^2 & \pp^2 & 1+\pp^2 
\end{pmatrix}
\cap \GL_4(\oo).
\]
\par

In Section \ref{sec.intro.main}, 
we defined $\lambda_\pi \in \Lambda_n$ for any $\pi \in \Irr(G_n)$. 
The main results are as follows.

\begin{thm}\label{main1}
Let $\pi \in \Irr(G_n)$.
Then the $\K_{n,\lambda_\pi}$-invariant subspace $\pi^{\K_{n,\lambda_\pi}}$ is one-dimensional. 
Moreover, if $\lambda \in \Lambda_n$ satisfies 
$\lambda < \lambda_\pi$, then $\pi^{\K_{n,\lambda}} = 0$. 
\end{thm}

For $\lambda=(\lambda_1,\ldots,\lambda_n) \in \Lambda_n$, we write
$|\lambda|$ for $\lambda_1 + \cdots + \lambda_n$.

\begin{thm}\label{conj1-2}
Let $\pi \in \Irr(G_n)$.
If $\lambda \in \Lambda_n$ satisfies 
$|\lambda| < |\lambda_\pi|$, then $\pi^{\K_{n,\lambda}} = 0$. 
\end{thm}

\subsection{Definition of $\lambda_\mm$}\label{sec.def_lambda}
For an irreducible representation $\pi$ of $G_n$, 
we defined $\lambda_\pi \in \Z^n$ in Section \ref{sec.intro.main}. 
Here we describe it in terms of multisegments, 
which then implies that $\lambda_\pi \in \Lambda_n$. 
\par

A segment $\Delta$ is written as $\Delta=[a,b]_\rho$ 
where $a,b \in \Z$ with $a \leq b$, 
and $\rho$ is a cuspidal representation of $G_d$ for some $d \geq 0$. 
We write a multisegment as a sum $\mm = \Delta_1+\dots+\Delta_r$ of segments 
where $r$ is a non-negative integer.
We call the integer $r$ the \emph{cardinality} of $\mm$
and denote it by $\Card(\mm)$. 
Recall that we set $l(\Delta) = b-a+1$. 
We write $l(\mm)$ for the sum 
$l(\Delta_1) + \cdots + l(\Delta_r)$ and call 
$l(\mm)$ the \emph{length} of $\mm$.
\par

If $a < b$, we write $\Delta^{-}$ for the segment $[a,b-1]_\rho$.
When $a=b$, we understand $\Delta^{-}$ to be the empty multisegment. 
We set $\mm^{-} = \Delta_1^{-} + \cdots + \Delta_r^{-}$.
By the fundamental result of Zelevinsky \cite[8.1 Theorem]{Z},
the highest derivative of $Z(\mm)$ is equivalent to
$Z(\mm^{-})$.
\par

We call $\Delta = [a,b]_\rho$ \emph{unipotent}
if $\rho$ is an unramified character of $F^\times$.
Similarly, we say that $\mm = \Delta_1 + \cdots + \Delta_r$ is \emph{unipotent}
if $\Delta_i$ is unipotent for $i=1,\ldots,r$.
Fix an unramified character $\chi$ of $F^\times$.
We say that a multisegment $\mm = \Delta_1 + \cdots + \Delta_r$
is \emph{of type $\chi$} if for $i=1,\ldots,r$, the segment 
$\Delta_i$ is of the form $[a_i,b_i]_\chi$ for some
integers $a_i,b_i$ satisfying $a_i \le b_i$.
\par

We denote by $\mm^\sharp$ the unique multisegment
such that $Z(\mm^\sharp)$ is equivalent to the Zelevinsky dual of $Z(\mm)$ 
(see, e.g., \cite[Section 7]{P}). 
We denote by $\mm^\ram$ the multisegment $((\mm^\sharp)^{-})^\sharp$.
When $\pi = Z(\mm)$, we set $\pi^\ram = Z(\mm^\ram)$.
We use ``ram'' only for unipotent multisegments. 
For an example of $\mm^\ram$, see Section \ref{sec:ZD} below. 
\par

When $n' < n$, we regard $\Lambda_{n'}$ as a submonoid of $\Lambda_n$ 
via the inclusion $\Lambda_{n'} \hookrightarrow \Lambda_n$ 
given by $(\lambda_1, \dots, \lambda_{n'}) \mapsto (0,\dots,0,\lambda_1,\dots,\lambda_{n'})$. 

\begin{defi}
Let $\mm$ be a multisegment. 

\begin{enumerate}
\item
If $\mm = \Delta_1+\dots+\Delta_r$ with 
$\Delta_i = [a_i,b_i]_{\rho_i}$ being not unipotent for all $i = 1, \dots, r$, 
then we set 
\[
\lambda_{\mm} = \sum_{i=1}^r (0,\dots,0,\underbrace{c_{\rho_i}, \dots, c_{\rho_i}}_{l(\Delta_i)}), 
\]
where $c_{\rho_i}$ is the conductor of $\rho_i$. 
Note that $c_{\rho_i} > 0$ for $1 \leq i \leq r$ by \cite[(5.1) Th{\'e}or{\`e}me]{JPSS2}. 

\item
If $\mm$ is unipotent, and if we write $\mm^{\ram} = \Delta_1+\dots+\Delta_r$, 
then we set 
\[
\lambda_{\mm} = \sum_{i=1}^r (0,\dots,0,\underbrace{1, \dots, 1}_{l(\Delta_i)}). 
\]

\item
In general, we decompose $\mm$ as  $\mm = \mm' + \mm_\unip$, 
where $\mm_\unip$ is unipotent, 
and each segment in $\mm'$ is not unipotent. 
Then we set 
\[
\lambda_{\mm} = \lambda_{\mm'} + \lambda_{\mm_\unip}.
\]
\end{enumerate}
\end{defi}

As seen in the next proposition, 
this is an alternative definition of $\lambda_\pi$. 
\begin{prop}\label{prop_lambda}
Let $n \geq 1$ and let $\pi = Z(\mm)$ be an irreducible representation of $G_n$
corresponding to a multisegment $\mm$. 
Then we have $\lambda_\pi = \lambda_\mm$.
\end{prop}
This proposition will be proven in Section \ref{sec.proof_lambda} below. 
We now give some examples. 

\begin{ex}\label{ex_lambda}
Let $\pi$ be an irreducible representation of $G_n$. 

\begin{enumerate}
\item
When $L(s,\pi) = 1$, then $\pi = Z(\Delta_1 + \dots + \Delta_r)$ with $\Delta_i$ not unipotent. 
If $\pi = Z(\Delta)$ with a segment $\Delta = [x,y]_\rho$, then we have
\[
\lambda_\pi = \lambda_\Delta 
= (\underbrace{0, \dots, 0}_{n-l(\Delta)}, \underbrace{c_\rho, \dots, c_\rho}_{l(\Delta)}) \in \Lambda_n. 
\]
Here, we note that $c_\rho > 0$.
In general, if $\pi = Z(\Delta_1 + \dots + \Delta_r)$, 
we have
\[
\lambda_\pi = \lambda_{\Delta_1} + \dots + \lambda_{\Delta_{r}} \in \Lambda_n.
\]

\item
When $\pi = Z([x_1,y_1]_\chi, \dots, [x_t,y_t]_\chi) \in \Irr(G_n)$ 
is a ladder representation 
of type $\chi$, 
where $\chi$ is an unramified character of $F^\times$,  
we have
\[
\lambda_\pi = \sum_{i=2}^t (0, \dots, 0, \underbrace{1, \dots, 1}_{\max\{y_i-x_{i-1}+2,0\}}) \in \Lambda_n.
\]
Indeed, by the description of the Zelevinsky duals of ladder representations in \cite[Section 3]{LMi}
(see also Section \ref{sec:ZD} below), 
we have
\[
\pi^{\ram} = Z([x_1-1,y_2]_\chi, [x_2-1,y_3]_\chi, \dots, [x_{t-1}-1,y_t]_\chi).
\]
Here, if $x_{i-1}-1 > y_i$, we ignore $[x_{i-1}-1,y_i]_\chi$. 

\item
Let $t \ge 1$, 
and let $\pi_i \in \Irr(G_{n_i})$ be as in either (1) or (2) above for $1 \le i \le t$.   
Assume $\pi=\pi_1 \times \cdots \times \pi_t$ is irreducible.
Then we have
$\lambda_\pi=\lambda_{\pi_1}+\cdots+\lambda_{\pi_t}$.

\item
Let $\pi$ be an irreducible tempered representation of $G_n$. 
Then the parabolically induced representation
\[
\pi|\cdot|^{-\half{m-1}} \times \pi|\cdot|^{-\half{m-3}} \times \dots \times \pi|\cdot|^{\half{m-1}}
\]
of $G_{nm}$ has a unique irreducible subrepresentation $\sigma$, 
which is denoted by $\Sp(\pi, m)$. 
Note that $\sigma$ is a (unitary) Speh representation. 
Combining the cases above, we obtain
\[
\lambda_\sigma = 
(\underbrace{0,\dots,0}_{(n-1)m}, \underbrace{c_\pi, \dots, c_\pi}_{m}) \in \Lambda_{nm}.
\]
\end{enumerate}
\end{ex}

\begin{rem}
In the appendix of the paper \cite{KY0} by the second and third
authors, they introduce a notion of mirahoric representations
(see A.1.6 loc.cit.).
Let us recall the definition.
Two segments $\Delta$ and $\Delta'$ are said to be {\em tightly linked} if
and only if
they are linked and either $\Delta$ is not unipotent or $\Delta \cap
\Delta'$ is non-empty. Let $\pi=L(\mm)$ be an irreducible representation
associated with a multisegment $\mm$ in the Langlands classification, i.e.,
$\pi$ is the Zelevinsky dual of $Z(\mm)$. 
They defined $\pi$ to be {\em mirahoric} if any two segments in $\mm$ are not
tightly linked.
In terms of the setup in this paper, the class of mirahoric representations
is equal to the class of irreducible representations $\pi$ such that
$\lambda_\pi=(0,...,0,c)$ for some $c$.   This can be seen from 
their proposition \cite[Proposition A.15]{KY0}, 
which says that a representation $\pi$ is mirahoric if and only if
the conductor of the highest derivative of $\pi$ is zero.
Hence, a main result \cite[Proposition A.3]{KY0} in the appendix 
can be interpreted as a special case of Theorem \ref{main1} restricted
to the mirahoric representations.

An irreducible representation $\pi=L(\mm)$
is generic if and only if any two segments of
$\mm$ are not linked. 
Therefore a generic representation is mirahoric.
However, a simple multisegment such as $\mm=[0,1]_\rho + [2,3]_\rho$,
where $\rho$ is an unramified character, gives a mirahoric representation
$L(\mm)$ which is not generic. (This is one of the reasons 
for treating the
unipotent case and the case $L(s, \pi)=1$ separately.)
\end{rem}

\subsection{Computation of $\lambda_{\mm}$}
When $\mm$ is a general unipotent multisegment, 
it is difficult to compute $\lambda_{\mm}$ directly from the definition.
In this subsection, we explain how to compute $\lambda_\mm$ efficiently. 
\par

Let $\mm$ be a unipotent multisegment. 
We may assume that $\mm$ is of type $\chi$ for some unramified character $\chi$ of $F^\times$.
We denote by $\mm_\maxi$ the set of
segments $\Delta$ in $\mm$ such that $\Delta$ is
maximal with respect to the inclusion among the segments
in $\mm$. We regard $\mm_\maxi$ as a multisegment
in which each segment has multiplicity at most $1$.
We set $\mm^\rest = \mm - \mm_\maxi$.
For example, if
\[
\mm = [0,0]_\chi + [1,2]_\chi
+ [1,2]_\chi + [2,2]_\chi,
\]
then we have
\[
\mm_\maxi = [0,0]_\chi + [1,2]_\chi
\]
and
\[
\mm^\rest = [1,2]_\chi + [2,2]_\chi.
\]
\par

\begin{prop} \label{lem:fund3}
We have
$\mm^\ram = (\mm_\maxi)^\ram
+ (\mm^\rest)^\ram$.
\end{prop}
We will prove this proposition in Section \ref{sec.proof_max} below. 

\begin{cor} \label{cor:fund4}
We have
$\lambda_{\mm} = \lambda_{\mm_\maxi}
+ \lambda_{\mm^\rest}$.
\end{cor}
\begin{proof}
Write $(\mm_\maxi)^\ram = \Delta_1+\dots+\Delta_r$ 
and $(\mm^\rest)^\ram = \Delta_{r+1} + \dots + \Delta_t$.
Then $\mm^\ram =  \Delta_1+\dots+\Delta_t$ by Proposition \ref{lem:fund3}.
Hence we have
\begin{align*}
\lambda_{\mm} 
&= 
\sum_{i=1}^t (0,\dots,0,\underbrace{1, \dots, 1}_{l(\Delta_i)})
\\&=
\sum_{i=1}^r (0,\dots,0,\underbrace{1, \dots, 1}_{l(\Delta_i)})
+
\sum_{i=r+1}^t (0,\dots,0,\underbrace{1, \dots, 1}_{l(\Delta_i)})
\\&= 
\lambda_{\mm_\maxi} + \lambda_{\mm^\rest}.
\end{align*}
This completes the proof.
\end{proof}
Since $\mm_\maxi$ is a ladder multisegment 
(i.e., the multisegment corresponding to a ladder representation),
we can compute $\lambda_{\mm_\maxi}$ 
as in Example \ref{ex_lambda} (2).
Hence, 
using this corollary, we can compute $\lambda_\mm$ inductively. 

\subsection{An example of computation of $\mm^\ram$}\label{sec:ZD}
By using Proposition \ref{lem:fund3}, 
one can compute $\mm^\ram$ for an arbitrary multisegment $\mm$
in a systematic way.   Let us give an example.
\par

Let 
$\mm=\sum_{i=1}^7 \Delta_i$
be a multisegment, where 
$\Delta_1=[5,6]_\chi$,
$\Delta_2=[3,7]_\chi$,
$\Delta_3=[3,4]_\chi$,
$\Delta_4=[2,5]_\chi$,
$\Delta_5=[3,3]_\chi$,
$\Delta_6=[1,2]_\chi$,
$\Delta_7=[0,0]_\chi$.
Then $\mm_\maxi=\Delta_2+\Delta_4+\Delta_6+\Delta_7$
and 
$\mm^\rest=\Delta_1+\Delta_3+\Delta_5$.
We also have 
$(\mm^\rest)_\maxi=\Delta_1+\Delta_3$
and 
$(\mm^\rest)^\rest=\Delta_5$.
By Proposition \ref{lem:fund3}, 
we are reduced to computing ``$\ram$'' of the three ladder
multisegments.

As explained in Section 3 of \cite{LMi},
the Zelevinsky dual of a ladder multisegment can be calculated fairly easily.
Let us compute the Zelevinsky of $\mm_\maxi$ by drawing pictures.
In the $xy$-plane, we draw each segment of $\mm_\maxi$
so that each lies on the line $y=i$ for $i=1,\dots, 4$.
(See the following figure.)
Whenever there exist points $(e,f)$ and $(e+1, f-1)$ with $e,f \in \Z$,
we draw a dotted line connecting them.   
Then the dotted lines form the multisegment of the Zelevinsky dual $(\mm_\maxi)^\sharp$.
One can use the algorithm of M{\oe}glin--Waldspurger \cite{MW} 
to verify that the procedure above actually
gives the Zelevinsky dual.
We obtain 
$(\mm_\maxi)^\sharp
=\Delta'_1+\Delta'_2+\Delta'_3+\Delta'_4+\Delta'_5$,
where 
$\Delta'_1=[7,7]_\chi$,
$\Delta'_2=[5,6]_\chi$,
$\Delta'_3=[4,5]_\chi$,
$\Delta'_4=[2,4]_\chi$,
$\Delta'_5=[0,3]_\chi$.
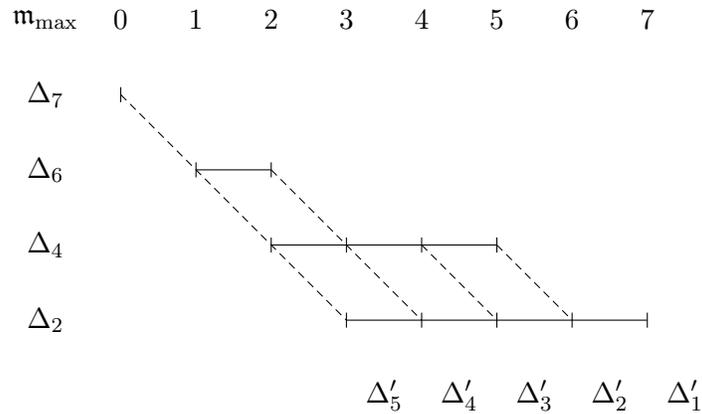
\begin{figure}[H]
\centering
\begin{tikzpicture}
\draw (-1,1) node {$\mm_\maxi$};

\draw (0,0);
\draw (1,-1) -- (2,-1);
\draw (2,-2) -- (3,-2) -- (4,-2) -- (5,-2);
\draw (3,-3) -- (4,-3) -- (5,-3) -- (6,-3) -- (7,-3);

\draw (0, -.1) -- (0,.1);
\draw (1, -1.1) -- (1,-.9);
\draw (2, -1.1) -- (2,-.9);
\draw (2, -2.1) -- (2,-1.9);
\draw (3, -2.1) -- (3,-1.9);
\draw (4, -2.1) -- (4,-1.9);
\draw (5, -2.1) -- (5,-1.9);
\draw (3, -3.1) -- (3,-2.9);
\draw (4, -3.1) -- (4,-2.9);
\draw (5, -3.1) -- (5,-2.9);
\draw (6, -3.1) -- (6,-2.9);
\draw (7, -3.1) -- (7,-2.9);

\draw (-1,0) node {$\Delta_7$};
\draw (-1,-1) node {$\Delta_6$};
\draw (-1,-2) node {$\Delta_4$};
\draw (-1,-3) node {$\Delta_2$};

\draw (0,1) node {0};
\draw (1,1) node {1};
\draw (2,1) node {2};
\draw (3,1) node {3};
\draw (4,1) node {4};
\draw (5,1) node {5};
\draw (6,1) node {6};
\draw (7,1) node {7};

\draw[densely dashed] (0,0) -- (3,-3);
\draw[densely dashed] (2,-1) -- (4,-3);
\draw[densely dashed] (4,-2) -- (5, -3);
\draw[densely dashed] (5,-2) -- (6, -3);

\draw (3.5, -4) node {$\Delta'_5$};
\draw (4.5, -4) node {$\Delta'_4$};
\draw (5.5, -4) node {$\Delta'_3$};
\draw (6.5, -4) node {$\Delta'_2$};
\draw (7.5, -4) node {$\Delta'_1$};
\end{tikzpicture}
\end{figure}

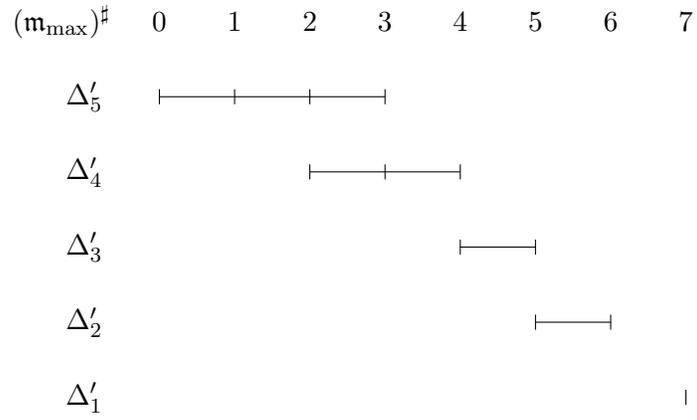
\begin{figure}[H]
\centering
\begin{tikzpicture}
\draw (-1.3,1) node {$(\mm_\maxi)^\sharp$};

\draw (0,1) node {0};
\draw (1,1) node {1};
\draw (2,1) node {2};
\draw (3,1) node {3};
\draw (4,1) node {4};
\draw (5,1) node {5};
\draw (6,1) node {6};
\draw (7,1) node {7};

\draw (0,0) -- (3,0);
\draw (2,-1) -- (4,-1);
\draw (4,-2) -- (5,-2);
\draw (5,-3) -- (6,-3); 
\draw (7,7);

\draw (-1,0) node {$\Delta'_5$};
\draw (-1,-1) node {$\Delta'_4$};
\draw (-1,-2) node {$\Delta'_3$};
\draw (-1,-3) node {$\Delta'_2$};
\draw (-1,-4) node {$\Delta'_1$};

\draw (0, -.1) -- (0,.1); 
\draw (1, -.1) -- (1,.1); 
\draw (2, -.1) -- (2,.1); 
\draw (3, -.1) -- (3,.1); 
\draw (2, -1.1) -- (2,-.9);
\draw (3, -1.1) -- (3,-.9);
\draw (4, -1.1) -- (4,-.9);
\draw (4, -2.1) -- (4,-1.9);
\draw (5, -2.1) -- (5,-1.9);
\draw (5, -3.1) -- (5,-2.9);
\draw (6, -3.1) -- (6,-2.9);
\draw (7, -4.1) -- (7,-3.9);
\end{tikzpicture}
\end{figure}

The multisegment of the highest derivative is obtained by shortening each segment by 1.   Hence, we have $((\mm_\maxi)^\sharp)^-$ as in the following figure.   
We obtain
$((\mm_\maxi)^\sharp)^-
=(\Delta'_1)^-+
(\Delta'_2)^-+
(\Delta'_3)^-+
(\Delta'_4)^-$
where 
$\Delta'_1=[5,5]_\chi$,
$\Delta'_2=[4,4]_\chi$,
$\Delta'_3=[2,3]_\chi$,
$\Delta'_4=[0,2]_\chi$.
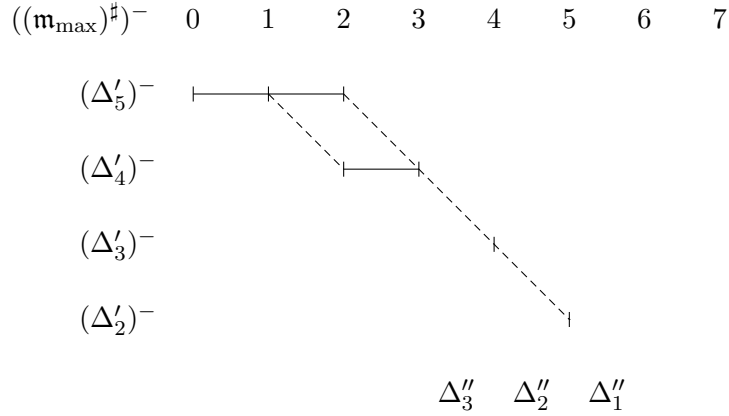
\begin{figure}[H]
\centering
\begin{tikzpicture}
\draw (-1.5,1) node {$((\mm_\maxi)^\sharp)^-$};

\draw (0,1) node {0};
\draw (1,1) node {1};
\draw (2,1) node {2};
\draw (3,1) node {3};
\draw (4,1) node {4};
\draw (5,1) node {5};
\draw (6,1) node {6};
\draw (7,1) node {7};

\draw (0,0) -- (2,0);
\draw (2,-1) -- (3,-1);
\draw (4,-2) -- (4,-2);
\draw (5,-3) -- (5,-3); 

\draw (-1,0) node {$(\Delta'_5)^-$};
\draw (-1,-1) node {$(\Delta'_4)^-$};
\draw (-1,-2) node {$(\Delta'_3)^-$};
\draw (-1,-3) node {$(\Delta'_2)^-$};

\draw (0, -.1) -- (0,.1); 
\draw (1, -.1) -- (1,.1); 
\draw (2, -.1) -- (2,.1); 
\draw (2, -1.1) -- (2,-.9);
\draw (3, -1.1) -- (3,-.9);
\draw (4, -2.1) -- (4,-1.9);
\draw (5, -3.1) -- (5,-2.9);

\draw[densely dashed] (1,0) -- (2,-1);
\draw[densely dashed] (2,0) -- (5,-3);

\draw (3.5, -4) node {$\Delta''_3$};
\draw (4.5, -4) node {$\Delta''_2$};
\draw (5.5, -4) node {$\Delta''_1$};

\end{tikzpicture}

Taking the Zelevinsky dual again, we arrive at $(\mm_\maxi)^\ram$ as in the following figure.
We obtain $(\mm_\maxi)^\ram=\Delta''_1+\Delta''_2+\Delta''_3$
where 
$\Delta''_1=[2,5]_\chi,
\Delta''_2=[1,2]_\chi,
\Delta''_3=[0,0]_\chi$.

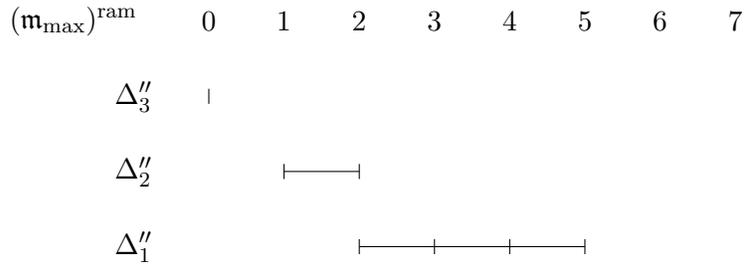
\begin{tikzpicture}
\draw (-1.8,1) node {$(\mm_\maxi)^\ram$};

\draw (0,1) node {0};
\draw (1,1) node {1};
\draw (2,1) node {2};
\draw (3,1) node {3};
\draw (4,1) node {4};
\draw (5,1) node {5};
\draw (6,1) node {6};
\draw (7,1) node {7};

\draw (1,-1) -- (2,-1);
\draw (2,-2) -- (5,-2);

\draw (0, -.1) -- (0,.1); 
\draw (1, -1.1) -- (1,-.9);
\draw (2, -1.1) -- (2,-.9);
\draw (2, -2.1) -- (2,-1.9);
\draw (3, -2.1) -- (3,-1.9);
\draw (4, -2.1) -- (4,-1.9);
\draw (5, -2.1) -- (5,-1.9);

\draw (-1,0) node {$\Delta''_3$};
\draw (-1,-1) node {$\Delta''_2$};
\draw (-1,-2) node {$\Delta''_1$};
\end{tikzpicture}
\end{figure}

Similarly, we have $(\Delta_1+\Delta_3)^\ram=[4,4]_\chi$
and $(\Delta_5)^\ram=\emptyset$.
Thus
$\mm^\ram=[4,4]_\chi+[2,5]_\chi+[1,2]_\chi+[0,0]_\chi$.

\subsection{The Weil--Deligne representations}
In this subsection, we give some justification for the use of the term ``ram'' in the notation $\pi^\ram$.    
This comes from the Galois side of the local Langlands correspondence (\cite{HT}, \cite{H}).
For several materials in this subsection, 
see \cite{Tate}.
\par

Let us fix an algebraic closure $\Fbar$ of $F$.
Let $W_F \subset \Gal(\Fbar/F)$ denote the Weil group of $F$.
By definition, $W_F$ is a locally profinite topological group.
If we denote by $W_F^\ab$ the quotient of $W_F$ by the closure
of $[W_F,W_F]$, then there exists an isomorphism 
$r_F \colon W_F^\ab \xto{\cong} F^\times$ that sends any lift of
geometric Frobenius to a uniformizer of $F$.
\par

A \emph{Weil--Deligne representation} is a triple $(\tau,V,N)$ where
$(\tau,V)$ is a finite dimensional complex representation of
$W_F$, and $N$ is a linear endomorphism of $V$ such that
the kernel of $\tau$ is open in $W_F$ and we have
$N \tau(\sigma)= |r_F(g)| \tau(\sigma) N$
for any $\sigma \in W_F$.
Let $I_F \subset W_F$ denote the inertia subgroup.
A Weil--Deligne representation $(\tau,V,N)$ is called \emph{unramified}
if $I_F$ acts trivially and $N$ acts as $0$ on $V$.
Any Weil--Deligne representation $V=(\tau,N,V)$ has
a unique maximal unramified Weil--Deligne subrepresentation $V_\ur$.
Explicitly, we have $V_\ur = V^{I_F} \cap \Ker\, N$.
We denote by $V^\ram$ the quotient $V/V_\ur$, 
and we call it the \emph{ramified quotient of $V$}.
\par

The local Langlands correspondence gives a one-to-one
correspondence between the isomorphism classes of
irreducible complex representations of $G_n$ and the isomorphism classes
of Frobenius semisimple $n$-dimensional Weil--Deligne representations
over the complex numbers.

\begin{lem} \label{lem:LLC}
Let $\pi$ be a unipotent irreducible admissible representation of $G_n$
and let $V$ denote the Weil--Deligne representation
corresponding to $\pi$ via the local Langlands correspondence.
Then $V^\ram$ corresponds to $\pi^\ram$.
\end{lem}
\begin{proof}
For a segment $[a,b]_\rho$, we denote by $\Delta[a,b]_\rho$ 
the generalized Steinberg representation, i.e., 
the unique irreducible quotient of 
\[
\rho|\cdot|^a \times \rho|\cdot|^{a+1} \times \dots \times \rho|\cdot|^b. 
\]
As in the Langlands classification, 
we write $\pi = L([a_1,b_1]_{\rho_1} + \dots + [a_r,b_r]_{\rho_r})$
if $\pi$ is the unique irreducible subrepresentation of 
\[
\Delta[a_1,b_1]_{\rho_1} \times \dots \times \Delta[a_r,b_r]_{\rho_r}
\]
with $\rho_i$ unitary and $a_1+b_1 \leq \dots \leq a_r+b_r$. 
Then the Zelevinsky dual $\pi^\sharp$ of $\pi$ is given by 
\[
\pi^\sharp = Z([a_1,b_1]_{\rho_1} + \dots + [a_r,b_r]_{\rho_r}).
\]
By \cite[8.1 Theorem]{Z}, 
the highest derivative $(\pi^\sharp)^-$ of $\pi^\sharp$ is 
\[
(\pi^\sharp)^- = Z([a_1,b_1-1]_{\rho_1} + \dots + [a_r,b_r-1]_{\rho_r}).
\]
Here, if $a_i = b_i$, we ignore $[a_i,b_i-1]_{\rho_i}$.
Hence 
\[
\pi^{\ram} = ((\pi^\sharp)^-)^\sharp
= L([a_1,b_1-1]_{\rho_1} + \dots + [a_r,b_r-1]_{\rho_r}).
\]
Therefore, the map $\pi \mapsto \pi^\ram$
corresponds to $V \mapsto V/\Ker\, N$ 
(see, e.g., \cite{Rodier}).
Since $\pi$ is unipotent, 
the corresponding $V$ satisfies that $V = V^{I_F}$ so that $V^\ram = V/\Ker\, N$.
\end{proof}

\section{Proofs of Propositions \ref{prop_lambda} and \ref{lem:fund3}}\label{sec.lambda}
The purpose of this section is to prove Propositions \ref{prop_lambda} and \ref{lem:fund3}. 
To do these, we introduce the notions of \emph{{\VNpair}s} and \emph{{\WLpair}s}.

\subsection{{\VNpair}s and {\WLpair}s}\label{sec:pairs}
A \emph{\VNpair} (over $\C$) is a pair $(V,N)$ of a finite dimensional
$\Z$-graded vector space $V$ over $\C$ and a $\C$-linear
endomorphism $N\colon V \to V$ of degree $1$.
Similarly, a \emph{\WLpair} (over $\C$) is a pair $(W,L)$ of a finite dimensional
$\Z$-graded vector space $W$ over $\C$ and a $\C$-linear
endomorphism $L \colon W \to W$ of degree $-1$.
\par

Let $(V,N)$ and $(V',N')$ be two {\VNpair}s. 
A morphism $f \colon (V,N) \rightarrow (V',N')$ is a $\C$-linear map $V \rightarrow V'$ preserving the degrees
such that $f \circ N = N' \circ f$.
\begin{lem}\label{lem:isom}
Let $(V,N)$ and $(V',N')$ be two {\VNpair}s. 
Then $(V,N) \cong (V',N')$ if and only if 
$V \cong V'$ as graded vector spaces, and 
$(\Image N, N|_{\Image N}) \cong (\Image N', N'|_{\Image N'})$. 
\end{lem}
\begin{proof}
The ``only if'' part is trivial. 
We prove the ``if'' part. 
Assume the two conditions. 
Let us choose an isomorphism 
\[
f_1 \colon
(\Image N,N|_{\Image N}) \xto{\cong} (\Image N',N'|_{\Image N'})
\]
of {\VNpair}s.
Let us also choose homogeneous elements $v_1,\ldots,v_r \in \Image N$ 
whose images in $\Image N/\Image N^2$ form a basis of this space.
For $i=1,\ldots,r$, 
let us choose homogeneous elements $e_1,\ldots,e_r \in V$
and $e'_1,\ldots,e'_r \in V'$ in such a way
that we have $N(e_i) = v_i$ and $N'(e'_i) = f_1 (v_i)$ for $i=1,\ldots,r$.
Let $W$ (\resp $W'$) denote the graded vector subspace of $V$ (\resp $V'$)
generated by $\Image N$ and $e_1,\ldots,e_r$
(\resp $\Image N'$ and $e'_1,\ldots,e'_r$).
\par

Let $\overline{N}\colon V/\Image N \to \Image N/\Image N^2$
denote the homomorphism induced by $N$.
It follows from the construction of $W$ that
the restriction of $\overline{N}$ to $W/\Image N$ gives
an isomorphism $W/\Image N \xto{\cong} \Image N/\Image N^2$.
Hence we have $V/\Image N = \Ker\, \overline{N} \oplus (W/\Image N)$.
By applying the snake lemma to the commutative diagram
\[
\begin{CD}
0 @>>> \Image N @>>> V @>>> V/\Image N @>>> 0 \\
@. @V{N}VV @V{N}VV @VV{\overline{N}}V @. \\
0 @>>> \Image N^2 @>>> \Image N @>>> \Image N/\Image N^2 @>>> 0,
\end{CD}
\]
we see that the homomorphism $\alpha \colon \Ker\, N \to \Ker\, \overline{N}$ 
induced by the quotient map $V \twoheadrightarrow V/\Image N$ is surjective.
Let us choose a graded vector subspace $U \subset \Ker\, N$ such that
the restriction of $\alpha$ to $U$ gives an isomorphism $U \xto{\cong} \Ker\, \overline{N}$.
Since $V/\Image N = \Ker\, \overline{N} \oplus (W/\Image N)$,
we have $V = U \oplus W$.
\par

A similar argument shows that there exists a
graded vector subspace $U' \subset \Ker\, N'$
such that $V' = U' \oplus W'$.
Since $V$ and $V'$ are isomorphic as graded
vector spaces, $U$ and $U'$ are isomorphic
as graded vector spaces.
Let us choose an isomorphism 
$f_2 \colon U \to U'$ of graded vector spaces.
\par

Let $f \colon V \to V'$ denote the homomorphism defined as follows:
$f(v)=f_1(v)$ for $v \in \Image N$, $f(e_i) = e'_i$ for $i=1,\ldots,r$, and $f(u)=f_2(u)$ for $u \in U$.
Then $f$ is an isomorphism of {\VNpair}s from $(V,N)$ to $(V',N')$. 
This completes the proof.
\end{proof}

Let $(V,N)$ be a {\VNpair}.
For an integer $c \in \Z$, we let $(V,N)(c)$ denote the $c$-th degree-shift of $(V,N)$. 
By definition, $(V,N)(c)=(V(c),N(c))$ 
where $V(c)$ is the $\Z$-graded vector space over $\C$ 
whose degree-$a$-part is equal to the degree-$(a-c)$-part of $V$ for any $a \in \Z$, 
and $N(c)\colon V(c) \to V(c)$ is the endomorphism induced by $N$.
(This notation of degree-shift corresponds to 
the Tate twist on the Galois side of the local Langlands correspondence.)
\par

For a segment $\Delta=[a,b]_\chi$ with $a,b \in \Z$, 
we let $(V_\Delta,N_\Delta)$ denote the {\VNpair} 
such that
$V_\Delta$ is the graded complex vector space with basis $e_a, e_{a+1},\ldots, e_b$ 
where for $i=a,\ldots,b$, 
the vector $e_i$ is homogeneous of degree $i$, 
and $N_\Delta \colon  V_\Delta \to V_\Delta$ is the endomorphism 
that sends $e_i$ to $e_{i+1}$ for $i=a,\ldots,b-1$ and sends $e_b$ to $0$.
Similarly, we denote by $(W_\Delta,L_\Delta)$ the
{\WLpair} such that
$W_\Delta = V_\Delta$ and
$L_\Delta \colon  W_\Delta \to W_\Delta$
is the endomorphism that sends $e_i$ to $e_{i-1}$
for $i=a+1,\ldots,b$ and sends $e_a$ to $0$.
\par

Let $\chi$ be an unramified character of $F^\times$.
For a multisegment $\mm=\Delta_1 + \cdots + \Delta_r$
of type $\chi$,
we define the {\VNpair} $(V_\mm,N_\mm)$ and the
{\WLpair} $(W_\mm,L_\mm)$ as the direct sums
\[
(V_\mm,N_\mm) =
\left( \bigoplus_{i=1}^r V_{\Delta_i}, \bigoplus_{i=1}^r N_{\Delta_i}\right),
\]
and
\[
(W_\mm, L_\mm) =
\left( \bigoplus_{i=1}^r W_{\Delta_i}, \bigoplus_{i=1}^r L_{\Delta_i}\right).
\]
\par

It follows from the Gabriel theory \cite{Ga}, or from the
theory of Jordan normal forms and some elementary argument 
(cf.~\cite{KZ}), that
these give one-to-one correspondences among the
multisegments of type $\chi$, the isomorphism classes of {\VNpair}s, 
and the isomorphism classes of {\WLpair}s.
\par

For a {\VNpair} $(V,N)$ (\resp a {\WLpair} $(W,L)$), 
let us consider the set $S(V,N)$ (\resp $S(W,L)$)
of $\C$-linear endomorphisms $L\colon  V \to V$ (\resp $N\colon  V \to V$)
of degree $-1$ (\resp degree $1$) satisfying $L\circ N=N \circ L$.
We sometimes regard $S(V,N)$ and
$S(W,L)$ as algebraic varieties over $\C$.
Since $S(V,N)$ and $S(W,L)$ are finite dimensional complex vector spaces,
$S(V,N)$ and $S(W,L)$ are, as algebraic varieties over $\C$,
isomorphic to affine spaces over $\C$.

\begin{lem} \label{lem:surj}
Let $(V,N)$ be a {\VNpair} and $(W,L)$ be a {\WLpair}. 
\begin{enumerate}
\item 
The map $S(V,N) \to S(\Image N, N|_{\Image N})$
that sends $L$ to $L|_{\Image N}$ is surjective.
\item 
The map $S(W,L) \to S(\Image L, L|_{\Image L})$
that sends $N$ to $N|_{\Image L}$ is surjective.
\end{enumerate}
\end{lem}
\begin{proof}
We only give a proof of the assertion (1).
We can prove the assertion (2) in a similar manner.
\par

Let us choose homogeneous, linearly independent elements $v_1, \ldots, v_m \in V$ 
such that $V$ is a direct sum of $\Image N$ 
and the subspace of $V$ generated by $v_1, \ldots, v_m$. 
For $i=1,\ldots,m$, we let $d_i$ denote the degree of $v_i$.
Given $L' \in S(\Image N, N|_{\Image N})$, 
choose a homogeneous element $w_i \in V$ of degree $d_i-1$ 
that satisfies $L'(N(v_i)) = N(w_i)$ for each $i=1, \ldots, m$.
Let $L$ denote the unique $\C$-linear map $V \to V$ 
such that $L(v) = L'(v)$ for $v \in \Image N$ and that $L(v_i) = w_i$ for $i=1,\ldots,m$.
It is then straightforward to check that $L \in S(V,N)$.
It follows from the construction of $L$ that $L|_{\Image N} = L'$.
Hence the claim follows.
\end{proof}

Let $(V,N)$ be a {\VNpair} and 
let $\mm$ be the multisegment (of type $\chi$) corresponding to $(V,N)$.
It follows from \cite{Z2} and \cite{MW0} that 
there exists a Zariski open dense subset $S^o(V,N) \subset S(V,N)$
such that, for $L \in S(V,N)$, 
the multisegment (of type $\chi$) corresponding to $(V,L)$
is equal to $\mm^\sharp$
if and only if $L \in S^o(V,N)$.
\par

Let $V$ be a finite dimensional $\Z$-graded vector space over $\C$,
and $N, L \colon  V \to V$ be $\C$-linear endomorphisms of degree $1$, $-1$, respectively.
We say that the triple $(V,N,L)$ is \emph{admissible} if
$N \circ L = L\circ N$ and 
the multisegment corresponding to the {\WLpair} $(V,L)$ is 
the Zelevinsky dual of the one corresponding to the {\VNpair} $(V,N)$.

\begin{lem} \label{lem:easy}
Let $(V,N)$ (\resp $(W,L)$) be a {\VNpair} (\resp a {\WLpair})
and let $\mm$ denote the multisegment corresponding to $(V,N)$ (\resp $(W,L)$).
Then the multisegment $\mm^{-}$ corresponds to
$(\Image N, N|_{\Image N})(-1)$ (\resp $(\Image L, L|_{\Image L})$).
\end{lem}
\begin{proof}
Easy.
\end{proof}

Let us give an example.
Let $\mm = \Delta_1 + \Delta_2 + \Delta_3 + \Delta_4$,
where $\Delta_1=[3,7]_\chi$, $\Delta_2=[2,5]_\chi$,
$\Delta_3= [1,2]_\chi$, and $\Delta_4=[0,0]_\chi$.

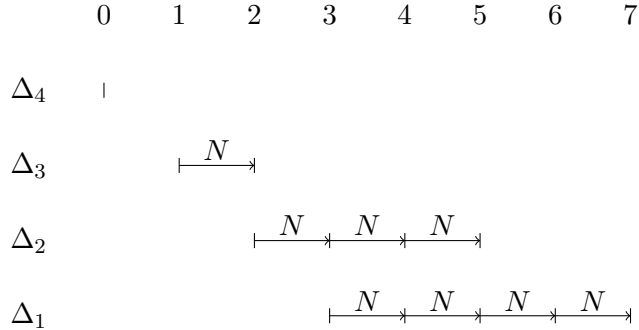
\begin{figure}[H]
\centering
\begin{tikzpicture}

\draw (0,1) node {0};
\draw (1,1) node {1};
\draw (2,1) node {2};
\draw (3,1) node {3};
\draw (4,1) node {4};
\draw (5,1) node {5};
\draw (6,1) node {6};
\draw (7,1) node {7};

\draw[->] (1,-1) -- (2,-1);
\draw[->] (2,-2) -- (3,-2);
\draw[->] (3,-2) -- (4,-2);
\draw[->] (4,-2) -- (5,-2);
\draw[->] (3,-3) -- (4,-3);
\draw[->] (4,-3) -- (5,-3);
\draw[->] (5,-3) -- (6,-3);
\draw[->] (6,-3) -- (7,-3);

\draw (1.5, -.8) node {$N$};
\draw (2.5, -1.8) node {$N$};
\draw (3.5, -1.8) node {$N$};
\draw (4.5, -1.8) node {$N$};
\draw (3.5, -2.8) node {$N$};
\draw (4.5, -2.8) node {$N$};
\draw (5.5, -2.8) node {$N$};
\draw (6.5, -2.8) node {$N$};

\draw (0, -.1) -- (0,.1);
\draw (1, -1.1) -- (1,-.9);
\draw (2, -1.1) -- (2,-.9);
\draw (2, -2.1) -- (2,-1.9);
\draw (3, -2.1) -- (3,-1.9);
\draw (4, -2.1) -- (4,-1.9);
\draw (5, -2.1) -- (5,-1.9);
\draw (3, -3.1) -- (3,-2.9);
\draw (4, -3.1) -- (4,-2.9);
\draw (5, -3.1) -- (5,-2.9);
\draw (6, -3.1) -- (6,-2.9);
\draw (7, -3.1) -- (7,-2.9);

\draw (-1,0) node {$\Delta_4$};
\draw (-1,-1) node {$\Delta_3$};
\draw (-1,-2) node {$\Delta_2$};
\draw (-1,-3) node {$\Delta_1$};
\end{tikzpicture}
\end{figure}
The picture of $(\Image\,N(-1), N|_{\Image\, N}(-1))$ is as follows:
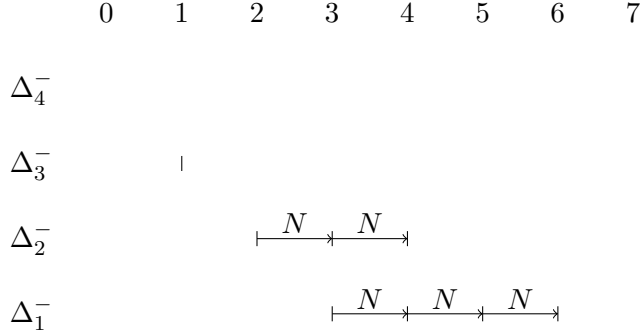
\begin{figure}[H]
\centering
\begin{tikzpicture}
\draw (0,1) node {0};
\draw (1,1) node {1};
\draw (2,1) node {2};
\draw (3,1) node {3};
\draw (4,1) node {4};
\draw (5,1) node {5};
\draw (6,1) node {6};
\draw (7,1) node {7};

\draw[->] (2,-2) -- (3,-2);
\draw[->] (3,-2) -- (4,-2);
\draw[->] (3,-3) -- (4,-3);
\draw[->] (4,-3) -- (5,-3);
\draw[->] (5,-3) -- (6,-3);

\draw (2.5, -1.8) node {$N$};
\draw (3.5, -1.8) node {$N$};
\draw (3.5, -2.8) node {$N$};
\draw (4.5, -2.8) node {$N$};
\draw (5.5, -2.8) node {$N$};

\draw (1, -1.1) -- (1,-.9);
\draw (2, -2.1) -- (2,-1.9);
\draw (3, -2.1) -- (3,-1.9);
\draw (4, -2.1) -- (4,-1.9);
\draw (3, -3.1) -- (3,-2.9);
\draw (4, -3.1) -- (4,-2.9);
\draw (5, -3.1) -- (5,-2.9);
\draw (6, -3.1) -- (6,-2.9);

\draw (-1,0) node {$\Delta_{4}^-$};
\draw (-1,-1) node {$\Delta_{3}^-$};
\draw (-1,-2) node {$\Delta_{2}^-$};
\draw (-1,-3) node {$\Delta_{1}^-$};

\end{tikzpicture}
\end{figure}

We see that this corresponds to $\mm^-$. 

\begin{lem} \label{lem:keypair}
Let $(V,N)$ be a {\VNpair} 
and let $\mm$ be the multisegment corresponding to $(V,N)$.
Then there exists a Zariski open dense subset 
$S^\theta(V,N) \subset S(V,N)$
such that for $L \in S(V,N)$, 
both $(V,N,L)$ and $(\Image L, N|_{\Image L}, L|_{\Image L})$ are admissible triples
if and only if $L \in S^\theta(V,N)$.
\end{lem}
\begin{proof}
It is easy to see that there exists a Zariski open subset 
$S^\theta(V,N) \subset S(V,N)$ such that
for $L \in S(V,N)$, both $(V,N,L)$ and 
$(\Image L, N|_{\Image L}, L|_{\Image L})$
are admissible triples
if and only if $L \in S^\theta(V,N)$.
\par

It remains to show that $S^\theta(V,N)$ is dense in $S(V,N)$.
Since $S(V,N)$ is irreducible as an algebraic variety over $\C$, 
it suffices to show that $S^\theta(V,N)$ is non-empty.
Let us choose $L \in S^o(V,N)$.
Since the morphism $S(V,L) \to S(\Image L, L|_{\Image L})$ is surjective by Lemma \ref{lem:surj}, 
there exists $N' \in S(V,L)$ such that 
both $(V,N',L)$ and$(\Image L, N'|_{\Image L}, L|_{\Image L})$ are admissible triples.
Then $(V,N)$ and $(V,N')$ are isomorphic 
since both correspond to the same multisegment.
Hence $(V,N',L)$ is isomorphic to $(V,N,L')$ for some $L' \in S(V,N)$.
Since $L'$ belongs to $S^\theta(V,N)$,
it follows that $S^\theta(V,N)$ is non-empty, as desired.
\end{proof}

\subsection{Proof of Proposition \ref{prop_lambda}}\label{sec.proof_lambda}
Now we prove Proposition \ref{prop_lambda}.

\begin{proof}[Proof of Proposition \ref{prop_lambda}]
Let $\pi = Z(\mm)$ be an irreducible representation of $G_n$.
We decompose $\mm$ as
\[
\mm = \mm' + \mm_1+\dots+\mm_{t}, 
\]
where 
\begin{itemize}
\item
each segment in $\mm'$ is not unipotent; 

\item
each $\mm_i$ is of type $\chi_i$
for some unramified character $\chi_i$ of $F^\times$ for $1 \leq i \leq t$; 

\item
if $i \not= j$, then $\chi_i\chi_j^{-1}$ is not of the form $|\cdot|^a$ for any $a \in \Z$. 
\end{itemize}
Set $\pi' = Z(\mm')$ and $\pi_i = Z(\mm_i)$. 
Then $\pi$ is isomorphic to the parabolic induction
$\pi' \times \pi_1 \times \cdots \times \pi_{t}$.
\par

For $\Pi = \pi, \pi', \pi_1, \ldots, \pi_{t}$, 
let $\Pi^{(0)} = \Pi$ and 
$\Pi^{(i)}$ denote the highest derivative of $\Pi^{(i-1)}$ for $i \geq 1$.
Then we have $\pi^{(i)} = \pi'^{(i)} \times \pi_1^{(i)} \times \cdots \times \pi_{t}^{(i)}$ 
for any integer $i \ge 0$.
Thus, to prove the claim, 
we may assume that $\mm = \mm'$ or $\mm = \mm_1$.
\par

First, we consider the case where $\mm = \mm'$. 
Let us write $\pi = Z(\mm)$ and
$\mm = [a_1,b_1]_{\rho_1} + \cdots +[a_r,b_r]_{\rho_r}$.
Then $\rho_1, \ldots, \rho_r$ are ramified cuspidal representations. 
For $i=1,\ldots,r$, let $c_i = c_{\rho_i}$ denote the conductor of $\rho_i$.
Then for $j \geq 0$, we have $\pi^{(j)} = Z(\mm^{(j)})$
where
\[
\mm^{(j)} = \sum_{1 \le i \le r \atop
b_i - a_i \ge j} [a_i, b_i -j]_{\rho_i}.
\]
This shows that the conductor of $\pi^{(j)}$ is equal to
\[
c^{(j)} = \sum_{1 \le i \le r \atop
b_i - a_i \ge j} (b_i -a_i + 1 -j) c_{i}.
\]
Hence we have
\[
c^{(j)} - c^{(j+1)}
= \sum_{1 \le i \le r \atop
b_i - a_i \ge j} c_{i}.
\]
From this, one can easily see that
\[
\lambda_\pi
= \sum_{i=1}^r
(0,\ldots,0, \underbrace{c_i,\ldots,c_i}_{b_i-a_i+1})
= \lambda_\mm, 
\]
as desired.
\par

Now we consider the case where $\pi = Z(\mm)$ is of type $\chi$ 
for an unramified character $\chi$ of $F^\times$.
Let us consider the {\VNpair} $(V,N)$ corresponding to $\mm$.
For $i \ge 0$, let us write $\pi^{(i)} = Z(\mm^{(i)})$.
As we have remarked at the beginning of Section \ref{sec.def_lambda},
we have $\pi^{(1)} = Z(\mm^{-})$.
Hence $\mm^{(i)}$ is obtained from $\mm$ 
by the $i$-fold iteration of the operation $(\ )^{-}$.
Therefore, it follows from Lemma \ref{lem:easy} that
$\mm^{(i)}$ corresponds to the {\VNpair} $(\Image N^i, N|_{\Image N^i})(-i)$.
Let us choose $L \in S^\theta(V,N)$ such that
$L|_{\Image N^i}$ belongs to $S^o(\Image N^i, N|_{\Image N^i})$ for any integer $i \ge 0$.
By Lemma \ref{lem:surj}, such an $L$ exists.
Then the conductor of $\pi^{(i)}$ is equal to the dimension of $\Image L \circ N^i$.
Hence if we write $\lambda_\pi = (\lambda_1, \dots, \lambda_n)$ 
and $d_i = \dim \Image L \circ N^i$ for $i \geq 0$, 
then we have
\[
\lambda_k = d_{n-k} - d_{n-k+1}
\]
for $k = 1, \dots, n$. 
Let us write $\pi^\ram = Z(\mm^\ram)$
with $\mm^\ram = \Delta_1+\dots+\Delta_r$ and $\Delta_i = [a_i,b_i]_\chi$ for $1 \leq i \leq r$. 
Then $\lambda_\mm = (\lambda'_1,\dots,\lambda'_n)$ with 
\[
\lambda'_k = \sum_{\substack{1 \leq i \leq r \\ b_i-a_i \geq n-k}}1
\]
for $k = 1,\dots,n$. 
By Lemmas \ref{lem:easy} and \ref{lem:keypair}, 
$\mm^\ram$ corresponds to the {\VNpair} $(\Image L, N|_{\Image L})$.
Since $L$ and $N$ commute, we have
$\dim \Image N^i \circ L = d_i$ for $i \geq 0$.
Hence we have
\[
d_i-d_{i+1} = \sum_{\substack{1 \leq i \leq r \\ b_i-a_i \geq i}}1
\]
for $i = 0,\dots,n-1$.
Therefore, we have
\[
\lambda'_k = d_{n-k} - d_{n-k+1} = \lambda_k
\]
for $k = 1,\dots,n$. 
This completes the proof.
\end{proof}

We do not use the following proposition, but it might be interesting.
\begin{prop} \label{lem:fund1}
For any multisegment $\mm$, we have
$(\mm^{-})^\ram = (\mm^\ram)^{-}$.
\end{prop}
\begin{proof}
Let $(V,N)$ be the {\VNpair} corresponding 
to the multisegment $\mm$.
If we choose a sufficiently general $L \in S(V,N)$, then
$(\mm^{-})^\ram$ and $(\mm^\ram)^{-}$ correspond to 
the pairs
$(\Image L\circ N, N|_{\Image L \circ N})(-1)$ and
$(\Image N\circ L, N|_{\Image N \circ L})(-1)$,
respectively.
Since $L\circ N=N \circ L$, the claim follows.
\end{proof}

\subsection{Proof of Proposition \ref{lem:fund3}}\label{sec.proof_max}
The following statement is easy to check.
However, we record it as a lemma for later use.
A proof is omitted.
\begin{lem} \label{lem:fund2}
For any multisegment $\mm$, 
we have $(\mm^{-})_\maxi = (\mm_\maxi)^{-}$ and $(\mm^{-})^\rest = (\mm^\rest)^{-}$.
\end{lem}

For a multisegment $\mm$, 
a \emph{full-sub-multisegment} of $\mm$ 
is a multisegment $\mm'$ 
such that
for any segment $\Delta$ in $\mm'$,
its multiplicity in $\mm'$ is equal to that in $\mm$.
\par

We say that a multisegment $\mm$ is \emph{totally ordered} 
if for any two segments $\Delta, \Delta'$ in $\mm$, 
we have either $\Delta \subset \Delta'$ or $\Delta' \subset \Delta$.

\begin{prop} \label{prop:KZ}
Let $\mm = \Delta_1 + \cdots + \Delta_r$ be a multisegment of type $\chi$
and $a \in \Z$ an integer. 
Let us write $\delta_a =[a,a+1]_\chi$.
Let $\mm_a$ denote the full-sub-multisegment of $\mm$ 
that consists of segments which intersect $\delta_a$,
and
let $\mm^\sharp_{(a)}$ denote the full-sub-multisegment
of $\mm^\sharp = \Delta'_1 + \cdots + \Delta'_s$
that consists of segments which contain $\delta_a$.
Namely, 
\[
\mm_a \coloneqq \sum_{1 \le i \le r \atop \Delta_i \cap \delta_a \neq \emptyset}
\Delta_i, 
\quad
\mm^\sharp_{(a)} \coloneqq
\sum_{1 \le i \le s \atop \Delta'_i \supset \delta_a}
\Delta'_i.
\]
Then we have the equality
\[
\Card(\mm^\sharp_{(a)})
= \Card(\mm_a) - \max_{\mm'} \Card(\mm')
\]
where $\mm'$ runs over the set of totally ordered
full-sub-multisegments of $\mm_a$.
\end{prop}
\begin{proof}
By replacing $\chi$ with $\chi|\cdot|^c$ for some integer $c$,
we may and will assume that there exists an integer $r$
such that any segment in
$\mm$ is contained in $[1,r]_\chi$.
\par

For two integers $a,b$ with $1 \leq a \leq b \leq r$, 
let $d_{a,b}=d_{a,b}(\mm)$ denote the multiplicity of the segment $[a,b]_\chi$ in $\mm$.
When $a>b$, we set $d_{a,b}=0$. 
Then it follows from the result of Knight--Zelevinsky \cite[Theorem 1.2]{KZ}
that $\Card(\mm^\sharp_{(a)})$
is equal to the right-hand side of the equality (1.6) in \cite{KZ}
for $(i,j)=(a,a+1)$.
\par

For two integers $x,y$ with $x \leq y$, 
let $[x,y]$ denote the set of integers $c$ satisfying $x \le c \le y$.
Let $a \in [1,r-1]$.
We rewrite the right-hand side of (1.6) in \cite{KZ} for $(i,j)=(a,a+1)$.
Let us first recall some notation in \cite{KZ}.
They fix a positive integer $r$ 
and consider the set $S$ of pairs of integers $(i,j)$ such that $1\leq i \leq j \leq r$.
For $ 1 \leq i \leq j \leq r$,
they consider the set $T_{i,j}$ of functions $\nu\colon  [1,i] \times [j,r] \to [i,j]$
such that $\nu(k,l) \leq \nu(k', l')$ whenever $k\leq k', l \leq l'$.
\par

Let $1 \leq a \leq r$.  
We only use the case $i=a, j=a+1$ and consider $T_{a, a+1}$.
In this case any function $\nu \in T_{a, a+1}$ takes one of two values $a, a+1$.
We express this using Figure \ref{fig1} below.
The rectangle depicts the set $[1,a] \times [a+1, r]$.
The left upper corner is $(1,a+1)$,
the left lower corner is $(a,a+1)$,
the right upper corner is $(1,r)$, and 
the right lower corner is $(a, r)$.
Because of the condition on $\nu$,
there exists a bold line as in the picture such that $\nu$ takes 
the value $a$ on the left (call the region $L$)
and the value $a+1$ on the right (call the region $R$).
\par

We look at the sum from $(1.6)$ loc.~cit.:
\[
\sum_{(k,l) \in [1,a]\times [a+1, r]}
d_{\nu(k,l)+k-a,
\nu(k,l)+l-a-1}.
\]
This equals
\[
\sum_{(k,l) \in L}
d_{k,l-1}
+
\sum_{(k,l) \in R}
d_{k+1, l}.
\]
Now consider Figure \ref{fig2} below.
The rectangle depicts the set $U=[1,a+1] \times [a, r]$.
Let $L'$ be the region $L$ moved to the left by 1
and $R'$ be the region $R$ moved down by 1.   
These are subsets of $U$, and the complement 
$V_\nu=U\setminus (L' \cup R')$ is shown in blue in the picture. 
\par

A {\path} from $(a+1,a)$ to $(1,r)$ is a map $p\colon  [0,r] \to \Z \times \Z$
satisfying the following conditions:
\begin{enumerate}
\item 
$p(0)=(a+1,a)$.
\item 
For $i=1, \ldots, r$, the element $p(i) \in \Z \times \Z$
is equal to $p(i-1)-(1,0)$ or $p(i-1)+(0,1)$.
\item 
$p(r)=(1,r)$.
\end{enumerate}
Then $V_\nu$ is equal to the image of a {\path} from 
$[a+1, a]$ to $[1,r]$.
By sending $\nu$ to this path, we obtain a
bijection from $T_{a, a+1}$ to the set $A_a$ of
{\path}s from $(a+1,a)$ to $(1,r)$.
\par

Notice now that the sum above is equal to
\[
\sum_{(k,l) \in L'}
d_{k,l}
+\sum_{(k,l) \in R'}
d_{k,l}
=
\sum_{(k,l) \in U}
d_{k,l}
-
\sum_{(k,l) \in V_\nu} d_{k,l}.
\]
Conversely, given a path from
$[a+1, a]$ to $[a,1]$,
we obtain a function $\nu \in T_{a, a+1}$
such that $V_\nu$ is the image of the given {\path}.
Thus, the right-hand side of 
(1.6) of \cite{KZ} is equal to 
\[
\sum_{(k,l) \in U}
d_{k,l}
-\max_{p \in A_a}
\sum_{i=0}^{r}
d_{p(i)}.
\]

\begin{figure}[H]
\centering
\begin{tikzpicture}
\draw[step=1cm,gray,very thin] (0,0) grid (9,6);

\draw[very thick] (0,0) -- (2,0) -- (2,1) -- (3,1) -- (3,3) -- (5,3) -- (5,5) 
-- (6,5) -- (6,6) -- (9,6);

\draw (1.7,4.2) node {\Large{$L$}};
\draw (6.8, 1.4) node {\Large{$R$}};

\end{tikzpicture}
\caption{}
\label{fig1}
\end{figure}

\begin{figure}[H]
\centering
\begin{tikzpicture}
\draw[step=1cm,gray,very thin] (-1,-1) grid (9,6);

\fill[blue!40!white] (-1,-1) rectangle (2,0);
\fill[blue!40!white] (1,0) rectangle (3,1);
\fill[blue!40!white] (2,1) rectangle (3,2);
\fill[blue!40!white] (2,2) rectangle (5,3);
\fill[blue!40!white] (4,3) rectangle (5,5);
\fill[blue!40!white] (5,4) rectangle (6,6);
\fill[blue!40!white] (5,5) rectangle (9,6);

\draw[step=1cm,gray,very thin] (-1,-1) grid (9,6);
\draw (-1,0) -- (1,0) -- (1,1) -- (2,1) -- (2,3) -- (4,3) -- (4,5) 
-- (5,5) -- (5,6) -- (8,6);

\draw (0,-1) -- (2,-1) -- (2,0) -- (3,0) -- (3,2) -- (5,2) -- (5,4) 
-- (6,4) -- (6,5) -- (9,5);

\draw (1.7, 4.2) node {\Large{$L'$}};
\draw (6.8, 1.4) node {\Large{$R'$}};
\draw (4.5, 3.5) node {\Large{$V_\nu$}};
\draw (-1.7, 5.5) node {\Large{U}};
\end{tikzpicture}
\caption{}
\label{fig2}
\end{figure}

Notice that
$\Card(\mm_a) = \sum_{(k,l) \in U}d_{k,l}$.
From this we see that $\Card(\mm^\sharp_{(a)})$
is equal to
\[
\Card(\mm_a) - E_a(\mm)
\]
where
\[
E_a(\mm) = \max_{p \in A_a} \sum_{i=0}^{r} d_{p(i)}.
\]
\par

For $p \in A_a$, let $\mm_{a,p}$ denote the full-sub-multisegment of
$\mm_a$ that consists of the segments $[a',b']_\chi$ in $\mm_a$
of the form $(a',b') = p(i)$ for some integer $i \in [0,r]$.
Then $\mm_{a,p}$ is totally ordered and we have
$\sum_{i=0}^{r} d_{p(i)} = \Card(\mm_{a,p})$.
By sending $p$ to $\mm_{a,p}$, we obtain a map from
$A_a$ to the set $T_a$ of totally
ordered full-sub-multisegments of $\mm_a$.
In general, this map is neither injective nor surjective.
However, for any totally ordered
full-sub-multisegment $\mm'$ of $\mm_a$, there
exists a {\path} $p \in A_a$ such that
$\mm'$ is a full-sub-multisegment of $\mm_{a,p}$.
In particular $\Card(\mm') \le \Card(\mm_{a,p})$ for this $p$.
\par

Thus, we obtain an equality
\[
\max_{\mm' \in T_a} \Card(\mm')
= \max_{p \in A_a} \Card(\mm_{a,p})
= E_a(\mm),
\]
which completes the proof.
\end{proof}

Now we can prove Proposition \ref{lem:fund3}.
\begin{proof}[Proof of Proposition \ref{lem:fund3}]
We prove the claim by induction on $l(\mm)$.
\par

Let $(V,N)$, $(V_1,N_1)$ and $(V_2,N_2)$ be the {\VNpair}s 
corresponding to the multisegments $\mm$, $\mm_\maxi$ and $\mm^\rest$, respectively.
Let us consider the sets $S(V,N)$, $S(V_1,N_1)$ and $S(V_2,N_2)$ 
introduced in Section \ref{sec:pairs}.
Let us choose sufficiently general $L \in S(V,N)$, $L_1 \in S(V_1,N_1)$, and $L_2 \in S(V_2,N_2)$.
Then the six multisegments 
$\mm^\ram$, 
$(\mm^{-})^\ram$, 
$(\mm_\maxi)^\ram$, 
$((\mm_\maxi)^{-})^\ram$, 
$(\mm^\rest)^\ram$,
and $((\mm^\rest)^{-})^\ram$
correspond to the pairs
$(\Image L, N|_{\Image L})$,
$(\Image L\circ N, N|_{\Image L \circ N})(-1)$,
$(\Image L_1, N_1|_{\Image L_1})$,
$(\Image L_1\circ N_1, N_1|_{\Image L_1 \circ N_1})(-1)$,
$(\Image L_2, N_2|_{\Image L_2})$, 
and
$(\Image L_2\circ N_2, N_2|_{\Image L_2 \circ N_2})(-1)$,
respectively.
\par

To prove the claim, it suffices to show that
the pair $(\Image L, N|_{\Image L})$ is isomorphic to
the pair $(\Image L_1 \oplus \Image L_2, N_1|_{\Image L_1}
\oplus N_2|_{\Image L_2})$.
By the inductive hypothesis, the claim is true for
the multisegment $\mm^{-}$.
Hence it follows from Lemma \ref{lem:fund2} that
\[
(\Image L\circ N, N|_{\Image L\circ N})
\cong 
(\Image L_1\circ N_1 \oplus \Image L_2 \circ N_2, 
N_1|_{\Image L_1 \circ N_1} \oplus N_2|_{\Image L_2 \circ N_2}).
\]
Hence by Lemma \ref{lem:isom}, it suffices to show that
the graded vector space $\Image L$ is isomorphic to
the graded vector space $\Image L_1 \oplus \Image L_2$.
\par

Let $\mm_a$ and $\mm^{\sharp}_{(a)}$ be as in Proposition \ref{prop:KZ}. 
Note that the dimension of the degree-$a$-part of $\Image L$ 
is equal to $\Card(\mm^\sharp_{(a)})$.
Let $\mm'$ be a totally ordered full-sub-multisegment of $\mm_a$ with the maximum cardinality. 
When $\mm_a$ is non-empty, 
the maximal segment $\Delta_1$ of $\mm'$ must belong to $\mm_\maxi$, 
since otherwise one can find a totally ordered full-sub-multisegment of $\mm_a$ 
that strictly contains $\mm'$ 
by adding to $\mm'$ a segment of $\mm_\maxi$ that contains $\Delta_1$,
which is a contradiction.
It is then easy to see that
\begin{itemize}
\item
when $\mm_a$ is non-empty, 
$\mm' - \Delta_1$ is a totally ordered full-sub-multisegment of $(\mm^\rest)_a$
with the maximum cardinality; and
\item
$\Delta_1$, which is regarded as a multisegment with $\Card(\Delta_1) = 1$,
is a totally ordered full-sub-multisegment of $(\mm_\maxi)_a$
with the maximum cardinality.
\end{itemize} 
Thus, it follows from Proposition \ref{prop:KZ}
that the dimension of the degree-$a$-part of $\Image L$ 
is equal to the sum of those of $\Image L_1$ and $\Image L_2$,
as desired.
\end{proof}

\section{Preliminaries on $\oo$-modules}\label{sec.modules}
To prove our main theorems, we prepare some results on $\oo$-modules in this section. 

\subsection{On $\oo$-modules of finite length}\label{sec.o-modules}
In this subsection, we introduce some terminologies on $\oo$-modules 
and give two basic results 
(Propositions \ref{prop:convexity}, \ref{prop:uniqueness}), 
which we call convexity and uniqueness, respectively.
The authors suspect that these two results are well-known
to some experts. In fact, one can deduce
them from the description of Hall polynomials given in
\cite[II, (4.3)]{M} 
in terms of sequences of partitions related with the Littlewood--Richardson rule.
However, for the sake of completeness, we do not omit 
the proof of these results, which the authors believe
to be helpful for most of the readers.
\par

Let $|\CC|$ denote the set of isomorphism classes of $\oo$-modules of finite length.
For an $\oo$-module $M$ of finite length,
we denote by $[M] \in |\CC|$ its isomorphism class.
\par

For an integer $n \geq 1$, 
let $|\CC^n| \subset |\CC|$ denote the subset of isomorphism classes $[M]$
such that $M$ is generated by at most $n$ elements.
We denote by $\iota_n \colon |\CC^n| \hookrightarrow |\CC^{n+1}|$ the inclusion map.
\par

Recall that $\Lambda_n$ is the set of $n$-tuples $(\lambda_1,\ldots,\lambda_n)$
of integers satisfying $0 \leq \lambda_1 \leq \cdots \leq \lambda_n$.
For $[M] \in |\CC^n|$, there exists a unique
element $(\lambda_1,\ldots,\lambda_n)$ of $\Lambda_n$ such that
the $\oo$-module $M$ is isomorphic to
\[
\oo/\pp^{\lambda_1} \oplus \cdots  \oplus \oo/\pp^{\lambda_n}.
\]
By sending $[M]$ to the $n$-tuple $(\lambda_1,\ldots,\lambda_n)$,
we obtain a bijective map $\seq_n \colon |\CC^n| \to \Lambda_n$.
We denote by $\jmath_n \colon \Lambda_n \to \Lambda_{n+1}$ the injective map
that sends $(\lambda_1,\ldots, \lambda_n)$ to $(0,\lambda_1,\ldots,\lambda_n)$.
Then the diagram
\[
\begin{CD}
|\CC^n| @>{\seq_n}>> \Lambda_n \\
@V{\iota_n}VV @VV{\jmath_n}V \\
|\CC^{n+1}| @>{\seq_{n+1}}>> \Lambda_{n+1}
\end{CD}
\]
is commutative.
\par

For two elements $[M], [M'] \in |\CC^n|$,
we write $[M] \leq [M']$ if $\seq_n([M]) \leq \seq_n([M'])$
with respect to the lexicographic order on $\Lambda_n$.
This gives a total order on the set $|\CC^n|$.
The map $\iota_n$ is compatible with the total orders on $|\CC^n|$ and on $|\CC^{n+1}|$
since the map $\jmath_n$ is compatible with
the lexicographic orders on $\Lambda_n$ and on $\Lambda_{n+1}$.
Hence the total orders on $|\CC^n|$ for all $n$
induce a total order $\leq$ on the set $|\CC|$.
\par

We regard $\Lambda_n$ as a subset of $\Z^n$.
Then $\Lambda_n$ is closed under the addition $+$ on $\Z^{n}$ 
and becomes a commutative submonoid of $\Z^n$ with the addition $+$.
For two elements $[M], [M'] \in |\CC^n|$, 
we denote by $[M] \vee [M']$ the unique element of $|\CC^n|$ 
whose image under $\seq_n$ is equal to $\seq_n([M]) + \seq_n([M'])$.
Then the set $|\CC^n|$ becomes a commutative monoid with the operation $\vee$ 
and the diagram
\[
\begin{CD}
|\CC^n| \times |\CC^n| @>{\vee}>> |\CC^n| \\
@V{\seq_n \times \seq_n}VV @VV{\seq_n}V \\
\Lambda_n \times \Lambda_n @>{+}>> \Lambda_n
\end{CD}
\]
is commutative.
The map $\iota_n$ is compatible with the monoid structures on $|\CC^n|$ and $|\CC^{n+1}|$
since the map $\jmath_n$ is compatible with the addition $+$.
Hence the binary operations $\vee$ on $|\CC^n|$ for all $n$ 
induce a binary operation on the set $|\CC|$, also denoted by $\vee$. 
This gives a structure
of a commutative monoid on the set $|\CC|$.
\par

The following lemma says that 
the total order $\leq$ on $|\CC|$ is compatible with the monoid structure on $|\CC|$.

\begin{lem}\label{lem:compatible}
Let $[M], [M'], [N], [N'] \in |\CC|$ and suppose that $[M] \leq [N]$ and $[M'] \leq [N']$.
Then we have $[M] \vee [M'] \leq [N] \vee [N']$. 
\end{lem}
\begin{proof}
We can easily see that 
the lexicographic order on $\Lambda_n$ is compatible 
with the monoid structure on $\Lambda_n$ given by $+$.
Hence the claim follows.
\end{proof}

Recall that $F$ is the field of fractions of $\oo$.
For an $\oo$-module $M$ of finite length, 
we let $M^\vee$ denote the $\oo$-module $\Hom_\oo(M,F/\oo)$.

\begin{lem}\label{lem:dual}
For any $\oo$-module $M$ of finite length, we have $[M] = [M^\vee]$.
\end{lem}
\begin{proof}
We may assume $M= \oo/\pp^{\lambda_1} \oplus \cdots \oplus \oo/\pp^{\lambda_n}$.
Since $(\ )^\vee$ commutes with finite direct sums,
we may further assume that $M = \oo/\pp^\lambda$.
Then we have $M^\vee \cong \pp^{-\lambda}/\oo$.
Hence by choosing a uniformizer $\varpi \in \pp$,
we obtain a desired isomorphism $M \cong M^\vee$.
\end{proof}

\begin{lem}\label{inj_surj}
Let $M, M'$ be $\oo$-modules of finite length.
\begin{enumerate}
\item
If there exists an injective homomorphism $M' \hookrightarrow M$, 
then we have $[M'] \leq [M]$.
\item 
If there exists a surjective homomorphism $M \twoheadrightarrow M'$, 
then we have $[M'] \leq [M]$.
\end{enumerate}
\end{lem}
\begin{proof}
Since an injective homomorphism $M' \hookrightarrow M$ induces
a surjective homomorphism $M^\vee \twoheadrightarrow {M'}^\vee$,
the claim (1) follows from the claim (2) and Lemma \ref{lem:dual}.
Let us prove the claim (2) below.
\par

Let $M, M'$ be $\oo$-modules of finite length, 
and suppose that there exists a surjective homomorphism $M \twoheadrightarrow M'$. 
Let us take an integer $n \geq 1$ such that both $[M]$ and $[M']$ belong to $|\CC^n|$.
We prove the claim by induction on $n$.
If $n=1$, then the claim is clear.
We assume $n >1$.
Let us write $\seq_n([M])=(\lambda_1,\ldots,\lambda_n)$ 
and $\seq_n([M'])=(\lambda'_1,\ldots,\lambda'_n)$.
\par

First, suppose that $\lambda_1 > \lambda'_1$. 
Then we have $[M'] < [M]$ as claimed.
Next, suppose that $\lambda_1 = \lambda'_1$.
Then both $M/\pp^{\lambda_1} M$ and $M'/\pp^{\lambda_1} M'$
are isomorphic to $(\oo/\pp^{\lambda_1})^{\oplus n}$ and
we have $[M] = [\pp^{\lambda_1} M] \vee [(\oo/\pp^{\lambda_1})^{\oplus n}]$
and $[M'] = [\pp^{\lambda_1} M'] \vee [(\oo/\pp^{\lambda_1})^{\oplus n}]$.
Note that the surjective homomorphism $M \twoheadrightarrow M'$
induces a surjective homomorphism $\pp^{\lambda_1} M \twoheadrightarrow
\pp^{\lambda_1} M'$. Hence by Lemma \ref{lem:compatible},
we are reduced to proving the claim (2) for $\pp^{\lambda_1} M$ and $\pp^{\lambda_1} M'$.
Since both $[\pp^{\lambda_1} M]$ and $[\pp^{\lambda_1} M']$ belong to $|\CC^{n-1}|$, 
the inductive hypothesis proves the claim in the case where $\lambda_1 = \lambda'_1$.
\par

Finally, suppose that $\lambda_1 < \lambda'_1$.
Again in this case, 
the surjective homomorphism $M \twoheadrightarrow M'$ induces 
a surjective homomorphism $\pp^{\lambda_1} M \twoheadrightarrow \pp^{\lambda_1} M'$. 
Note that $\pp^{\lambda_1} M$ is generated by less than $n$ elements, 
whereas the minimum number of generators of $\pp^{\lambda_1} M'$ is equal to $n$.
This leads to a contradiction.
\end{proof}

\begin{prop}[Convexity] \label{prop:convexity}
Let 
\begin{equation} \label{eq:shortex}
0 \to M' \to M \to M'' \to 0
\end{equation}
be a short exact sequence of $\oo$-modules of finite length.
Then we have the inequality
\[
[M] \geq [M'] \vee [M''].
\]
\end{prop}
\begin{proof}
Let $n$, $n'$, and $n''$ denote the minimal numbers of generators
of the $\oo$-modules $M$, $M'$ and $M''$, respectively.
We prove the claim by induction on $n'+n''$.
\par

If $n'+n''=0$, then we have $M'=M''=0$ and the claim is clear.
Since $M \to M''$ and $M^\vee \to M'^\vee$ are surjective, 
we have $n \geq n''$ and $n \geq n'$.
If $n > \max\{n',n''\}$, then the claim is obvious.
Hence we may assume that $n=\max\{n',n''\}$.
By considering the short exact sequence
\[
0 \to M''^\vee \to M^\vee \to M'^\vee \to 0
\]
instead of \eqref{eq:shortex} if necessary, 
we may further assume that $n=n''$.
Let us write $\seq_n(M'')=(\lambda''_1,\ldots,\lambda''_n)$ and $I = \pp^{\lambda''_1}$.
Then both $M/I M$ and $M''/I M''$ are isomorphic to $(\oo/I)^{\oplus n}$,
and we have $[M]=[I M] \vee [(\oo/I)^{\oplus n}]$ and $[M'']=[I M''] \vee [(\oo/I)^{\oplus n}]$.
Moreover \eqref{eq:shortex} induces a short exact sequence
\[
0 \to M' \to IM \to IM'' \to 0.
\]
Since $\seq_n(IM'')=(0,\lambda''_2-\lambda''_1,\ldots,\lambda''_n-\lambda''_1)$,
the minimal number of generators of $IM''$ is strictly smaller than $n''$.
Hence by induction, we have $[IM] \geq [M'] \vee [IM'']$.
By adding $[(\oo/I)^{\oplus n}]$ to both sides and
using Lemma \ref{lem:compatible}, we obtain the desired inequality.
\end{proof}

\begin{lem} \label{lem:equality}
Let $M$ be an $\oo$-module of finite length.
Then for any non-zero ideal $I \subset \oo$, 
we have $[M] = [IM] \vee [M/IM]$.
\end{lem}
\begin{proof}
Let us write $I = \pp^\lambda$.
Let us choose a positive integer $n$ such that $[M] \in |\CC^n|$.
Let us write $\seq_n([M]) = (\lambda_1,\ldots,\lambda_n)$.
For $i=1,\ldots,n$, set $\lambda''_i = \min\{\lambda,\lambda_i\}$.
Then $M/IM$ is isomorphic to $\bigoplus_{i=1}^n\oo/\pp^{\lambda''_i}$ and
$IM$ is isomorphic to $\bigoplus_{i=1}^n \pp^{\lambda''_i}/\pp^{\lambda_i}$.
Thus, we have $[M]=[IM]\vee [M/IM]$, 
as desired.
\end{proof}

\begin{prop}[Uniqueness]\label{prop:uniqueness}
Suppose that $[M], [M'], [M''] \in |\CC|$ satisfy $[M] = [M'] \vee [M'']$.
Then there exists a unique $\oo$-submodule $N \subset M$ 
satisfying $[N] = [M']$ and $[M/N] = [M'']$.
Moreover, for any $\oo$-submodule $N' \subset M$ other than $N$, 
we have either $[N'] < [M']$ or $[M/N'] < [M'']$.
\end{prop}
\begin{proof}
First, we prove the existence and the uniqueness of $N$.
Let $n$, $n'$, and $n''$ denote 
the minimal numbers of generators of the $\oo$-modules $M$, $M'$ and $M''$, respectively.
We prove the claim by induction on $n'+n''$.
\par

If $n'+n''=0$, then we have $M=M'=M''=0$ and the claim is obvious.
The relation $[M]=[M']\vee[M'']$ implies that $n=\max\{n',n''\}$.
By considering $M^\vee$ instead of $M$ if necessary, 
we may assume that $n=n''$.
Let us write $\seq_n(M'')=(\lambda''_1,\ldots,\lambda''_n)$ and $I = \pp^{\lambda''_1}$.
Then for any $\oo$-submodule $N \subset M$ satisfying $[M/N]=[M'']$, we have $N \subset IM$.
Since $\seq_n(IM'')=(0,\lambda''_2-\lambda''_1,\ldots,\lambda''_n-\lambda''_1)$,
the minimal number of generators of $IM''$ is strictly smaller than $n''$.
Hence by induction, 
there exists a unique $\oo$-submodule $N \subset IM$ 
satisfying $[N] = [M']$ and $[IM/N]=[IM'']$.
Since $N$ is contained in $IM$, the $\oo$-module
$I(M/N)$ is isomorphic to $IM/N$ and hence
$(M/N)/I(M/N)$ is isomorphic to $M/IM$.
It follows from Lemma \ref{lem:equality}
that we have $[M/N] = [IM/N] \vee [M/IM] = [M'']$.
Hence the claim follows.
\par

Finally, let us prove the last assertion of the proposition.
Let $N' \subset M$ be an $\oo$-submodule other than $N$.
Suppose that $[N'] \geq [M']$ and $[M/N'] \geq [M'']$.
Since $N' \neq N$ we have either $[N']>[M']$ or $[M/N'] > [M'']$.
Hence it follows from Lemma \ref{lem:compatible} and Proposition \ref{prop:convexity} 
that
\[
[M] \geq [N'] \vee [M/N'] > [M'] + [M''] = [M],
\]
which is a contradiction.
Hence we have either $[N'] < [M']$ or $[M/N'] < [M'']$. 
This completes the proof.
\end{proof}

\begin{cor}\label{fil0}
Suppose that $[M], [M_1], \ldots, [M_r] \in |\CC|$ satisfy $[M] = [M_1] \vee \cdots \vee [M_r]$.
Then there exists a unique increasing filtration
\[
0 = \Fil^0_0 M \subset  \cdots
\subset \Fil^0_{r}M = M
\]
of $M$ by $\oo$-submodules satisfying $[M_i]=[\Gr_i^{\Fil^0} M]$ for $i=1,\ldots,r$,
where $\Gr_i^{\Fil^0} M = \Fil^0_i M / \Fil^0_{i-1} M$.
Moreover, for any filtration
\[
0 = \Fil_0 M \subset \cdots
\subset \Fil_r M = M
\]
of $M$ by $\oo$-submodules other than $\Fil^0_\bullet M$, 
we have $[\Gr_i^{\Fil} M] < [M_i]$ for some $i \in \{1,\ldots,r\}$.
\end{cor}
\begin{proof}
We prove the existence and the uniqueness of $\Fil^0_{\bullet}M$ by induction on $r$. 
If $r = 1$, it is obvious. 
If $r > 1$, set $[M'] = [M_1]$ and $[M''] = [M_2] \vee \dots \vee [M_{r}]$. 
By Proposition \ref{prop:uniqueness}, 
there exists a unique $\oo$-submodule $N$ of $M$ such that 
$[N] = [M_1]$ and $[M/N] = [M_2] \vee \dots \vee [M_{r}]$. 
By the inductive hypothesis, we have a unique filtration $\Fil^0_{\bullet}(M/N)$
satisfying the conditions with respect to $[M/N] = [M_2] \vee \dots \vee [M_{r}]$. 
By setting $\Fil^0_{i+1}M$ to be the inverse image of $\Fil^0_i(M/N)$ for $1 \leq i \leq r-1$, 
and $\Fil^0_{1}M = N$, 
we obtain $\Fil^0_{\bullet}M$. 
\par

The last assertion follows from 
the same argument as in the proof of Proposition \ref{prop:uniqueness}.
\end{proof}

\subsection{Generators of $\oo$-modules}
\begin{lem} \label{lem:sc}
Let $f \colon M \twoheadrightarrow N$ be a surjective homomorphism of $\oo$-modules, 
and $M' \subset M$ an $\oo$-submodule.
Let $x \in N$ and $y \in M/M'$ be elements whose images in $N/f(M')$ coincide.
Then there exists a lift $\wt{y} \in M$ of $y$ satisfying $f(\wt{y})=x$.
\end{lem}
\begin{proof}
Let us take an arbitrary lift $\wt{y}' \in M$ of $y$ and set $x'=f(\wt{y}')$. 
Since the images of $x$ and $x'$ coincide in $N/f(M')$, 
there exists $z \in M'$ satisfying $x-x' = f(z)$. 
Then the element $\wt{y} = \wt{y'} + z \in M$ has the desired property.
\end{proof}

\begin{lem} \label{lem:cover}
Let $N$ be an $\oo$-module of finite length.
Let $L$ and $L'$ be finitely generated free $\oo$-modules of the same rank 
and let $f \colon L \twoheadrightarrow N$ and $f' \colon L'\twoheadrightarrow N$ 
be surjective homomorphisms of $\oo$-modules.
Then there exists an isomorphism
$\alpha \colon L \xto{\cong} L'$ of $\oo$-modules satisfying $f = f' \circ \alpha$.
\end{lem}
\begin{proof}
Since $N$ is of finite length over a noetherian local ring $\oo$, 
one can take a projective cover $\beta \colon P \to N$ of $N$
(see \cite[17.16 Examples (3)]{AF}).
Then there exists homomorphisms $\gamma \colon L \to P$ and $\gamma' \colon L' \to P$ 
satisfying $f = \beta \circ \gamma$ and $f' = \beta \circ \gamma'$.
Since projective covers are essential surjections,
the homomorphisms $\gamma$ and $\gamma'$ are surjective. 
Hence by the projectivity of $P$, 
one can choose a right inverse $s$ and $s'$ of $\gamma$ and $\gamma'$, respectively.
Since $\Ker\, \gamma$ and $\Ker\, \gamma'$ are free $\oo$-modules of the same rank,
there exists an isomorphism 
$\alpha' \colon \Ker\, \gamma \xto{\cong} \Ker\, \gamma'$ of $\oo$-modules.
By taking the direct sum of $\alpha'$ 
and the isomorphism $s(P) \xto{\cong} s'(P)$ given by $s' \circ \gamma$, 
we obtain a desired isomorphism $\alpha \colon L \to L'$.
\end{proof}

\begin{cor} \label{cor:lift}
Let $N$ be an $\oo$-module of finite length generated by $n$ elements $x_1,\ldots,x_n$.
Then for any free $\oo$-module $L$ of rank $n$ 
and for any surjective homomorphism $f \colon L \to N$, 
there exists an $\oo$-basis $y_1,\ldots,y_n$ of $L$ satisfying $f(y_i)=x_i$ for $i=1,\ldots,n$.
\end{cor}
\begin{proof}
Let $L' = \oo^{\oplus n}$ and 
let $f' \colon L' \twoheadrightarrow N$ denote the surjection 
that sends the standard basis of $L'$ to the elements $x_1,\ldots,x_n$.
By applying Lemma \ref{lem:cover}, 
we obtain an isomorphism $\alpha\colon L \xto{\cong} L'$ satisfying $f=f'\circ\alpha$.
Then the image under $\alpha^{-1}$ of the standard basis of $L'$ gives a desired basis of $L$.
\end{proof}

From now on until the end of this section, 
we fix an integer $n \geq 1$ and a partition
\[
\bn = (n_1, \ldots, n_r), \ 
n = n_1 + \cdots + n_r,\ 
n_i \geq 1
\] 
of $n$. 
For $i=1,\ldots, r$, we set
\[
a_i = n_1+\cdots+n_{i-1}+1,
\ 
b_i = n_1+\cdots+n_i.
\]
\par

We use the following terminology.
\begin{defi}
Let $M$ be an $\oo$-module generated by at most $n$ elements.

\begin{enumerate}
\item
We say that an increasing filtration $\Fil_\bullet M$ of $M$ by $\oo$-submodules
is \emph{$\bn$-admissible} if the following conditions are satisfied:

\begin{itemize}
\item
$\Fil_0 M = 0$ and $\Fil_{r} M = M$.
\item 
For $i=1,\ldots,r$, 
the graded quotient $\Gr^\Fil_i M = \Fil_i M/\Fil_{i-1} M$ is generated by at most $n_i$ elements.
\end{itemize}

\item 
Let $\Fil_\bullet M$ be an $\bn$-admissible filtration of $M$. 
We say that a sequence $y_1,\ldots,y_n$ of elements of $M$ 
is \emph{compatible with $\Fil_\bullet M$} 
if, for $i=1,\ldots,r$, 
the $b_i$ elements $y_1, \ldots, y_{b_i}$ generate the $\oo$-module $\Fil_i M$.
\end{enumerate}
\end{defi}

\begin{lem} \label{lem:representatives}
Let $M$ be an $\oo$-module of finite length.
Let $L$ be a free $\oo$-module of rank $n$ 
and let $f \colon L \twoheadrightarrow M$ be a surjective homomorphism of $\oo$-modules.
Suppose that an $\bn$-admissible filtration $\Fil_\bullet L$ of $L$ is given.
Let $\Fil_\bullet M$ denote the filtration on $M$ induced from $\Fil_\bullet L$ via $f$, 
i.e., $\Fil_i M = f(\Fil_i L)$.
For $i=1,\ldots,r$, let 
$f_i \colon \Gr^\Fil_i L \twoheadrightarrow \Gr^\Fil_i M$
denote the surjective homomorphism induced by $f$.
Then we have the following.

\begin{enumerate}
\item
$\Fil_\bullet M$ is an $\bn$-admissible filtration of $M$.

\item 
Let $x_1,\ldots,x_n$ be a sequence of elements of $M$ compatible with $\Fil_\bullet M$. 
Then there exists a sequence $y_1,\ldots,y_n$ of elements of $L$ 
compatible with $\Fil_\bullet L$
such that $x_j=f(y_j)$ for $j=1,\ldots,n$.

\item 
Let $x_1,\ldots,x_n$ be a sequence of elements of $M$ compatible with $\Fil_\bullet M$. 
Suppose that, for $i=1,\ldots,r$, 
an $\oo$-basis $z_{a_i}, \ldots, z_{b_i}$ of $\Gr^\Fil_i L$ is given in such a way that 
for $j=a_i,\ldots,b_i$, the image $f_i(z_j)$ is equal to the class of $x_j$ in $\Gr^\Fil_i M$.
Then there exists a sequence $y_1,\ldots,y_n$ of elements of $L$ 
compatible with $\Fil_\bullet L$
such that $x_j=f(y_j)$ for $j=1,\ldots,n$,
and such that the class of $y_j$ in $\Gr^\Fil_i L$ is equal to $z_j$ for $j=a_i,\ldots,b_i$.
\end{enumerate}
\end{lem}
\begin{proof}
The assertion (1) is clear.
We can deduce the assertion (2) from the assertion (3),
since in the situation of (2) one can find, by using Corollary \ref{cor:lift}, 
an $\oo$-basis $z_{a_i}, \ldots, z_{b_i}$ of $\Gr^\Fil_i L$ 
as in the statement of the assertion of (3) for $i=1,\ldots,r$.
(Here, we note that $\Fil_i L$ is a free $\oo$-module of rank 
$n_1+\dots+n_i$.)
\par

We prove the assertion (3).
Using Lemma \ref{lem:sc}, 
one can choose an element $y_j \in \Fil_i L$ for $j=a_i,\ldots,b_i$ in such a way that 
$f(y_j)=x_j$ and the image of $y_j$ in $\Gr^\Fil_i L$ is equal to $z_j$.
Then the sequence $y_1,\ldots,y_n$ of elements of $L$ has the desired property.
\end{proof}

The following is well-known. 
\begin{lem}\label{lem:permutation}
Let $M$ be an $\oo$-module of finite length, 
and $m_1, \dots, m_r$ be non-negative integers. 
Then the number of filtrations $0 = \Fil_0M \subset \dots \subset \Fil_rM = M$
with $\Gr_i^\Fil M$ generated by exactly $m_i$ elements for any $1 \leq i \leq r$
is invariant under the permutations of $m_1, \dots, m_r$. 
\end{lem}
\begin{proof}[Outline of the proof]
First, reduce to the case where the permutation is an adjacent transposition.  
Then reduce to the case where $r = 2$. 
Finally, use the duality (Lemma \ref{lem:dual}) to treat this case.
\end{proof}

\section{The Mackey decomposition}\label{s.mackey}
In this section, we give the Mackey decomposition (Proposition \ref{mackey}) 
of the invariants by compact open subgroups of the form $\K_{n,\lambda}$. 
As an application, we give a reduction step in the proof of our main results.
\par

\subsection{Invariant subspaces of parabolically induced representations}\label{sec.inv}
Fix an integer $n \geq 1$.
Let us consider the $F$-vector space $F^n$.
We regard an element of $F^n$ as a column vector.
The group $G_n = \GL_n(F)$ acts on $F^n$ from the left by the multiplication.
Let $L_1,L_2 \subset F^n$ be $\oo$-lattices with $L_1 \supset L_2$.
We denote by $\K_{L_1,L_2}$ 
the set of elements $g \in G_n$ satisfying the following conditions:
\begin{itemize}
\item 
We have $g L_1 = L_1$ and $g L_2 = L_2$.
\item 
The endomorphism of the $\oo$-module 
$L_1/L_2$ induced by the multiplication by $g$ is the identity map.
\end{itemize}
Then $\K_{L_1,L_2}$ is a compact open subgroup of $G_n$.

\begin{lem}\label{lem:indep}
The $G_n$-conjugacy class of $\K_{L_1,L_2}$ depends only on $n$ 
and an isomorphism class $[L_1/L_2]$ of the $\oo$-module $L_1/L_2$.
\end{lem}
\begin{proof}
Let $L_1, L_2, L'_1, L'_2$ be $\oo$-lattices of $F^n$ 
such that $L_1 \supset L_2$, $L'_1 \supset L'_2$ 
and that $L_1/L_2$ is isomorphic to $L'_1/L'_2$ as $\oo$-modules.
Let us choose an isomorphism $L_1/L_2 \cong L'_1/L'_2$, 
and let $f$ (\resp $f'$) denote 
the composite $L_1 \twoheadrightarrow L_1/L_2 \xto{\cong} L'_1/L'_2$
(\resp the quotient map $L'_1 \to L'_1/L'_2$).
Then it follows from Lemma \ref{lem:cover} that
there exists an isomorphism $\alpha \colon L_1 \xto{\cong} L'_1$ satisfying $f=f'\circ \alpha$.
By extending $\alpha$ to an automorphism $F^n \xto{\cong} F^n$ by $F$-linearity, 
we obtain an element $g \in G_n$ such that $\alpha(x) = gx$.
It is then straightforward to check that 
$\K_{L'_1,L'_2} = g \K_{L_1,L_2} g^{-1}$.
This completes the proof.
\end{proof}

By abuse of notation, we denote the group $\K_{L_1,L_2}$ by $\K_{n,[L_1/L_2]}$.
We note that, for $[M]$ in $|\CC^n|$,
the group $\K_{n,[M]}$ is well-defined only up to $G_n$-conjugation.
If $\lambda = \seq_n([M])$, 
the $G_n$-conjugacy class of $\K_{n,[M]}$ is equal to the class of $\K_{n,\lambda}$. 
Indeed, if we set $L_1 = \oo^n$ and $L_1 = \oplus_{i=1}^n\pp^{\lambda_i}$ 
with $\lambda = (\lambda_1, \dots, \lambda_n)$, 
then we see that $\K_{L_1,L_2} = \K_{n,\lambda}$. 
\par

Fix a partition $\bn = (n_1,\ldots,n_r)$ of $n$ with integers $n_1, \dots, n_r \geq 1$.
Let $\pi_1,\ldots,\pi_r$ be representations of $G_{n_1}, \ldots, G_{n_r}$
of finite length, respectively.
Consider the representation $\pi_1 \times \cdots \times \pi_r$ of $G_n$,
which is parabolically induced from the representation $\pi_1 \boxtimes \cdots \boxtimes \pi_r$ 
of the standard Levi subgroup $G_{n_1} \times \cdots \times G_{n_r}$ of $G_n$.
Then, for any $[M] \in |\CC^n|$, 
the Mackey decomposition gives the following description of 
the $\K_{n,[M]}$-invariant part of $\pi_1 \times \cdots \times \pi_r$.

\begin{prop}[The Mackey decomposition]\label{mackey}
There exists an isomorphism
\[
(\pi_1 \times \cdots \times \pi_r)^{\K_{n,[M]}}
\cong \bigoplus_{\Fil_\bullet M}
\pi_1^{\K_{n_1,[\Gr_1^{\Fil} M]}}
\otimes \cdots \otimes
\pi_r^{\K_{n_r,[\Gr_r^{\Fil} M]}}
\]
of complex vector spaces. 
Here $\Fil_\bullet M$ in the direct sum above runs over 
the set of $\bn$-admissible filtrations of $M$, 
i.e., the increasing filtrations
\[
0 = \Fil_0 M \subset  \cdots \subset \Fil_{r}M =M
\]
on $M$ by $\oo$-submodules such that for $i=1,\ldots,r$, 
the $\oo$-module $\Gr_i^\Fil M = \Fil_{i}M/\Fil_{i-1} M$ is generated by at most $n_i$ elements.
\end{prop}
\begin{proof}
Let $P_\bn \subset G_n$ denote the standard parabolic subgroup 
corresponding to the partition $\bn=(n_1,\ldots,n_r)$.
Consider the quotient homomorphism $q \colon P_\bn \to G_{n_1} \times \cdots \times G_{n_r}$.
Let us choose a complete set $S \subset G_n$ of representatives 
of the double coset $P_\bn \bs  G_n /\K_{n,[M]}$.
\par

Then the Mackey decomposition yields an isomorphism
\begin{equation} \label{eq:Mackey}
(\pi_1 \times \cdots \times \pi_r)^{\K_{n,[M]}}
\cong \bigoplus_{g \in S}
(\pi_1 \boxtimes \cdots \boxtimes
\pi_r)^{q(P_\bn \cap g \K_{n,[M]} g^{-1})}.
\end{equation}
\par

Let $\FF_M$ denote the set of $\bn$-admissible filtrations on $M$. 
In view of \eqref{eq:Mackey}, 
it suffices to construct a bijection $\alpha \colon P_\bn \bs  G_n /\K_{n,[M]} \xto{\cong} \FF_M$ 
satisfying the following property:
If $P_\bn g \K_{n,[M]}$ corresponds to the filtration $\Fil_\bullet M$ via $\alpha$, 
then the subgroup $q(P_\bn \cap g \K_{n,[M]} g^{-1})$ of $G_{n_1} \times \cdots \times G_{n_r}$ 
is a conjugate of the subgroup 
$\K_{n_1,[\Gr_1^{\Fil} M]} \times \cdots \times \K_{n_r,[\Gr_r^{\Fil} M]}$.
\par

\begin{lem}\label{lem:double}
By choosing a pair $(L_1,L_2)$ of $\oo$-lattices with $L_1 \supset L_2$ 
and an isomorphism $\gamma \colon L_1/L_2 \cong M$ of $\oo$-modules, 
we identify $\K_{n,[M]}$ with $\K_{L_1,L_2}$. 
We denote the composite $L_1 \twoheadrightarrow L_1/L_2 \xto{\gamma} M$ by $f_1$.

\begin{enumerate}
\item
Let $\LL_{M}(F^n)$ be the set of pairs $(L,f)$ of an $\oo$-lattice $L \subset F^n$
and a surjective homomorphism $f \colon L \twoheadrightarrow M$ of $\oo$-modules.
Then there is a (canonical) bijection $G_n/\K_{n,[M]} \rightarrow \LL_{M}(F^n)$ given by 
$g\K_{L_1,L_2} \mapsto (gL_1, y \mapsto f_1(g^{-1}y))$. 

\item
There is a bijection from $G_n / P_\bn$ to the set of $\bn$-admissible filtrations on $L_1$
given by 
$h P_\bn \mapsto \Fil_\bullet^{h}L_1 \coloneqq L_1 \cap h(F e_1 + \cdots + F e_{b_\bullet})$, 
where $\{e_1, \dots, e_n\}$ is the standard basis of $F^n$.

\item
Let $\LL_{M}'(F^n)$ be 
the set of triples $(L,\Fil_\bullet L,f)$ of 
an $\oo$-lattice $L \subset F^n$,
an $\bn$-admissible filtration $\Fil_\bullet L$ on $L$,
and a surjective homomorphism $f \colon L \twoheadrightarrow M$ of $\oo$-modules.
We let the group $G_n$ act on $\LL_M'(F^n)$ by 
$g.(L,\Fil_\bullet L,f) = (gL, g \Fil_\bullet L, y\mapsto f(g^{-1}y))$. 
Then there is a (canonical) bijection 
$P_\bn \bs  G_n /\K_{n,[M]} \rightarrow G_n \backslash \LL_{M}'(F^n)$ given by 
sending $P_\bn g \K_{L_1,L_2}$ to the $G_n$-orbit of 
$(gL_1, \Fil_\bullet gL_1, y \mapsto f_1(g^{-1}y))$, 
where $\Fil_{i} gL_1 \coloneqq gL_1 \cap (F e_1 + \cdots + F e_{b_{i}})$.
\end{enumerate}
\end{lem}
\begin{proof}
We show (1). 
We let the group $G_n$ act from the left on the set $\LL_{M}(F^n)$ 
by the rule $g.(L,f) = (gL,y \mapsto f(g^{-1}y))$.
One can prove that the action of $G_n$ on $\LL_{M}(F^n)$ is transitive in the following way.
Let $(L,f)$ and $(L',f')$ be two elements of $\LL_{M}(F^n)$.
Then by Lemma \ref{lem:cover}, there exists an isomorphism
$\beta \colon L \xto{\cong} L'$ satisfying $f = f' \circ \beta$.
By extending $\beta$ to an automorphism of $F^n$ by $F$-linearity,
we obtain an element $g \in G_n$ such that $\beta(x) = gx$. 
Then we have $(L',f') = g.(L,f)$.
Hence the map $g \mapsto g.(L_1,f_1)$ gives a surjective map $G_n \rightarrow \LL_{M}(F^n)$. 
Since the stabilizer of $(L_1,f_1)$ with respect to the action of $G_n$ 
is equal to $\K_{L_1,L_2}$, 
it gives the desired bijection. 
\par

It is straightforward to check that 
this bijection does not depend on the choice of the triple $(L_1, L_2, \gamma)$ 
in the following sense.
Let $(L'_1, L'_2, \gamma')$ be another choice. 
It follows from the proof of Lemma \ref{lem:indep} that 
there exists $g \in G_n$ satisfying $g L_1 = L'_1$, $g L_2 = L'_2$ 
and $\gamma(y \mod{L_2}) =\gamma'(gy \mod{L'_2})$ for all $y \in L_1$.
Then for any such $g \in G_n$, we have $\K_{L'_1,L'_2} = g \K_{L_1,L_2} g^{-1}$ 
and the diagram
\[
\begin{CD}
G_n/\K_{L_1,L_2} @>>> \LL_{M}(F^n) \\
@VVV @| \\
G_n/\K_{L'_1,L'_2} @>>> \LL_{M}(F^n)
\end{CD}
\]
is commutative. 
Here the left vertical map sends $h \K_{L_1,L_2}$ to $h g^{-1} \K_{L'_1,L'_2}$.
Hence we obtain (1). 
\par

Note that $G_n / P_\bn$ is naturally identified with 
the set of partial flags $0 = V_0 \subset \dots \subset V_r = F^n$ 
with $\dim(V_i/V_{i-1}) = n_i$ for $i = 1,\dots,r$. 
Note that $(L_1 \cap V') \otimes_\oo F = V'$ for any subspace $V'$ of $F^n$.
On the other hand, if $\Fil_\bullet L_1$ is an $\bn$-admissible filtration of $L_1$, 
since $L_1$ is a free $\oo$-module of rank $n=n_1 + \cdots + n_r$, each subquotient 
$\Gr_i^\Fil L_1$ is a free $\oo$-module of rank $n_i$ for any $i$. Hence we have $L_1 \cap (\Fil_i L_1 \otimes_\oo F) = \Fil_i L_1$ for any $i$. 
Therefore we have (2). 
\par

Since the double cosets in $P_\bn \bs  G_n /\K_{n,[M]}$ are in one-to-one correspondence 
with the $G_n$-orbits in $(G_n/P_\bn) \times (G_n/\K_{n,[M]})$ with respect to the diagonal left $G_n$-action, 
the assertion (3) follows from (1) and (2). 
\end{proof}

We continue the proof of Proposition \ref{mackey}. 
We identify $P_\bn \bs  G_n /\K_{n,[M]}$ with $G_n \backslash \LL_{M}'(F^n)$ 
by Lemma \ref{lem:double}.
By sending the triple $(L,\Fil_\bullet L,f) \in \LL'_M(F^n)$ to the filtration on $M$ 
induced from $\Fil_\bullet L$ via $f$,
we obtain a map $\alpha \colon P_\bn \bs G_n /\K_{n,[M]} \to \FF_M$.
Let $\Fil^\st_\bullet \oo^n$ be the standard $\bn$-admissible filtration on $\oo^n$, 
i.e., the unique $\bn$-admissible filtration on $\oo^n$ such that
the standard basis of $\oo^n$ is a sequence compatible with $\Fil^\st_\bullet \oo^n$. 
Let us fix a surjective homomorphism $f \colon \oo^n \to M$ and let $L$ denote its kernel. 
Then we can regard $\K_{n,[M]}$ as $\K_{\oo^n,L}$. 
In this case, one can describe the map $\alpha$ as follows.
Let $s \in P_\bn \bs G_n /\K_{\oo^n,L}$.
Then by the Iwasawa decomposition, we have
$s = P_\bn k \K_{\oo^n,L}$ for some $k \in \GL_n(\oo)$.
Then $\alpha(s)$ is the filtration
\[
0=f(k^{-1} \Fil^\st_0 \oo^n) \subset \cdots
\subset f(k^{-1} \Fil^\st_r \oo^n)=M
\]
on $M$. 
We note that $k^{-1} \Fil^\st_{i} \oo^n$ is the $\oo$-submodule of $\oo^n$ 
generated by the first $b_i$ columns of $k^{-1}$.
\par

Now let us choose a filtration $\Fil_\bullet M$ on $M$ in $\FF_M$.
Let us fix a sequence $x_1,\ldots,x_n \in M$ compatible with $\Fil_\bullet M$.
By considering the homomorphism $\oo^n \to M$ 
that sends the standard basis to the sequence $x_1,\ldots,x_n$,
one can check that the map $\alpha$ is surjective.
Suppose that two triples $(L,\Fil_\bullet L,f)$ and $(L',\Fil'_\bullet L',f')$ 
are sent to $\Fil_\bullet M$ via $\alpha$.
Let us choose a basis $y_1,\ldots,y_n$ of $L$ 
and a basis $y'_1,\ldots,y'_n$ of $L'$ as in the assertion (2) of Lemma \ref{lem:representatives}. 
By considering the change-of-basis matrix, 
we can see that the two triples are in the same $G_n$-orbit. 
This proves that the map $\alpha$ is injective.
In conclusion, $\alpha \colon P_\bn \bs G_n /\K_{n,[M]} \to \FF_M$ is bijective.
\par

Again, we realize $\K_{n,[M]}$ as $\K_{\oo^n, L}$ for a lattice $L \subset \oo^n$
with a surjection $f \colon \oo^n \twoheadrightarrow M$ such that $\Ker\, f = L$. 
Then by the Iwasawa decomposition, 
any $s \in P_\bn \bs G_n /\K_{\oo^n,L}$ is of the form $s = P_\bn k_s \K_{\oo^n,L}$ 
for some $k_s \in \GL_n(\oo)$.
In this case, the corresponding triple is 
the $G_n$-orbit of $(\oo^n,\Fil^\st_\bullet \oo^n, f_s)$, where $f_s(y) = f(k_s^{-1}y)$. 
In particular, $\Ker\, f_s = k_sL$. 
Then 
\[
P_\bn \cap k_s \K_{\oo^n,L} k_s^{-1}
= \{p \in P_\bn \cap \GL_n(\oo) \;|\; f_s \circ m(p) = f_s\}, 
\]
where $m(p) \colon \oo^n \to \oo^n$ denotes the homomorphism 
given by the multiplication by $p$ from the left.
Recall that $\{e_1, \dots, e_n\}$ is the standard basis of $F^n$. 
For $1 \leq i \leq r$, 
we set 
$L_i$ to be the image of $k_s L \cap (\oo e_1 + \dots + \oo e_{b_i})$ 
under the canonical projection 
\[
\oo^{b_i} = \oo e_1 + \dots + \oo e_{b_i} 
\twoheadrightarrow \oo^{n_i} = \oo e_{a_i} + \dots + \oo e_{b_i}.
\]
Then $\oo^{n_i} \supset L_i$ are lattices in $F^{n_i} = Fe_{a_i} + \dots + Fe_{b_i}$
such that $\oo^{n_i}/L_i \cong \Gr^\Fil_iM$, 
where $\Fil_\bullet M$ is the filtration corresponding to 
the $G_n$-orbit of $(\oo^n,\Fil^\st_\bullet \oo^n, f_s)$. 
Moreover, we have
\[
q(P_\bn \cap k_s \K_{\oo^n,L} k_s^{-1}) \subset \K_{\oo^{n_1},L_1} \times \dots \times \K_{\oo^{n_r},L_r}. 
\] 
\par

We show that this inclusion is indeed an equality.
Let $(k_1, \dots, k_r) \in \K_{\oo^{n_1},L_1} \times \dots \times \K_{\oo^{n_r},L_r}$ be given. 
Set $x_j = f_s(e_j) \in M$ for $1 \leq j \leq n$, 
and set $z_j = k_ie_j \in \oo^{n_i} = \Gr_i^{F^\st}\oo^n$ for $a_i \leq j \leq b_i$. 
Since $k_i$ fixes 
$\oo^{n_i} \ni x \mapsto f_s(x) \bmod \Fil_{i-1}M \in \Gr^\Fil_iM \cong \oo^{n_i}/L_i$, 
we see that 
the image of $f_s(z_j)$ in $\Gr^\Fil_iM$ is the same as the one of $x_j$.
By the assertion (3) of Lemma \ref{lem:representatives},
one can take a sequence $e_1', \dots, e_n' \in \oo^n$ 
which is compatible with $\Fil_\bullet^\st \oo^n$
such that $x_j = f_s(e'_j)$ and 
the class of $e'_j$ in $\Gr_i^{\Fil^\st}\oo^n = \oo^{n_i}$ is equal to $z_j$.
Define $k \in G_n$ so that $e'_j = ke_j$ for $1 \leq j \leq n$.
Since $F e'_1+\dots+Fe'_{n_i} = F e_1+\dots+Fe_{n_i}$ for $1 \leq i \leq r$, 
we have $k \in P_\bn$.
Moreover, since $k$ preserves $\oo^n$ and $f_s(kx) = f_s(x)$, 
it also preserves $k_sL = \Ker\, f_s$. 
Hence $k \in \K_{\oo^n, k_sL} = k_s \K_{\oo^n,L} k_s^{-1}$. 
Since $q(k) = (k_1, \dots, k_r)$, we conclude that 
$q(P_\bn \cap k_s \K_{\oo^n,L} k_s^{-1}) 
= \K_{\oo^{n_1},L_1} \times \dots \times \K_{\oo^{n_r},L_r}$. 
Namely, $q(P_\bn \cap g \K_{n,[M]} g^{-1})$ is 
a $G_{n_1} \times \cdots \times G_{n_r}$-conjugate of 
$\K_{n_1,[\Gr^\Fil_1 M]} \times \cdots \times \K_{n_r,[\Gr^\Fil_r M]}$.
This completes the proof of Proposition \ref{mackey}.
\end{proof}

\begin{rem}\label{rem2}
One can interpret the statement and the proof of Proposition \ref{mackey}
in terms of the topos theory. 
For more precise statements, 
see the previous paper of the second and third authors \cite{KY}.
\end{rem}

\subsection{Proof of the main theorems: a reduction step}\label{sec.reduction}
Let $\pi$ be an irreducible representation of $G_n$.
Then we can write $\pi = \pi' \times \pi_1 \times \dots \times \pi_r$ 
as an irreducible parabolic induction
such that 
\begin{itemize}
\item
$\pi'$ is an irreducible representation such that $L(s,\pi') = 1$; 
\item
$\pi_i = Z(\mm_i)$ with $\mm_i$ of type $\chi_i$ 
for some unramified character $\chi_i$ of $F^\times$; 
\item
if $i \not= j$, then $\chi_i\chi_j^{-1}$ is not of the form $|\cdot|^a$ for any $a \in \Z$. 
\end{itemize}
If we knew Theorem \ref{main1} (\resp Theorem \ref{conj1-2}) 
for $\pi'$ and $\pi_i$ for $1 \leq i \leq r$,
by Proposition \ref{mackey} and Corollary \ref{fil0}, 
we would obtain the same theorem for $\pi$. 
In other words, 
Theorems \ref{main1} and \ref{conj1-2}
are reduced to the following two cases: 
\begin{itemize}
\item
The case where $\pi = Z(\mm)$ with $\mm$ of type $\chi$
for some unramified character $\chi$ of $F^\times$; 
\item
The case where $L(s,\pi) = 1$. 
\end{itemize}
\par

We will deal with the first case in Section \ref{sec.proofs1}, 
whereas 
the second case will be treated in Sections \ref{sec.pf_L=1} and \ref{s.ess_speh}.

\section{Proof of the main theorems: the unipotent case}\label{sec.proofs1}
In this section, we prove Theorems \ref{main1} and \ref{conj1-2} 
for $\pi = Z(\mm)$ with $\mm$ of type $\chi$
for some unramified character $\chi$ of $F^\times$. 

\subsection{Proof of Theorem \ref{main1} for ladder representations of type $\chi$}
\label{sec.pf_ladder}
In this section, we prove Theorem \ref{main1} in the case where 
$\pi = Z([x_1,y_1]_\chi, \dots, [x_t,y_t]_\chi) \in \Irr(G_n)$ is of type $\chi$
with an unramified character $\chi$ of $F^\times$
such that $\pi$ is a ladder representation, i.e., $x_1 > \dots > x_t$ and $y_1 > \dots > y_t$. 
Recall from Example \ref{ex_lambda} (2) that 
\[
\lambda_\pi = \sum_{i=2}^t(0,\dots,0, \underbrace{1,\dots,1}_{\max\{y_i-x_{i-1}+2,0\}}) \in \Lambda_n.
\]
\par

For $[M] \in |\CC^n|$ and for a partition $\bn = (n_1, \dots, n_{t})$ of $n$ with $n_i \in \Z$, 
we set $\NN_{\bn}(M)$ to be the number of $\bn$-admissible filtrations of $M$. 
Here, when $n_i < 0$ for some $i$, 
we understand that $\NN_{\bn}(M) = 0$. 

\begin{prop}\label{alt}
We have 
\[
\dim(\pi^{\K_{n,[M]}}) = \sum_{w \in S_t}\sgn(w) \NN_{\bn_w}(M), 
\]
where $\bn_w = (y_1-x_{w(1)}+1,\dots, y_t-x_{w(t)}+1)$.
\end{prop}
\begin{proof}
By the determinantal formula \cite{LMi}, 
in the Grothendieck group of the category of representations of $G_n$ of finite length, 
we have
\[
\pi = \sum_{w \in S_t} \sgn(w) Z([x_{w(1)}, y_1]_\chi) \times \dots \times Z([x_{w(t)}, y_t]_\chi). 
\]
Here, when $x = y+1$ (\resp $x > y+1$), 
we formally set $Z([x,y]_{\chi}) = \1_{G_0}$ (\resp $Z([x,y]_{\chi}) = 0$).
Note that in \cite{LMi}, the determinantal formula was formulated using the Langlands classification, 
but by taking the Zelevinsky dual, it translates to the statement above. 
\par

Recall that for a compact open subgroup $\K$ of $G_n$, 
the functor $\pi \mapsto \pi^\K$ is exact.
Hence, by Proposition \ref{mackey}, we have 
\begin{align*}
\pi^{\K_{n,[M]}} 
&= \sum_{w \in S_t} \sgn(w) \left(\prod_{i=1}^t Z([x_{w(i)}, y_i]_\chi)\right)^{\K_{n,[M]}}
\\&= \sum_{w \in S_t} \sgn(w) \sum_{\Fil_\bullet^w M} \bigotimes_{i=1}^t
Z([x_{w(i)}, y_i]_\chi)^{\K_{n_i, [\Gr^{\Fil^w}_i M]}}, 
\end{align*}
where $\Fil_\bullet^w M$ runs over the set of $\bn_w$-admissible filtrations
with $\bn_w = (y_1-x_{w(1)}+1,\dots, y_t-x_{w(t)}+1)$. 
Here, if $y_i-x_{w(i)}+1 < 0$ for some $i$, 
we understand that there is no $\bn_w$-admissible filtration. 
Since $Z([x_{w(i)},y_i]_\chi)$ is a character which is trivial on $\GL_{y_i-x_{w(i)}+1}(\oo)$, 
the dimension of $Z([x_{w(i)},y_i]_\chi)^{\K_{n_i, [\Gr^{\Fil^w}_i M]}}$ is always one 
if $y_i-x_{w(i)}+1 \geq 0$.
Hence we obtain the assertion. 
\end{proof}

Set $b = \max_{2 \leq i \leq t}\max\{y_i-x_{i-1}+2,0\}$. 
If $b = 0$, then $\pi$ is unramified so that Theorem \ref{main1} is trivial for $\pi$. 
Hence we may assume that $b > 0$. 
Let $[M_\pi] \in |\CC^n|$ be such that $\seq_n([M_\pi]) = \lambda_\pi$.
Then $M_\pi \cong \oplus_{i=1}^b \oo/\pp^{a_i}$ for some $a_i \geq 1$. 

\begin{lem}\label{bn}
If $[M] \leq [M_\pi]$, 
for any filtration $\Fil_\bullet M$,  
the $\oo$-module $[\Gr_i^\Fil M]$ can be generated by at most $b$ elements.
\end{lem}
\begin{proof}
This follows from Lemma \ref{inj_surj}. 
\end{proof}

Now we calculate the alternating sum in the right-hand side of Proposition \ref{alt}.
We will see that there are many non-trivial cancellations. 
See Section \ref{sec.ex_alternating} below for an explicit example of this calculation. 
\par

Choose $2 \leq a \leq t$ such that $y_a-x_{a-1}+2 = b$. 
The following lemma is a key in computing 
the alternating sum in the right-hand side of Proposition \ref{alt}.

\begin{lem}\label{vanish}
Suppose that $[M] \leq [M_\pi]$. 
\begin{enumerate}
\item
Let $X_1$ be the subset of $S_t$ consisting of $w$ such that 
$w(k) \geq a-1$ for any $k \geq a$. 
For $w \in S_t \setminus X_1$, 
take $1 \leq i,j \leq a-1$ such that 
$w(i)$ achieves the largest value, and $w(j)$ achieves the second largest value
among $\{w(1), \dots, w(a-1)\}$, 
and set $w' = w(i,j)$. 
Then $w(i), w(j) \geq a-1$, and 
the map $w \mapsto w'$ is an involution on $S_t \setminus X_1$. 
Moreover, $\NN_{\bn_w}(M) = \NN_{\bn_{w'}}(M)$. 
In particular, 
\[
\sum_{w \in S_t \setminus X_1}\sgn(w) \NN_{\bn_w}(M) = 0.
\]

\item
Let $X_2$ be the subset of $X_1$ consisting of $w$ such that 
$w(k) \geq a$ for any $k > a$. 
For $w \in X_1 \setminus X_2$, 
take a unique $1 \leq i \leq a-1$ such that $w(i) \geq a$, 
and set $w' = w(i,a)$. 
Then the map $w \mapsto w'$ is an involution on $X_1 \setminus X_2$. 
Moreover, $\NN_{\bn_w}(M) = \NN_{\bn_{w'}}(M)$. 
In particular, 
\[
\sum_{w \in X_1 \setminus X_2}\sgn(w) \NN_{\bn_w}(M) = 0.
\]

\item
Let $S_{(a-1,t-a+1)}$ be the subgroup of $S_t$ 
consisting of $w$ such that $w(k) \geq a$ for any $k \geq a$, 
and set $X_3 = \{(a-1,w(a))w \;|\; w \in S_{(a-1,t-a+1)}\}$.
Then 
\[
X_2 = S_{(a-1,t-a+1)} \sqcup X_3.
\]

\item
Let $X_4$ be the subset of $S_t$ consisting of $w$ such that 
$w(a) = a-1$ and $w(k) < a-1$ for some $k > a$. 
Then $X_4 \subset S_t \setminus X_1$, 
and the involution in (1) preserves $X_4$. 
Moreover, the disjoint union $X_3 \sqcup X_4$ is equal to 
the subset of $S_t$ consisting of $w$ such that $w(a) = a-1$.

\item
For $w \in X_4$, take $1 \le i \leq a-1$ 
such that $w(i)$ achieves the largest value among $\{w(1), \dots, w(a-1)\}$, 
in particular $w(i) \geq a$.
Set $\wt{w} = w(a,i)$ and $X_5 = \{\wt{w} \;|\; w \in X_4\}$. 
Then $X_5 \subset S_t \setminus X_1$, 
and the involution in (1) preserves $X_5$. 
\end{enumerate}
\end{lem}
\begin{proof}
We prove (1).
Let $w \in S_t \setminus X_1$, 
and $1 \leq i,j \leq a-1$ be as in the statement. 
Note that $i$ and $j$ depend on $w$, but
the map $w \mapsto w'$ gives a well-defined involution on $S_t \setminus X_1$. 
Since there exists $k \geq a$ such that $w(k) < a-1$,  
we notice that $w(i),w(j) \geq a-1$. 
Hence
\begin{align*}
\min\{y_i-x_{w(i)}+1, y_j-x_{w(j)}+1, y_i-x_{w(j)}+1, y_j-x_{w(i)}+1\} \geq y_a-x_{a-1}+2 = b.
\end{align*}
By Lemma \ref{bn}, we see that $\NN_{\bn_{w}}(M) = \NN_{\bn_{w'}}(M)$.
Since $\sgn(w') = -\sgn(w)$, the last part follows. 
Hence we obtain (1). 
\par

We prove (2).
When $w \in X_1 \setminus X_2$, 
there exists $k > a$ such that $w(k) = a-1$. 
In particular, $w(a) \geq a$. 
Hence the map $w \mapsto w'$ gives a well-defined involution on $X_1 \setminus X_2$. 
By the same argument as in (1), we obtain (2). 
\par

The assertions (3) are (4) are obvious from the definitions. 
\par

We prove (5). 
Let $w \in X_4$. 
Then $\wt{w}(k) = w(k) < a-1$ for some $k > a$ so that $\wt{w} \not\in X_1$.
Take $1 \leq i \leq a-1$ as in the statement so that $\wt{w} = w(a,i)$. 
Note that $\wt{w}(a) = w(i) \geq a$.
Let $1 \leq j_1,j_2 \leq a-1$ be such that 
$\wt{w}(j_1)$ (\resp $\wt{w}(j_2)$) achieves the largest (\resp the second largest) value
among $\{\wt{w}(1), \dots, \wt{w}(a-1)\}$.
Note that $a \leq \wt{w}(j_1) < w(i)$ and $\wt{w}(j_2) \geq a-1$.
If $\wt{w}(j_2) \geq a$, then $j_1,j_2,i,a$ are all distinct from each other. 
In this case,
\[
(\wt{w})' = \wt{w}(j_1,j_2) = w(a,i)(j_1,j_2) = w(j_1,j_2)(a,i) = \wt{w(j_1,j_2)}. 
\]
Hence we have $(\wt{w})' \in X_5$. 
If $\wt{w}(j_2) = a-1$, then $j_2 = i$. 
In this case, 
\begin{align*}
(\wt{w})' = \wt{w}(j_1,i) = w(a,i)(j_1,i) = w(j_1,i)(a,j_1) = w'(a,j_1) = \wt{w'}. 
\end{align*}
Hence we again have $(\wt{w})' \in X_5$. 
\end{proof}

Now we prove Theorem \ref{main1} for a ladder representation 
$\pi = Z([x_1,y_1]_\chi, \dots, [x_t,y_t]_\chi)$ 
of type $\chi$ with unramified character $\chi$. 
\begin{proof}[Proof of Theorem \ref{main1} for ladder representations of type $\chi$]
When $b = 0$, since $\pi$ is unramified, the assertion is trivial. 
From now on, we assume that $b >0$.
In particular, one has $t \geq 2$. 
\par

Set 
\[
\pi' = 
Z([x_1,y_1]_\chi, \dots, [x_{a-2},y_{a-2}]_\chi, [x_a,y_{a-1}]_\chi, [x_{a+1},y_{a+1}]_\chi, \dots, [x_t,y_t]_\chi). 
\]
This is a ladder representation of some $G_{n'}$. 
We claim that 
\[
\dim(\pi^{\K_{n,[M]}}) = \sum_{M' \subset M} \dim(\pi'^{\K_{n',[M/M']}}), 
\]
where $M'$ runs over the set of $\oo$-submodules of $M$ 
generated by exactly $b$ elements. 
\par

Suppose for a moment that this claim is true. 
Note that 
\[
\lambda_\pi = \lambda_{\pi'} + (0,\dots,0, \underbrace{1,\dots,1}_{b}).
\]
By induction on $t$, we may assume that 
we have $\dim(\pi'^{\K_{n',[M/M']}}) = 0$ if $[M/M'] < [M_{\pi'}]$. 
In particular, $\dim(\pi^{\K_{n,[M]}}) = 0$ if $[M] < [M_\pi]$. 
Moreover, when $[M] = [M_\pi]$, 
by Corollary \ref{fil0}, 
there exists a unique $\oo$-submodule $M'$ of $M$ 
generated by exactly $b$ elements such that $[M/M'] = [M_{\pi'}]$. 
Hence we have 
\[
\sum_{M' \subset M} \dim(\pi'^{\K_{n',[M/M']}}) = 1. 
\]
Therefore, the claim implies that 
\[
\dim(\pi^{\K_{n,[M]}}) = \left\{
\begin{aligned}
&1 \iif [M] = [M_\pi], \\
&0 \iif [M] < [M_\pi].
\end{aligned}
\right. 
\]
For the rest of the proof, we show the claim.
\par

Let $X_1, X_2, X_3, X_4, X_5 \subset S_t$ be as in Lemma \ref{vanish}. 
We denote the inverse map of $S_{(a-1,t-a+1)} \ni w \mapsto (a-1,w(a))w \in X_3$
by $X_3 \ni w \mapsto \wt{w} \in S_{(a-1,t-a+1)}$. 
Then by Lemma \ref{vanish} (1)--(3), we have
\begin{align*}
\dim(\pi^{\K_{n,[M]}}) 
&= 
\sum_{w \in X_2}\sgn(w) \NN_{\bn_w}(M)
\\&=
\sum_{w \in X_3}\sgn(\wt{w}) \left(\NN_{\bn_{\wt{w}}}(M) - \NN_{\bn_{w}}(M)\right).
\end{align*}
For $w \in X_3$, there exists $1 \leq i_0 \leq a-1$ uniquely such that $w(i_0) = \wt{w}(a) \geq a$. 
Since $w(a) = \wt{w}(i_0) = a-1$, we have 
\begin{itemize}
\item
$\min\{y_{i_0}-x_{w(i_0)}+1, y_{i_0}-x_{\wt{w}(i_0)}+1\} \geq y_a-x_{a-1}+2 = b$; 
\item
$y_a-x_{\wt{w}(a)}+1 \geq b$, whereas $y_a-x_{w(a)}+1 = b-1$.
\end{itemize}
By Lemma \ref{bn}, $\NN_{\bn_{\wt{w}}}(M)-\NN_{\bn_{w}}(M)$
is equal to the number of filtrations
\[
0 = \Fil_0M \subset \dots \subset \Fil_tM = M
\]
of $M$ by $\oo$-submodules such that 
\begin{itemize}
\item
$\Gr_i^\Fil M$ is generated by at most $y_i-x_{w(i)}+1$ elements for $i \not= a$;
\item
$\Gr_a^\Fil M$ is generated by exactly $b$ elements.
\end{itemize}
By Lemma \ref{lem:permutation}, 
this number is equal to 
the number of pairs $(M', \Fil'_\bullet (M/M'))$, 
where $M' \subset M$ is an $\oo$-submodule generated by exactly $b$ elements, 
and $\Fil'_\bullet (M/M')$ is a filtration 
\[
0 = \Fil'_0 (M/M') 
\subset \dots \subset 
\Fil'_{a-1}(M/M') \subset \Fil'_{a+1}(M/M')
\subset \dots \subset 
\Fil'_{t} (M/M') = M/M'
\]
of $M/M'$ by $\oo$-submodules such that 
$\Gr_i^{\Fil'} (M/M')$ 
is generated by at most $y_i-x_{w(i)}+1$ elements for $i \not= a$. 
Here, we set $\Gr_i^{\Fil'} (M/M') = \Fil'_i(M/M')/\Fil'_{i-1}(M/M')$ unless $i = a,a+1$, 
and $\Gr_{a+1}^{\Fil'} (M/M') = \Fil'_{a+1}(M/M')/\Fil'_{a-1}(M/M')$.
Therefore, 
\[
\sum_{w \in X_2}\sgn(w) \NN_{\bn_w}(M)
= \sum_{w \in X_3} \sum_{M' \subset M} \sgn(\wt{w}) \NN_{\bn'_{w}}(M/M'), 
\]
where $M'$ runs over the set of $\oo$-submodules of $M$ generated by exactly $b$ elements, 
and we set $\bn'_{w} = (n_{w,1},\dots,n_{w,a-1},n_{w,a+1},\dots,n_{w,t})$ with 
$n_{w,i} = y_i-x_{w(i)}+1$ for $i \not= a$. 
\par

Note that $X_4 \cap X_5 = \emptyset$. 
By the same argument as above, we have 
\begin{align*}
\sum_{w \in X_4 \sqcup X_5}\sgn(w) \NN_{\bn_w}(M)
&= 
\sum_{w \in X_4} \sgn(\wt{w}) \left(\NN_{\bn_{\wt{w}}}(M) - \NN_{\bn_w}(M)\right)
\\&= 
\sum_{w \in X_4} \sum_{M' \subset M} \sgn(\wt{w}) \NN_{\bn'_{w}}(M/M'), 
\end{align*}
where $M'$ runs over the set of $\oo$-submodules of $M$ generated by exactly $b$ elements, 
and $\bn'_{w}$ is as above. 
However, by Lemma \ref{vanish} (1), (4), (5), 
we see that the left-hand side is zero. 
Therefore, 
\[
\dim(\pi^{\K_{n,[M]}}) 
= \sum_{M' \subset M} \sum_{w \in X_3 \sqcup X_4} \sgn(\wt{w}) \NN_{\bn'_{w}}(M/M').
\]
\par

Next, we consider the alternating sum
\[
\dim(\pi'^{\K_{n',[M/M']}}) = \sum_{w' \in S_{t-1}} \sgn(w')\NN_{\bn'_{w'}}(M/M').
\]
Here, we regard $S_{t-1}$ as the set of bijective maps 
\[
w' \colon \{1,\dots,a-1,a+1,\dots,t\} \rightarrow \{1,\dots,a-2,a,\dots,t\} 
\]
by identifying $a-1$ and $a$.
For $w \in X_3 \sqcup X_4$, 
define $w'$ to be the restriction of $w$ to $\{1,\dots,a-1,a+1,\dots,t\}$. 
Then we have a bijective map $X_3 \sqcup X_4 \rightarrow S_{t-1}$
since $X_3 \sqcup X_4$ is the subset of $S_t$ consisting of $w$ such that $w(a) = a-1$.
Note that for $w \in X_3 \sqcup X_4$, 
the sign $\sgn(w')$ of $w'$ as an element of $S_{t-1}$
is equal to $\sgn(\wt{w})$.
\par

Therefore, 
\begin{align*}
\dim(\pi^{\K_{n,[M]}}) 
&= \sum_{M' \subset M} \sum_{w \in X_3 \sqcup X_4} \sgn(\wt{w}) \NN_{\bn'_{w}}(M/M')
\\&= \sum_{M' \subset M} \sum_{w' \in S_{t-1}} \sgn(w') \NN_{\bn'_{w'}}(M/M')
\\&= \sum_{M' \subset M} \dim(\pi'^{\K_{n',[M/M']}}). 
\end{align*}
Hence we obtain the claim. 
This completes the proof of Theorem \ref{main1} for ladder representations of type $\chi$.
\end{proof}

\subsection{Example of calculation of the alternating sum}\label{sec.ex_alternating}
To understand the proof of Theorem \ref{main1} for ladder representations of type $\chi$,
the following explicit example may be helpful. 

\begin{ex}\label{ex_ladder}
For simplicity, we drop $\chi$ from the notation.
Let us consider a ladder representation
\[
\pi = Z([5,7],[3,6],[2,5],[0,3]) \in \Irr(G_{15}). 
\]
Then $\lambda_\pi = (0,\dots,0,1,3,3,3) \in \Lambda_{15}$
so that $M_\pi = \oo/\pp \oplus (\oo/\pp^3)^{\oplus 3}$.
By the determinantal formula, we have 
{\footnotesize
\begin{align*}
\pi &= 
Z([5,7]) \times Z([3,6]) \times Z([2,5]) \times Z([0,3])
-Z([3,7]) \times Z([5,6]) \times Z([2,5]) \times Z([0,3])
\\&
-Z([5,7]) \times Z([3,6]) \times Z([0,5]) \times Z([2,3])
+Z([3,7]) \times Z([5,6]) \times Z([0,5]) \times Z([2,3])
\\&
-Z([5,7]) \times Z([2,6]) \times Z([3,5]) \times Z([0,3])
+Z([2,7]) \times Z([5,6]) \times Z([3,5]) \times Z([0,3])
\\&
+Z([5,7]) \times Z([0,6]) \times Z([3,5]) \times Z([2,3])
-Z([0,7]) \times Z([5,6]) \times Z([3,5]) \times Z([2,3])
\\&
-Z([5,7]) \times Z([0,6]) \times Z([2,5]) \times Z([3,3])
+Z([5,7]) \times Z([2,6]) \times Z([0,5]) \times Z([3,3])
\\&
+Z([0,7]) \times Z([5,6]) \times Z([2,5]) \times Z([3,3])
-Z([2,7]) \times Z([5,6]) \times Z([0,5]) \times Z([3,3])
\\&
-Z([2,7]) \times Z([3,6]) \times Z([5,5]) \times Z([0,3])
+Z([3,7]) \times Z([2,6]) \times Z([5,5]) \times Z([0,3])
\\&
+Z([0,7]) \times Z([3,6]) \times Z([5,5]) \times Z([2,3])
-Z([3,7]) \times Z([0,6]) \times Z([5,5]) \times Z([2,3])
\\&
+Z([2,7]) \times Z([0,6]) \times Z([5,5]) \times Z([3,3])
-Z([0,7]) \times Z([2,6]) \times Z([5,5]) \times Z([3,3])
\\&
+Z([2,7]) \times Z([3,6]) \times Z([0,5]) \times Z([5,3])
-Z([3,7]) \times Z([2,6]) \times Z([0,5]) \times Z([5,3])
\\&
-Z([0,7]) \times Z([3,6]) \times Z([2,5]) \times Z([5,3])
+Z([3,7]) \times Z([0,6]) \times Z([2,5]) \times Z([5,3])
\\&
+Z([0,7]) \times Z([2,6]) \times Z([3,5]) \times Z([5,3])
-Z([2,7]) \times Z([0,6]) \times Z([3,5]) \times Z([5,3]).
\end{align*}
}
By Proposition \ref{mackey}, we have 
\begin{align*}
&\dim(\pi^{\K_{15,\lambda_\pi}})
\\&= 
\NN_{(3,4,4,4)}(M_\pi)-\NN_{(5,2,4,4)}(M_\pi)
-\NN_{(3,4,6,2)}(M_\pi)+\NN_{(5,2,6,2)}(M_\pi)
\\&
-\NN_{(3,5,3,4)}(M_\pi)+\NN_{(6,2,3,4)}(M_\pi)
+\NN_{(3,7,3,2)}(M_\pi)-\NN_{(8,2,3,2)}(M_\pi)
\\&
-\NN_{(3,5,4,1)}(M_\pi)+\NN_{(3,5,6,1)}(M_\pi)
+\NN_{(8,2,4,1)}(M_\pi)-\NN_{(6,2,6,1)}(M_\pi)
\\&
-\NN_{(6,4,1,4)}(M_\pi)+\NN_{(5,5,1,4)}(M_\pi)
+\NN_{(8,4,1,2)}(M_\pi)-\NN_{(5,7,1,2)}(M_\pi)
\\&
+\NN_{(6,7,1,1)}(M_\pi)-\NN_{(8,5,1,1)}(M_\pi)
+\NN_{(6,4,6,-1)}(M_\pi)-\NN_{(5,5,6,-1)}(M_\pi)
\\&
-\NN_{(8,4,4,-1)}(M_\pi)+\NN_{(5,7,4,-1)}(M_\pi)
+\NN_{(8,5,3,-1)}(M_\pi)-\NN_{(6,7,3,-1)}(M_\pi).
\end{align*}
By Lemma \ref{bn}, we have
\begin{align*}
\dim(\pi^{\K_{15,\lambda_\pi}})
&= 
\NN_{(3,4,4,4)}(M_\pi)-\NN_{(4,2,4,4)}(M_\pi)
-\NN_{(3,4,4,2)}(M_\pi)+\NN_{(4,2,4,2)}(M_\pi)
\\&
-\NN_{(3,4,3,4)}(M_\pi)+\NN_{(4,2,3,4)}(M_\pi)
+\NN_{(3,4,3,2)}(M_\pi)-\NN_{(4,2,3,2)}(M_\pi)
\\&
-\NN_{(3,4,4,1)}(M_\pi)+\NN_{(3,4,4,1)}(M_\pi)
+\NN_{(4,2,4,1)}(M_\pi)-\NN_{(4,2,4,1)}(M_\pi)
\\&
-\NN_{(4,4,1,4)}(M_\pi)+\NN_{(4,4,1,4)}(M_\pi)
+\NN_{(4,4,1,2)}(M_\pi)-\NN_{(4,4,1,2)}(M_\pi)
\\&
+\NN_{(4,4,1,1)}(M_\pi)-\NN_{(4,4,1,1)}(M_\pi)
\\&=
\NN_{(3,4,4,4)}(M_\pi)-\NN_{(4,2,4,4)}(M_\pi)
-\NN_{(3,4,4,2)}(M_\pi)+\NN_{(4,2,4,2)}(M_\pi)
\\&
-\NN_{(3,4,3,4)}(M_\pi)+\NN_{(4,2,3,4)}(M_\pi)
+\NN_{(3,4,3,2)}(M_\pi)-\NN_{(4,2,3,2)}(M_\pi).
\end{align*}
Note that 
if a filtration $\Fil_\bullet M_\pi$ satisfies that 
$\Gr_3^\Fil M_\pi$ is generated by exactly $4$ elements, 
then $\Gr_i^\Fil M_\pi$ for $i = 1,2,4$ can be generated by at most $3$ elements 
by Lemmas \ref{lem:permutation}, \ref{inj_surj} and \ref{bn}.
Hence
\begin{align*}
\NN_{(3,4,4,4)}(M_\pi)-\NN_{(3,4,3,4)}(M_\pi) &= \NN_{(3,3,4,3)}(M_\pi)-\NN_{(3,3,3,3)}(M_\pi), \\
\NN_{(4,2,4,4)}(M_\pi)-\NN_{(4,2,3,4)}(M_\pi) &= \NN_{(3,2,4,3)}(M_\pi)-\NN_{(3,2,3,3)}(M_\pi), \\
\NN_{(3,4,4,2)}(M_\pi)-\NN_{(3,4,3,2)}(M_\pi) &= \NN_{(3,3,4,2)}(M_\pi)-\NN_{(3,3,3,2)}(M_\pi), \\
\NN_{(4,2,4,2)}(M_\pi)-\NN_{(4,2,3,2)}(M_\pi) &= \NN_{(3,2,4,2)}(M_\pi)-\NN_{(3,2,3,2)}(M_\pi).
\end{align*}
Therefore, 
\begin{align*}
\dim(\pi^{\K_{15,\lambda_\pi}})
&=
[(\NN_{(3,3,4,3)}(M_\pi)-\NN_{(3,3,3,3)}(M_\pi))
-(\NN_{(3,2,4,3)}(M_\pi)-\NN_{(3,2,3,3)}(M_\pi))]
\\&-[
(\NN_{(3,3,4,2)}(M_\pi)-\NN_{(3,3,3,2)}(M_\pi))
-(\NN_{(3,2,4,2)}(M_\pi)-\NN_{(3,2,3,2)}(M_\pi))
].
\end{align*}
The right-hand side is equal to the number of filtrations
\[
0 = \Fil_0M_\pi \subset \Fil_1M_\pi \subset \Fil_2M_\pi 
\subset \Fil_3M_\pi \subset \Fil_4M_\pi = M_\pi
\]
such that
\begin{itemize}
\item
$\Gr^\Fil_2 M_\pi$ is generated by exactly $3$ elements; 
\item
$\Gr^\Fil_3 M_\pi$ is generated by exactly $4$ elements; 
\item
$\Gr^\Fil_4 M_\pi$ is generated by exactly $3$ elements.
\end{itemize}
Since $M_\pi = \oo/\pp \oplus (\oo/\pp^3)^{\oplus 3}$, 
such a filtration exists uniquely and is given by 
\[
\Fil_1M_\pi = 0, \quad
\Fil_2M_\pi = (\pp^2/\pp^3)^{\oplus 3}, \quad
\Fil_3M_\pi = \oo/\pp \oplus (\pp^1/\pp^3)^{\oplus 3}, \quad
\Fil_4M_\pi = M_\pi.
\]
Therefore, we conclude that 
$\dim(\pi^{\K_{15,\lambda_\pi}}) = 1$, as desired.
\end{ex}

\subsection{Proof of Theorem \ref{main1} for general $Z(\mm)$ of type $\chi$}
Now we consider $\pi = Z(\mm)$ with $\mm$ of type $\chi$
for some unramified character $\chi$ of $F^\times$. 

\begin{lem}\label{lem:sum}
Let $\mm_1$ and $\mm_2$ be multisegments.
Then $Z(\mm_1 + \mm_2)$ appears as a subquotient 
of $Z(\mm_1) \times Z(\mm_2)$ with multiplicity one.
\end{lem}
\begin{proof}
See \cite[Proposition 2.3]{T} (or \cite[Proposition 3.5 (5)]{LMi2}). 
\end{proof}

Recall that when $\mm = \Delta_1+\dots+\Delta_r$, we set $\Card(\mm) = r$. 
\begin{lem} \label{lem:aux0}
Let $\mm$, $\mm_1$, and $\mm_2$ be multisegments.
Suppose that $Z(\mm)$ appears as a subquotient of
$Z(\mm_1) \times Z(\mm_2)$.
Then $Z(\mm^{-})$ appears as a subquotient of
$Z(\mm_1^{-}) \times Z(\mm_2^{-})$ if and only
if $\Card(\mm) = \Card(\mm_1) + \Card(\mm_2)$.
\end{lem}
\begin{proof}
Suppose that
 $Z(\mm^{-})$ appears as a subquotient of
$Z(\mm_1^{-}) \times Z(\mm_2^{-})$.
By considering cuspidal supports, 
we have $l(\mm^{-}) = l(\mm_1^{-})
+ l(\mm_2^{-})$.
For a similar reason, we have 
$l(\mm) = l(\mm_1) + l(\mm_2)$.
Since $l(\mm) = l(\mm^{-}) + \Card(\mm)$,
we have the desired equality 
$\Card(\mm) = \Card(\mm_1) + \Card(\mm_2)$.
\par

Conversely, suppose that
the equality
$\Card(\mm) = \Card(\mm_1) + \Card(\mm_2)$
holds.
We set $c=\Card(\mm)$. Then the $c$-th derivatives
of $Z(\mm)$ and $Z(\mm_1) \times Z(\mm_2)$
are equal to $Z(\mm^{-})$ and $Z(\mm_1^{-})
\times Z(\mm_2^{-})$, respectively.
Since the $c$-th derivative is an exact functor
(cf.\ \cite[3.2, 3.5]{BZ}), the assertion follows.
\end{proof}

\begin{proof}[Proof of Theorem \ref{main1} for $\pi = Z(\mm)$ of type $\chi$]
Let $\pi = Z(\mm)$ be an irreducible representation of $G_n$, 
where $\mm = \Delta_1+\dots+\Delta_r$ is a multisegment of type $\chi$
for some unramified character $\chi$ of $F^\times$. 
Let $t_\mm$ be the number 
of pairs of linked segments in $\{\Delta_1, \dots, \Delta_r\}$. 
Note that $t_\mm \leq \binom{l(\mm)}{2}$ since $r \leq l(\mm)$.
\par

We prove the claim by induction 
on the element $(l(\mm),t_\mm)$ 
in the set $S = \{(l,t) \in \Z_{\ge 0}^2 \;|\; t \leq \binom{l}{2}\}$.
Here we endow this set with the following total order. 
We have $(l,t) \le (l',t')$ 
if and only if
we have either $l<l'$, or $l=l'$ and $t \leq t'$.
Note that for a fixed element $(l,t) \in S$, 
there are only finitely many elements in $S$ that are less than $(l,t)$.
\par

Recall that we have a decomposition $\mm = \mm_\maxi + \mm^\rest$
as in Section \ref{sec.def_lambda}. 
We note that $Z(\mm_\maxi)$ is a ladder representation. 
In particular, if $\mm = \mm_\maxi$, 
then we have the claim for $\mm$ (Section \ref{sec.pf_ladder}).
\par

From now on, we assume that $\mm_\maxi \neq \mm$.
Set 
\[
\Pi = Z(\mm_\maxi) \times Z(\mm^\rest). 
\]
Since $l(\mm^\rest) < l(\mm)$, 
it follows from Proposition \ref{mackey}, Corollary \ref{cor:fund4} and
the inductive hypothesis that 
\[
\dim(\Pi^{\K_{n,\lambda}}) = \left\{
\begin{aligned}
&1 \iif \lambda = \lambda_\mm, \\
&0 \iif \lambda < \lambda_\mm. 
\end{aligned}
\right. 
\]
It follows from Lemma \ref{lem:sum}
that $Z(\mm)$ appears as a subquotient of $\Pi$.
This implies that 
the $\K_{n,\lambda}$-invariant part
of $Z(\mm)$ is equal to zero if $\lambda < \lambda_{\mm}$.
Hence it remains to show that the $\K_{n,\lambda_{\mm}}$-invariant 
part of $Z(\mm)$ is one-dimensional.
To do this, we may assume that $\Pi$ is reducible, 
which implies that $t_\mm > 0$.
\par

For an irreducible representation $\pi$ of $G_n$
and a representation $\sigma$ of $G_n$ of finite length, we write
$\pi \subq \sigma$ if $\pi$ appears as a subquotient of $\sigma$.
Let $\mm' \neq \mm$ be a multisegment and suppose that $Z(\mm') \subq \Pi$.
It follows from \cite[7.1 Theorem]{Z} that 
$\mm'$ is obtained by successively applying 
elementary operations to $\mm$.
In particular we have $l(\mm') = l(\mm)$
and $t_{\mm'} < t_\mm$.
Hence by the inductive hypothesis, we have
\[
\dim(Z(\mm')^{\K_{n,\lambda'}}) = \left\{
\begin{aligned}
&1 \iif \lambda' = \lambda_{\mm'}, \\
&0 \iif \lambda' < \lambda_{\mm'}. 
\end{aligned}
\right. 
\]
Note that this implies $\lambda_{\mm'} \ge \lambda_{\mm}$.
In fact, if $\lambda_{\mm'} < \lambda_{\mm}$, then
the $\K_{n,\lambda_{\mm'}}$-invariant part of
$\Pi$ would be non-zero, which is a contradiction.
\par

Now we claim that $\lambda_{\mm'} > \lambda_{\mm}$.
For a proof by contradiction, suppose that $\lambda_{\mm'} =\lambda_{\mm}$.
Since $l(\mm^\ram) = |\lambda_{\mm}|$ and $l(\mm'^\ram) = |\lambda_{\mm'}|$, 
by Proposition \ref{lem:fund3}, we have
\[
l(\mm'^\ram) = l(\mm^\ram) = l((\mm_\maxi)^\ram) + l((\mm^\rest)^\ram).
\]
In particular, we have 
\[
\Card(\mm'^\sharp) = 
l(\mm') - l(\mm'^\ram)
= l(\mm) - l(\mm^\ram) = \Card(\mm^\sharp).
\]
\par

By our assumption, we have 
$Z(\mm), Z(\mm') \subq Z(\mm_\maxi) \times Z(\mm^\rest)$.
Proposition \ref{lem:fund3} together with Lemma \ref{lem:sum} implies that
$Z(\mm^\ram) \subq Z((\mm_\maxi)^\ram) \times Z((\mm^\rest)^\ram)$.
By taking the Zelevinsky duals, we have
$Z(\mm^\sharp), Z(\mm'^\sharp) \subq 
Z((\mm_\maxi)^\sharp) \times Z((\mm^\rest)^\sharp)$,
and 
$Z((\mm^\sharp)^{-}) \subq
Z(((\mm_\maxi)^\sharp)^{-}) \times Z(((\mm^\rest)^\sharp)^{-})$.
Hence it follows from Lemma \ref{lem:aux0} that
\[
\Card(\mm^\sharp)=\Card((\mm_\maxi)^\sharp)
+ \Card((\mm^\rest)^\sharp).
\]
Since we have seen that $\Card(\mm'^\sharp) = \Card(\mm^\sharp)$,
it again follows from 
Lemma \ref{lem:aux0} that
$Z((\mm'^\sharp)^{-}) \subq
Z(((\mm_\maxi)^\sharp)^{-}) \times Z(((\mm^\rest)^\sharp)^{-})$.
Again by taking the Zelevinsky duals, we see that
\[
Z(\mm'^\ram) \subq
Z((\mm_\maxi)^\ram) \times Z((\mm^\rest)^\ram).
\]
This implies that $\mm'^\ram$ is obtained from
$\mm^\ram = (\mm_\maxi)^\ram + (\mm^\rest)^\ram$
by a successive chain of elementary operations.
\par

Since we have assumed that $\lambda_{\mm'} = \lambda_{\mm}$,
it follows that $\mm'^\ram = \mm^\ram$ and hence
$(\mm'^\sharp)^{-} = (\mm^\sharp)^{-}$.
Observe that for any integer $a\in \Z$, 
the number of segments in $\mm'$ that contain $\chi |\cdot|^a$ 
is equal to the number of segments in $\mm$ that contain $\chi |\cdot|^a$.
Hence the equality 
$(\mm'^\sharp)^{-} = (\mm^\sharp)^{-}$
implies the equality $\mm'^\sharp = \mm^\sharp$.
By taking the Zelevinsky duals, we obtain the
equality $\mm' =\mm$, which is a contradiction.
This completes the proof of the inequality
$\lambda_{\mm'} > \lambda_{\mm}$.
\par

Since $\mm' \neq \mm$ is an arbitrary multisegment satisfying $\mm' \subq \Pi$, 
we see that 
the equation $\dim(\Pi^{\K_{n,\lambda_{\mm}}}) = 1$ implies  
$\dim(Z(\mm)^{\K_{n,\lambda_{\mm}}}) = 1$. 
This completes the proof.
\end{proof}

\subsection{Proof of Theorem \ref{conj1-2} for $Z(\mm)$ of type $\chi$}
\label{sec.pf_conj_unip}
In this section, we give a proof of Theorem \ref{conj1-2} 
for $\pi = Z(\mm)$ with $\mm$ of type $\chi$, 
where $\chi$ is an unramified character of $F^\times$.
\par

We consider the polynomial ring $R = \Z[x_1,x_2,\ldots]$
in countably many variables $\{x_i\}_{i \geq 1}$.
For an $\oo$-module $M$ of finite length, 
we define a homomorphism $\xi_M \colon R \to \Z$ of $\Z$-modules as follows. 
We set $\xi_M(1)=1$ if $M=0$, and $\xi_M(1)=0$ otherwise. 
For a monomial $x_{m_1} \cdots x_{m_s}$ in $R$, 
we define its image by $\xi_M$ to be the number of increasing filtrations
\[
0 = \Fil_0 M \subset \cdots \subset \Fil_s M = M
\]
on $M$ by $\oo$-submodules such that for $i=1,\ldots,s$, 
the $i$-th graded piece $\Gr^\Fil_i M$ is generated exactly by $m_i$ elements. 
By Lemma 4.13, the homomorphism $\xi_M$ is well-defined.
\par

For an integer $m \geq 0$, we set
\[
y_m = 1 + x_1 + \cdots + x_m \in R.
\]

\begin{lem} \label{product_y}
Let $M$ be an $\oo$-module of finite length.
Then the integer $\xi_M(y_{m_1} \cdots y_{m_s})$ is equal to
the number $\NN_{(m_1,\ldots,m_s)}(M)$
of $(m_1,\ldots,m_s)$-admissible filtrations on $M$.
\end{lem}
\begin{proof}
This is immediate from the definition of
the homomorphism $\xi_M$ and the definition of
$(m_1,\ldots,m_s)$-admissible filtrations.
\end{proof}

By setting $\deg x_m = m$ for $m \geq 1$, 
we regard $R$ as a graded ring. 
For any integer $m \ge 0$,
let $R_m$ denote the degree-$m$-part of $R$
and set 
\[
I_m = \bigoplus_{i \ge m} R_i.
\]
Then $I_m$ is an ideal of $R$ and we have
$I_{m} \cdot I_{m'} \subset I_{m+m'}$.

\begin{lem} \label{I_m}
Let $m \ge 0$ be an integer and let
$M$ be an $\oo$-module of length less than $m$.
Then we have $\xi_M(I_m)=0$.
\end{lem}
\begin{proof}
Let $f=x_{m_1} \cdots x_{m_s}$ be an arbitrary monomial
that belongs to $I_m$.
It suffices to show $\xi_M(f)=0$.
By definition of $I_m$, we have $m_1+ \cdots + m_s \ge m$.
Suppose that there exists an increasing filtration
\[
0 = \Fil_0 M \subset \cdots \subset \Fil_s M = M
\]
on $M$ by $\oo$-submodules such that for $i=1,\ldots,s$, 
the $i$-th graded piece $\Gr^\Fil_i M$ is generated exactly
by $m_i$ elements. Then, since $\Gr^\Fil_i M$ is of length
at least $m_i$, the length of $M$ is at least
$m_1 + \cdots + m_s \ge m$, 
which is a contradiction.
Hence by the definition of $\xi_M$, we have $\xi_M(f)=0$
as desired.
\end{proof}

Now we prove Theorem \ref{conj1-2} for $\pi = Z(\mm)$ with $\mm$ of type $\chi$. 
\begin{proof}[Proof of Theorem \ref{conj1-2} for $\pi = Z(\mm)$ of type $\chi$]
Let us write
\[
\mm^\sharp = \Delta_1 + \cdots + \Delta_s.
\]
For $i=1,\ldots,s$, we set 
$\pi_i = Z(\Delta_i^\sharp)$.
Let $n$ and $n_i$ be such that $\pi \in \Irr(G_n)$ and $\pi_i \in \Irr(G_{n_i})$.
Then $\pi$ appears as a subquotient of
$\pi_1 \times \cdots \times \pi_s$ and we have
$|\lambda_\pi| = |\lambda_{\pi_1}| + \cdots + |\lambda_{\pi_s}|$.
Let $\lambda = (\lambda_1,\ldots,\lambda_n) \in \Lambda_n$ 
be such that $|\lambda| < |\lambda_\pi|$. 
Then $\pi^{\K_{n,\lambda}}$ is a subquotient of
$(\pi_1 \times \cdots \times \pi_{s})^{\K_{n,\lambda}}$.
Let $M = \oo/\pp^{\lambda_1} \oplus \cdots \oplus \oo/\pp^{\lambda_n}$.
By Proposition \ref{mackey}, we have
\[
(\pi_1 \times \cdots \times \pi_{s})^{\K_{n,\lambda}}
\cong \bigoplus_{\Fil_\bullet M} 
\pi_1^{\K_{n_1,[\Gr^\Fil_1 M]}} \otimes \cdots \otimes
\pi_r^{\K_{n_{s},[\Gr^\Fil_{s} M]}}
\]
where $\Fil_\bullet M$ runs over the set of increasing filtrations
\[
0 = \Fil_0 M \subset \cdots \subset \Fil_{s} M =M
\]
on $M$ by $\oo$-submodules such that for $i=1,\ldots,s,$
the $\oo$-module $\Gr^\Fil_i M = \Fil_i M/\Fil_{i-1} M$
is generated by at most $n_i$ elements.
Fix such a filtration $\Fil_\bullet M$.
Since 
\begin{align*}
& |\lambda_{\pi_1}|+ \cdots + |\lambda_{\pi_{s}}|
= |\lambda_\pi| \\
> & |\lambda| = \length_\oo M
= \length_\oo \Gr^\Fil_1 M + \cdots + \length_\oo \Gr^\Fil_{s} M,
\end{align*}
we have $|\lambda_{\pi_i}| > \length_\oo \Gr^\Fil_i M$
for some $i$.
If we knew the claim for $\pi_i$ for any $i = 1,\dots, s$, 
then we would have $(\pi_1 \times \cdots \times \pi_{s})^{\K_{n,\lambda}} = 0$, 
which implies that $\pi^{\K_{n, \lambda}} = 0$. 
Hence we reduce the claim to the case where $s=1$.
\par

From now on we assume that $s=1$.
Let us write $\Delta_1 = [1,n]_\chi$
for some unramified character $\chi$.
Then $\pi = Z([1,1]_\chi + \cdots + [n,n]_\chi)$
is an unramified twist of the Steinberg representation.
By Tadi\'c's determinantal formula \cite{Tdet},
we have
\[
\pi = \sum_{r=1}^n
(-1)^{n-r}
\sum_{0=n_0 < n_1 < \cdots < n_r =n}
Z([n_0+1,n_1]_\chi)
\times \cdots \times
Z([n_{r-1}+1,n_r]_\chi)
\]
in the Grothendieck group of the category of representations of $G_n$ of finite length.
Then it follows from Proposition \ref{mackey} that, 
for any $\oo$-module $M$ of finite length, the dimension of the
$\K_{n,[M]}$-invariant part $\pi^{\K_{n,[M]}}$
is equal to the number
\[
\sum_{r=1}^n
(-1)^{n-r}
\sum_{0=n_0 < n_1 < \cdots < n_r =n}
\NN_{(n_1-n_0,\ldots,n_r-n_{r-1})}(M).
\]
\par

We set
\begin{equation} \label{eq:Hall}
f_n = \sum_{r=1}^n
(-1)^{n-r}
\sum_{0=n_0 < n_1 < \cdots < n_r =n}
y_{n_1-n_0} \cdots y_{n_r-n_{r-1}} \in R.
\end{equation}
Then it follows from
Lemma \ref{product_y} that for any $\oo$-module $M$ of finite length, 
the dimension of $\pi^{\K_{n,[M]}}$
is equal to $\xi_M(f_n)$.
Therefore, it suffices to prove that $\xi_M(f_n) = 0$ 
for any $\oo$-module $M$ of length at most $n-2$.
By Lemma \ref{I_m}, it suffices to show that
$f_n$ belongs to the ideal $I_{n-1}$.
\par

Let us consider the ring $R[[t]]$ of formal power series in the variable $t$. 
We set
\[
h = \sum_{i=1}^{\infty} y_i t^i \in t R[[t]].
\]
Then $f_n$ is equal to 
the coefficient of $t^n$ in
\[
F = (-1)^n\sum_{r=0}^{\infty} (-1)^r h^r.
\]
Since
\[
h = \frac{t + \sum_{i=1}^\infty x_i t^i}{1-t},
\]
we have
\[
F = \frac{(-1)^n}{1+h} = \frac{(-1)^n(1-t)}{1 + \sum_{i=1}^n x_i t^i}.
\]
Since the coefficients of $t^i$ in $(1 + \sum_{i=1}^n x_i t^i)^{-1}$
belongs to $R_i$ for any $i \ge 0$, the claim follows.
\end{proof}

\section{Proof of the main theorems: the case where $L(s,\pi)=1$}\label{sec.pf_L=1}
In this section, we prove Theorem \ref{conj1-2} for $\pi \in \Irr(G_n)$ with $L(s,\pi) = 1$, 
and we reduce Theorem \ref{main1} to the case of Speh representations.

\subsection{Proof of Theorem \ref{conj1-2} when $L(s,\pi) =1$}
First, we reduce Theorem \ref{conj1-2} for $\pi$ to the case where $\pi$ is cuspidal.
Let $(\pi,V)$ be an irreducible representation of $G_n$ such that $L(s,\pi)=1$. 
Note that there exist a partition
$n=n_1 + \cdots + n_r$ of $n$, and
cuspidal representations $\pi_1, \ldots, \pi_r$
of $G_{n_1}, \ldots, G_{n_r}$, respectively such that
the following conditions are satisfied:
\begin{itemize}
\item 
For $i=1,\ldots,r$, we have $L(s,\pi_i)=1$; 
\item 
$\pi$ appears as a subquotient of the
parabolic induction $\pi_1 \times \cdots \times \pi_r$; 
\item 
we have $|\lambda_\pi| = |\lambda_{\pi_1}|
+ \cdots + |\lambda_{\pi_r}|$.
\end{itemize}
Then 
by the same argument as in the proof of Theorem \ref{conj1-2} for $\pi = Z(\mm)$ of type $\chi$
in Section \ref{sec.pf_conj_unip}, 
we can reduce the claim for $\pi$ to the ones for $\pi_i$ for $i=1,\dots,r$, 
i.e., the case where $\pi$ is cuspidal.
\par

To prove the claim for cuspidal $\pi$, 
we consider certain Hecke operators. 
Let $X_\lambda \subset M_n(\oo)$ denote 
the subset of matrices $A = (a_{i,j}) \in M_n(\oo)$ 
such that $a_{i,j} \equiv \delta_{i,j} \bmod \pp^{\lambda_i}$ for $1 \leq i,j \leq n$.
Then 
\begin{itemize}
\item
$X_\lambda$ contains $\K_{n,\lambda}$; 
\item
$X_\lambda$ is closed under the multiplication of matrices; and 
\item
$X_\lambda$ is bi-invariant under the action of $\K_{n,\lambda}$.
\end{itemize}
We let $\HH_\lambda$ denote the complex vector
space of $\C$-valued compactly supported
bi-$\K_{n,\lambda}$-invariant functions on $G_n$
whose supports are contained in $X_\lambda$.
Then $\HH_\lambda$ has a structure of
$\C$-algebra whose multiplication law
is given by the convolution with respect
to the Haar measure on $G_n$ satisfying
$\mathrm{vol}(\K_{n,\lambda}) = 1$. 
The unit element $1$ of $\HH_\lambda$ is equal to
the characteristic function of $\K_{n,\lambda}$.
Let $\af_\lambda \subset \HH_\lambda$ be
the subspace of functions 
whose supports are contained in the
complement $X_\lambda \setminus \K_{n,\lambda}$ of $\K_{n,\lambda}$
in $X_\lambda$.
Then we have $\HH_\lambda = \C \cdot 1 \oplus \af_\lambda$,
and $\af_\lambda$ is a two-sided ideal of $\HH_\lambda$.
\par

Let $(\pi,V)$ be an irreducible representation of $G_n$.
The action of $G_n$ on $V$ induces an action of
$\HH_\lambda$ on $V^{\K_{n,\lambda}}$.
We let
\[
\theta_V \colon \HH_\lambda \to \End_\C(V^{\K_{n,\lambda}})
\]
denote the induced homomorphism of $\C$-algebras.
We set $\HH_{\lambda,V} = \theta_V(\HH_\lambda)$ 
and $\af_{\lambda,V} = \theta_V(\af_\lambda)$.
Then $\HH_{\lambda,V}$ is a finite dimensional
$\C$-algebra and $\af_{\lambda,V}$ is a
two-sided ideal of $\HH_{\lambda,V}$.

\begin{lem} \label{lem:nilp}
Suppose that $\pi$ is cuspidal.
Then any element $T \in \af_{\lambda,V}$ is nilpotent.
\end{lem}
\begin{proof}
Since $V^{\K_{n,\lambda}}$ is finite dimensional, it suffices to
show that, for any $v \in V^{\K_{n,\lambda}}$ and for any
linear form $\wt{v} \colon  V^{\K_{n,\lambda}} \to \C$, we have
$\wt{v}(T^m v) = 0$ for any sufficiently large integer $m$.
\par

Let us choose $\wt{T} \in \af_\lambda$ satisfying $\theta_V(\wt{T}) =T$.
For an integer $m \ge 0$, 
we let $X_\lambda^{\ge m}$ denote 
the subset of matrices $A \in X_\lambda$ satisfying $\det A \in \pp^m$.
We note that $X_\lambda^{\ge 1}$ is equal to
$X_\lambda \setminus \K_{n,\lambda}$.
Since the product of any $m$ matrices in
$X_\lambda^{\ge 1}$ belongs to
$X_\lambda^{\ge m}$,
it follows that the $m$-th power $\wt{T}^m$ 
of $\wt{T}$ is, as a function on $G_n$, 
supported on $X_\lambda^{\ge m} \cap G_n$.
\par

Let $(\wt{\pi}, \wt{V})$ denote the contragredient representation of $(\pi,V)$.
We regard $\wt{v}$ as a vector 
in the $\K_{n,\lambda}$-invariant part $(\wt{V})^{\K_{n,\lambda}}$ of $\wt{V}$.
Since $\pi$ is cuspidal, the matrix coefficient 
\[
f(g) = \langle \pi(g) v, \wt{v} \rangle
\]
of $\pi$ is compactly supported modulo the center $Z_n$ of $G_n$.
Observe that the intersection $G_n \cap \left(
\bigcap_{m \ge 1} Z_n X_\lambda^{\ge m}\right)$ is empty.
This implies that any subset $K$ of $G_n$ which is
compact modulo $Z_n$ does not intersect
$X_\lambda^{\ge m}$ for any sufficiently large $m$.
Thus, the function $f(g)$ is identically zero on $X_\lambda^{\ge m}$
for any sufficiently large $m$, which implies that
$\wt{v}(T^m v) = 0$ as desired.
\end{proof}

\begin{proof}[Proof for Theorem \ref{conj1-2} when $L(s,\pi) = 1$.]
As we have remarked above, 
we may and will assume that $(\pi,V)$ is cuspidal.
\par

Let us assume that $V^{\K_{n,\lambda}} \neq 0$.
Since $V^{\K_{n,\lambda}}$ is finite dimensional,
one can take a minimal non-zero left $\HH_{\lambda,V}$-submodule $W$ of $V^{\K_{n,\lambda}}$.
Lemma \ref{lem:nilp} implies that
$\af_{\lambda,V}$ is contained in the Jacobson radical of
$\HH_{\lambda,V}$. Hence any element of $\af_{\lambda,V}$
acts as zero on $W$.
\par

Let us choose non-zero vectors $w \in W$ and
$\wt{w} \in (\wt{V})^{\K_{n,\lambda}}$ such that
$\langle w,\wt{w} \rangle \neq 0$.
Let $f(g)$ denote the matrix coefficient of $\pi$
defined as
\[
f(g) = \langle \pi(g) w, \wt{w} \rangle.
\]
Let $\Phi$ denote the characteristic function
of $X_{\lambda}$. Let us consider the zeta integral
\[
Z(\Phi,s,f) = \int_{G_n} \Phi(g) |\det g|^s f(g) dg
\]
of \cite{GJ}. 
By definition, we have
\[
Z(\Phi,s,f) = \sum_{m \ge 0} I_m q^{-ms},
\]
where
\begin{align*}
I_m 
& = \int_{X_\lambda^{\ge m} \setminus X_\lambda^{\ge m+1}} f(g) dg \\
& = \left\langle 
\int_{X_\lambda^{\ge m} \setminus X_\lambda^{\ge m+1}} 
\pi(g) w dg, \wt{w}\right\rangle
\end{align*}
as a formal power series in $q^{-s}$.
Since $\af_\lambda$ annihilates $w$,
it follows that $I_m =0$ for $m \ge 1$. Hence
\[
Z(\Phi,s,f) = I_0
= \left\langle 
\int_{\K_{n,\lambda}} \pi(k) w dk, \wt{w}\right\rangle
= \left(\int_{\K_{n,\lambda}} dk \right)
\langle w,\wt{w} \rangle
\]
is a non-zero constant.
\par

Let us consider the Fourier transform 
\[
\wh{\Phi}(x) = \int_{M_n(F)}
\Phi(y) \psi(xy) dy
\]
of $\Phi$ with respect to $\psi$, 
where $dy$ is the Haar measure on $M_n(F)$ which is self-dual with respect to $\psi$.
Then $\wh{\Phi}$ is supported on the subset $Y_\lambda \subset M_n(F)$ of
matrices $B = (b_{i,j}) \in M_n(F)$ 
such that $b_{i,j} \in \pp^{-\lambda_j}$ for $1 \leq i,j \leq n$.
We set $\check{f} (g) = f(g^{-1})$. 
Note that $\check{f}$ is a matrix coefficient of $(\wt{\pi},\wt{V})$.
Since $\det B \in \pp^{-|\lambda|}$ for
any $B \in Y_\lambda$, it follows that
the zeta integral $Z(\wh{\Phi},s,\check{f})$ is,
as a formal power series in $q^{-s}$, belongs to
$q^{|\lambda| s} \C[[q^{-s}]]$.
\par

By our assumption, we have $L(s,\pi) = L(s,\wt{\pi}) =1$.
Hence it follows from the local functional equation that we have
\begin{equation} \label{eq:localFE}
Z\left(\wh{\Phi},1-s+\frac{n-1}{2},\check{f}\right)
= \ep(s,\pi,\psi) Z\left(\Phi,s+\frac{n-1}{2},f\right)
\end{equation}
where $\ep(s,\pi,\psi)$ denotes the $\ep$-factor 
of $\pi$. It is known that 
$\ep(s,\pi,\psi) = cq^{-|\lambda_\pi| s}$ for some nonzero constant $c$.
Since the left-hand side is in $q^{-|\lambda| s} \C[[q^{s}]]$, 
we see that $|\lambda| \geq |\lambda_\pi|$. 
This proves Theorem \ref{conj1-2} for $\pi$.
\end{proof}

As explained in Section \ref{sec.reduction}, 
Proposition \ref{mackey} and results in Section \ref{sec.pf_conj_unip} and this subsection
complete Theorem \ref{conj1-2} in all cases. 

\subsection{Proof of Theorem \ref{main1}: reduction to Speh representations}
In this subsection, we prove Lemma \ref{thm.reduction}. 
By this lemma, 
Theorem \ref{main1} for $\pi$ with $L(s,\pi) = 1$
is reduced to the case where $\pi = Z(\Delta)$. 

\begin{lem}\label{thm.reduction}
Let $\pi = Z(\mm) \in \Irr(G_n)$ be such that $L(s,\pi) = 1$. 
Write $\mm = \Delta_1 + \dots + \Delta_r$.
Assume that 
\[
\dim (Z(\Delta_i)^{\K_{n_i, \lambda_i}}) = 
\left\{
\begin{aligned}
&1 \iif \lambda_i = \lambda_{\Delta_i}, \\
&0 \iif \lambda_i < \lambda_{\Delta_i}
\end{aligned}
\right. 
\]
for $1 \leq i \leq r$, where $n_i$ is such that $Z(\Delta_i) \in \Irr(G_{n_i})$. 
Then we have 
\[
\dim (\pi^{\K_{n, \lambda}}) = 
\left\{
\begin{aligned}
&1 \iif \lambda = \lambda_{\pi}, \\
&0 \iif \lambda < \lambda_{\pi}.
\end{aligned}
\right. 
\]
\end{lem}
\begin{proof}
Set $\Pi = Z(\Delta_1) \times \dots \times Z(\Delta_r)$.
First, we claim that 
\[
\dim (\Pi^{\K_{n, \lambda}}) = 
\left\{
\begin{aligned}
&1 \iif \lambda = \lambda_{\pi}, \\
&0 \iif \lambda < \lambda_{\pi}.
\end{aligned}
\right. 
\]
Write $\lambda_\pi = (\lambda_1, \dots, \lambda_n)$, 
and consider $M = \oplus_{i=1}^n \oo/\pp^{\lambda_n}$. 
Then $\K_{n,\lambda_\pi}$ is conjugate to $\K_{n,[M]}$.  
By Proposition \ref{mackey}, we have 
\[
\Pi^{\K_{n,\lambda_\pi}} 
\cong \bigoplus_{\Fil_\bullet M} 
Z(\Delta_1)^{\K_{n_1, [\Gr^\Fil_1 M]}} \otimes \dots \otimes Z(\Delta_r)^{\K_{n_r, [\Gr^\Fil_r M]}}, 
\]
where $\Fil_\bullet M$ runs over the set of $\bn$-admissible filtrations
with $\bn = (n_1,\dots,n_r)$.
Since $\lambda_{\pi} = \lambda_{\Delta_1} + \dots + \lambda_{\Delta_r}$, 
by Corollary \ref{fil0}, there exists a unique $\bn$-admissible filtration $\Fil^0_\bullet M$ 
such that $\seq_n([\Gr^{\Fil^0}_i M]) = \lambda_{\Delta_i}$ for $1 \leq i \leq r$.
Moreover, for any other filtration $\Fil_\bullet M$, it holds that 
$\seq_n([\Gr^{\Fil}_i M]) < \lambda_{\Delta_i}$ for some $1 \leq i \leq r$. 
Hence by our assumption, we have 
$Z(\Delta_1)^{\K_{n_1, [\Gr^\Fil_1 M]}} \otimes \dots \otimes Z(\Delta_r)^{\K_{n_r, [\Gr^\Fil_r M]}} 
= 0$, 
and 
\[
\dim\left( \Pi^{\K_{n,\lambda_\pi}} \right) 
= \dim\left(
Z(\Delta_1)^{\K_{n_1, \lambda_{\Delta_1}}}
\otimes \dots \otimes 
Z(\Delta_r)^{\K_{n_r, \lambda_{\Delta_r}}}
\right)
= 1. 
\]
Conversely, suppose that $[M] \in |\CC^n|$ satisfies $\Pi^{\K_{n,[M]}} \not= 0$. 
Then by Propositions \ref{mackey}, \ref{prop:convexity} and by our assumption, we have
\[
\seq_n([M]) \geq \lambda_{\Delta_1} + \dots + \lambda_{\Delta_r} = \lambda_\pi. 
\]
In other words, if $\lambda < \lambda_\pi$, then $\Pi^{\K_{n,\lambda}} = 0$. 
Hence we obtain the claim. 
\par

In particular, since $\pi$ is a subquotient of $\Pi$, 
we have $\pi^{\K_{n,\lambda}} = 0$ for $\lambda < \lambda_\pi$.
\par

We show $\dim(\pi^{\K_{n,\lambda}}) = 1$ by induction on the number $t_\pi$ 
of pairs of linked segments in $\{\Delta_1, \dots, \Delta_r\}$. 
If $t_\pi = 0$, then by \cite[4.2 Theorem]{Z}, $\Pi$ is irreducible so that $\pi = \Pi$. 
In this case, the assertion is obtained above. 
\par

Now assume that $t_\pi > 0$. 
By \cite[7.1 Theorem]{Z}, 
if $\pi' = Z(\mm') \in \Irr(G_n)$ is an irreducible constituent of $\Pi$, 
then the multisegment $\mm'$ is obtained from $\mm$
by a chain of elementary operations.
In particular, if $\pi' \not\cong \pi$, we have $t_{\pi'} < t_{\pi}$.
Moreover, since $L(s,\pi) = 1$, we see that $\lambda_{\pi'} > \lambda_{\pi}$. 
By the inductive hypothesis, we have $\pi'^{\K_{n,\lambda_\pi}} = 0$. 
Therefore, we have $\Pi^{\K_{n,\lambda_\pi}} = \pi^{\K_{n,\lambda_\pi}}$ 
since $\pi$ appears in the irreducible constituents of $\Pi$ with multiplicity one. 
It follows from the above claim that $\pi^{\K_{n,\lambda_\pi}}$ is one-dimensional. 
This completes the proof.
\end{proof}

Note that Theorem \ref{main1} for $\pi$ is equivalent 
to the one for its unramified twist $\pi|\cdot|^c$. 
Therefore, we may assume that $\pi$ has a unitary central character.
In Section \ref{s.ess_speh} below, we
will prove Theorem \ref{main1} for $\pi = Z(\Delta)$ with 
a unitary central character such that $L(s,\pi) = 1$. 
The proof of this case is rather similar to the generic case in \cite{JPSS2}.
To carry out the proof, 
we will establish the theory of Rankin--Selberg integrals for $Z(\Delta)$ in Section \ref{s.RS}. 

\begin{rem}
We note that 
Lemma \ref{thm.reduction} does not work for $\pi$ with $L(s,\pi) \not= 1$ 
since the equality $\lambda_{\pi} = \lambda_{\Delta_1} + \dots + \lambda_{\Delta_r}$ 
does not hold in general. 
It is one of the two reasons why we should treat the case where $L(s,\pi) = 1$ and the other case
separately. 
The other reason will be explained in Remark \ref{rem.gcd} below. 
\end{rem}

\section{Rankin--Selberg integrals for Speh representations}\label{s.RS}
In \cite{JPSS2}, 
Jacquet--Piatetskii-Shapiro--Shalika proved
Theorem \ref{main1} for $\pi$ generic. 
The ingredient they used is the Rankin--Selberg integrals \cite{JPSS}, 
which express the $L$-factors 
of the products of two generic representations of $G_n$ and $G_{n-1}$. 
(They also have expressions for products of representations of groups of other ranks, 
but the one used for the study of local newforms is the one mentioned above.)  
\par

In \cite{LM}, Lapid and Mao introduced the Rankin--Selberg integrals 
for the products of Speh representations in the equal rank case. 
To prove Theorem \ref{main1} for Speh representations in the next section,  
we introduce the Rankin--Selberg integrals 
for the product of Speh representations in the case $G_{nm} \times G_{(n-1)m}$. 

\subsection{Subgroups of $\GL_{nm}$(F)}
Fix positive integers $m$ and $n$. 
In this subsection, we fix notations for some subgroups of $\GL_{nm}(F)$. 
\par

Set $G = G_{nm} = \GL_{nm}(F)$ and $K = \GL_{nm}(\oo)$. 
Let $B = TN$ be the Borel subgroup of $G$ consisting of upper triangular matrices, 
where $T$ is the diagonal torus. 
\par

We write an element of $G$ as $g = (g_{i,j})_{1 \leq i,j \leq m}$ 
with $g_{i,j} \in M_n(F)$.
Define 
\begin{itemize}
\item
$L$ to be  the subgroup of $G$ consisting of block diagonal matrices, 
i.e., $g = (g_{i,j})_{1 \leq i,j \leq m} \in G$ with $g_{i,j} = 0$ for $i \not= j$; 

\item
$U$ to be the subgroup of $G$ consisting of block upper unipotent matrices, 
i.e., $g = (g_{i,j})_{1 \leq i,j \leq m} \in G$ 
with $g_{i,i} = \1_n$ for $1 \leq i \leq m$ and $g_{i,j} = 0$ for $i > j$; 

\item
$S$ to be the subgroup of $G$ consisting of 
$g = (g_{i,j})_{1 \leq i,j \leq m} \in G$ such that 
each $g_{i,j}$ is a diagonal matrix; 

\item
$V$ to be the subgroup of $G$ consisting of 
$g = (g_{i,j})_{1 \leq i,j \leq m} \in G$ such that 
each $g_{i,j} - \delta_{i,j}\1_n$ is a strictly upper triangular matrix; 
 
\item
$D$ to be the subgroup of $G$ consisting of 
$g = (g_{i,j})_{1 \leq i,j \leq m} \in G$ such that 
each $g_{i,j}$ is of the form 
\[
g_{i,j} = \begin{pmatrix}
g'_{i,j} & u_{i,j} \\ 0 & \delta_{i,j}
\end{pmatrix}
\]
for some $g'_{i,j} \in M_{n-1}(F)$ and $u_{i,j} \in F^{n-1}$. 
\end{itemize}
Then $P = LU$ is the standard parabolic subgroup 
with $L \cong G_n \times \dots \times G_n$ ($m$-times) as its Levi subgroup, 
and 
$Q = SV$ is a non-standard parabolic subgroup 
with $S \cong G_m \times \dots \times G_m$ ($n$-times) as its Levi subgroup. 
\par

We set $G' = G_{(n-1)m}$. 
We denote analogous subgroups by taking $'$, e.g., 
$K' = \GL_{(n-1)m}(\oo)$, $P' = L'U'$, $Q' = S'V'$ and so on.
Define an embedding $\iota \colon G' \hookrightarrow G$ by 
\[
\iota(g') = \left( \begin{pmatrix}
g'_{i,j} & 0 \\ 0 & \delta_{i,j}
\end{pmatrix} \right)_{1 \leq i,j \leq m}, 
\]
where we write $g' = (g'_{i,j})_{1 \leq i,j \leq m}$ with $g'_{i,j} \in M_{n-1}(F)$. 
Sometimes, we identify $G'$ with the image of $\iota$. 
Note that $G'$ is contained in $D$. 
\par

For example, when $n = 3$ and $m = 2$, the subgroups above are as follows: 
\begin{align*}
L = \left(
\begin{array}{ccc|ccc}
*&*&*&&& \\
*&*&*&&& \\
*&*&*&&& \\
\hline
&&&*&*&* \\
&&&*&*&* \\
&&&*&*&* 
\end{array}
\right), 
&\quad
U = \left(
\begin{array}{ccc|ccc}
1&&&*&*&* \\
&1&&*&*&* \\
&&1&*&*&* \\
\hline
&&&1&& \\
&&&&1& \\
&&&&&1 
\end{array}
\right), 
\\
S = \left(
\begin{array}{ccc|ccc}
*&&&*&& \\
&*&&&*& \\
&&*&&&* \\
\hline
*&&&*&& \\
&*&&&*& \\
&&*&&&* 
\end{array}
\right), 
&\quad
V = \left(
\begin{array}{ccc|ccc}
1&*&*&&*&* \\
&1&*&&&* \\
&&1&&& \\
\hline
&*&*&1&*&* \\
&&*&&1&* \\
&&&&&1 
\end{array}
\right), 
\\
D = \left(
\begin{array}{ccc|ccc}
*&*&*&*&*&* \\
*&*&*&*&*&* \\
&&1&&& \\
\hline
*&*&*&*&*&* \\
*&*&*&*&*&* \\
&&&&&1 
\end{array}
\right), 
&\quad
G' = \left(
\begin{array}{ccc|ccc}
*&*&&*&*& \\
*&*&&*&*& \\
&&1&&& \\
\hline
*&*&&*&*& \\
*&*&&*&*& \\
&&&&&1 
\end{array}
\right).
\end{align*}
\par

It is easy to see the following. 
\begin{lem}\label{group}
\begin{enumerate}
\item
$D = VG'$ and $G' \cap V = V'$ so that $V \bs D \cong V' \bs G'$.
\item
$N \cap D = (N \cap V) N'$ and $(N \cap V) \cap N' = N' \cap V'$ 
so that $(N \cap V) \bs (N \cap D) \cong (N' \cap V') \bs N'$. 
\end{enumerate}
\end{lem}
\begin{proof}
Omitted.
\end{proof}

\subsection{Two models of Speh representations}
We introduce the Zelevinsky model and the Shalika model of 
a Speh representation.   
For the detail of these models and the relation between these models, see \cite[Section 3]{LM}. 
\par

We define a function $\Psi$ of $G = G_{nm}$ by 
\[
\Psi(g) = \psi \left(
\sum_{\substack{1 \leq i < nm \\ n \nmid i}}g_{i,i+1}
\right).
\]
We denote the restriction of $\Psi$ to $N$ (\resp $V$) by the same symbol $\Psi$, 
which is a character of $N$ (\resp $V$).
\par

Let $\pi$ be an irreducible tempered representation of $G_n$. 
Then the parabolically induced representation
\[
\pi|\cdot|^{-\half{m-1}} \times \pi|\cdot|^{-\half{m-3}} \times \dots \times \pi|\cdot|^{\half{m-1}}
\]
of $G$ has a unique irreducible subrepresentation $\Sp(\pi, m)$. 
We call $\Sp(\pi, m)$ a \emph{Speh representation}. 
Note that if $\pi = \rho$ is cuspidal, then $\Sp(\rho,m) = Z([-\half{m-1},\half{m-1}]_\rho)$. 
\par

From now on, we set $\sigma = \Sp(\pi, m)$ 
for some irreducible tempered representation $\pi$ of $G_n$. 
By \cite[8.3]{Z}, we know that
\[
\Hom_{G}(\sigma, \Ind_N^G(\Psi))
\]
is one-dimensional. 
Following \cite[Section 3.1]{LM}, 
we write $\WW^{\psi}_\ze(\sigma)$ for the image of a nonzero element, 
and call it the \emph{Zelevinsky model} of $\sigma$.
\par

In the case $m=1$, 
the Zelevinsky model $\WW^\psi(\pi) = \WW^\psi_\ze(\pi)$ is what is known
as the \emph{Whittaker model} of $\pi$. 
Note that the character $\Psi$ is a generic character of $N$ in this case,
and the one-dimensionality above implies that 
every tempered representation $\pi$ of $G_n$ is generic. 
\par

As explained in \cite[Section 3.1]{LM}, for any $W \in \WW^{\psi}_\ze(\sigma)$, we have
\[
W|_L \in \WW^{\psi}(\pi|\cdot|^{\half{(m-1)(n-1)}}) 
\otimes \WW^{\psi}(\pi|\cdot|^{\half{(m-3)(n-1)}}) 
\otimes \dots \otimes 
\WW^{\psi}(\pi|\cdot|^{-\half{(m-1)(n-1)}}).
\]
\par

By \cite{MW}, we know that
\[
\Hom_{G}(\sigma, \Ind_V^G(\Psi))
\]
is also one-dimensional. 
Following \cite[Section 3.1]{LM}, 
we write $\WW^{\psi}_\sh(\sigma)$ for the image of a nonzero element, 
and call it the \emph{Shalika model} of $\sigma$.
As explained in \cite[Section 3.1]{LM}, the usage of this terminology may not be a common one. 
\par

We recall a theorem of Lapid and Mao.    
\begin{thm}[{\cite[Theorem 4.3]{LM}}]
\label{inner}
For $W_1, W_2 \in \WW^\psi_\sh(\sigma)$, 
the integral 
\[
\BB(W_1,W_2,s) = \int_{V \bs D}W_1(g)\overline{W_2(g)}|\det g|^s dg
\]
converges for $\Re(s) > -1$, 
and admits meromorphic continuation to the complex plane. 
Moreover, $(W_1,W_2) \mapsto \BB(W_1,W_2,0)$ 
is a $G$-invariant inner product on $\WW^\psi_\sh(\sigma)$.
\end{thm}
\begin{proof}
See loc.~cit. 
See also \cite[Propositions 4.1, 6.2]{LM}.
\end{proof}

Note that $\WW^\psi_\ze(\sigma)$ and $\WW^\psi_\sh(\sigma)$ are isomorphic to each other
since both are isomorphic to $\sigma$. 
We can give isomorphisms explicitly as follows.
\begin{prop}[{\cite[Lemmas 3.8, 3.11]{LM}}]\label{transition}
Let $W_\ze \in \WW^\psi_\ze(\sigma)$ and $W_\sh \in \WW^\psi_\sh(\sigma)$. 
Then $W_\ze$ (\resp $W_\sh$) is compactly supported 
on $(V \cap N) \bs V$ (\resp $(N \cap V) \bs (N \cap D)$). 
Moreover, 
an isomorphism $\TT = \TT^\psi \colon \WW^\psi_\ze(\sigma) \xrightarrow{\sim} \WW^\psi_\sh(\sigma)$
is given by the integral 
\[
\TT W_\ze(g) = \int_{(V \cap N) \bs V}W_\ze(ug) \Psi(u)^{-1}du. 
\]
The inverse of $\TT$ is given by the integral
\[
\TT^{-1} W_\sh(g) = \int_{(N \cap V) \bs (N \cap D)}W_\sh(ug) \Psi(u)^{-1}du. 
\]
\end{prop}
\begin{proof}
See loc.~cit.
\end{proof}

\subsection{Rankin--Selberg integrals in the Zelevinsky models}
For irreducible tempered representations $\pi$ and $\pi'$ of $G_n$ and $G_{n-1}$, respectively, 
we have Speh representations 
$\sigma = \Sp(\pi,m) \in \Irr(G)$ and $\sigma' = \Sp(\pi',m) \in \Irr(G')$. 
For $W \in W^{\psi}_\ze(\sigma)$, $W' \in W^{\psi^{-1}}_\ze(\sigma')$ and $s \in \C$, 
consider the integral 
\[
I_m(s,W,W') = 
\int_{N' \bs G'} W(\iota(g)) W'(g) |\det g|^{s-\half{m}} dg.
\]
We call this \emph{the Rankin--Selberg integral in the Zelevinsky models}.

\begin{lem}\label{iwasawa}
Formally, $I_m(s,W,W')$ is equal to
\[
\int_{P' \bs G'} 
\left( \int_{(N' \cap L') \bs L'}W(\iota(lg)) W'(lg) |\det l|^{s-\half{m}} \delta_{P'}^{-1}(l)dl \right)
|\det g|^{s-\half{m}} dg.
\]
\end{lem}
\begin{proof}
This follows from a well-known integral formula. 
\end{proof}

When $m = 1$, several properties of $I_1(s,W,W')$ were obtained in \cite{JPSS}. 
The following is a generalization of \cite[(2.7) Theorem]{JPSS}, 
whose proof is analogous to that of \cite[Theorem 5.1]{LM}.

\begin{thm}\label{RS}
Let $\pi$ and $\pi'$ be irreducible tempered representations of $G_n$ and $G_{n-1}$, respectively. 
We denote the central character of $\pi'$ by $\omega_{\pi'}$.
\begin{enumerate}
\item
The integral $I_m(s,W,W')$ is absolutely convergent for $\Re(s) \gg 0$. 

\item
The function 
\[
\left(\prod_{i=1}^m L(s-m+i, \pi \times \pi')\right)^{-1} I_m(s,W,W')
\]
is in $\C[q^{-s}, q^{s}]$. 
In particular, it is entire. 

\item
The functional equation 
\[
I_m(m-s,\wt{W},\wt{W}') 
= \omega_{\pi'}(-1)^{(n-1)m} \left( \prod_{i=1}^m \gamma(s-m+i, \pi \times \pi', \psi) \right)
I_m(s,W,W')
\]
holds, where $\wt{W}(g) = W(w_{nm} {}^tg^{-1} w'_n)$ and $\wt{W}'(g') = W'(w_{(n-1)m} {}^tg'^{-1} w'_{n-1})$
with 
\[
w_{nm} = 
\begin{pmatrix}
&&1 \\
&\iddots& \\
1&&
\end{pmatrix},
\quad
w'_n = \begin{pmatrix}
&&\1_{n} \\
&\iddots& \\
\1_{n} &&
\end{pmatrix} \in G.
\]
Here $\gamma(s, \pi \times \pi', \psi)$ is the gamma factor defined by 
\[
\gamma(s, \pi \times \pi', \psi) 
= \ep(s, \pi \times \pi', \psi)\frac{L(1-s,\wt{\pi} \times \wt{\pi}')}{L(s,\pi \times \pi')}.
\]
\end{enumerate}
\end{thm}
\begin{proof}
When $m=1$, the assertions are \cite[(2.7) Theorem]{JPSS}.
\par

Note that $\delta_{P'}(l) = \prod_{i=1}^{m}|\det l_i|^{(m+1-2i)(n-1)}$ 
for $l = \diag(l_1, \dots, l_m) \in L'$.
Moreover,
\begin{align*}
\prod_{i=1}^m|\det l_i|^{\half{(m+1-2i)(1-n)}}W(\iota(lg)) &\in \WW^{\psi}(\pi)^{\otimes m}
\quad\text{and}\\
\prod_{i=1}^m|\det l_i|^{\half{(m+1-2i)(2-n)}}W'(lg) &\in \WW^{\psi^{-1}}(\pi')^{\otimes m}
\end{align*}
for fixed $g \in G'$. 
It follows that 
the inner integral of Lemma \ref{iwasawa} is of the form
\[
\sum_{\alpha,\beta} \prod_{i=1}^m I_1(s-m+i, W_{i,\alpha}, W'_{i,\beta})
\] 
for some $W_{i,\alpha} \in \WW^{\psi}(\pi)$ and $W'_{i,\beta} \in \WW^{\psi^{-1}}(\pi')$ 
(depending on $g$).
Hence we obtain the assertions (1) and (2). 
\par

We prove the assertion (3). 
For $g \in G'$ and $l = \diag(l_1, \dots, l_m) \in L'$ with $l_i \in \GL_{n-1}(F)$, 
we note that 
\begin{align*}
w'_n{}^t\iota(g)^{-1}w'_n &= \iota \left(w'_{n-1}{}^tg^{-1}w'_{n-1}\right), 
\\
w_{nm}{}^t\iota(l)^{-1} w'_n &= 
\diag\left(
w_n \begin{pmatrix}
{}^tl_m^{-1} & 0 \\ 0 & 1
\end{pmatrix}, 
\dots, 
w_n \begin{pmatrix}
{}^tl_1^{-1} & 0 \\ 0 & 1
\end{pmatrix}
\right).
\end{align*}
Hence we have
\begin{align*}
\wt{W}(\iota(lg)) 
&= W(w_{nm}{}^t\iota(l)^{-1} w'_n \cdot w'_n {}^t\iota(g)^{-1}w'_n)
\\&= W\left(
\diag\left(
w_n \begin{pmatrix}
{}^tl_m^{-1} & 0 \\ 0 & 1
\end{pmatrix}, 
\dots, 
w_n \begin{pmatrix}
{}^tl_1^{-1} & 0 \\ 0 & 1
\end{pmatrix}
\right)
\iota(w'_{n-1}{}^tg^{-1}w'_{n-1})
\right). 
\end{align*}
Similarly, we have
\begin{align*}
\wt{W}'(lg) &= W'(w_{(n-1)m}{}^tl^{-1}w'_{n-1} \cdot w'_{n-1}{}^tg^{-1} w'_{n-1})
\\&= 
W'(\diag(w_{n-1}{}^tl_m^{-1}, \dots, w_{n-1}{}^tl_1^{-1}) \cdot w'_{n-1}{}^tg^{-1} w'_{n-1}).
\end{align*}
Moreover, the map $g \mapsto \theta(g) \coloneqq w'_{n-1}{}^tg^{-1}w'_{n-1}$ 
is a homeomorphism on $P' \bs G'$ 
such that $d\theta(g) = dg$ and $|\det \theta(g)| = |\det g|^{-1}$.
Hence
\begin{align*}
&I_{m}(m-s, \wt{W}, \wt{W}')
\\&=
\int_{P' \bs G'} 
\left( 
\int_{(N' \cap L') \bs L'}\wt{W}(\iota(lg)) \wt{W}'(lg) |\det l|^{-(s-\half{m})} \delta_{P'}^{-1}(l)dl \right)
|\det g|^{-(s-\half{m})} dg
\\&=
\int_{P' \bs G'} 
\left( \int_{(N' \cap L') \bs L'}\wt{W}(\iota(lg)) \wt{W}'(lg) 
\prod_{i=1}^{m}|\det l_i|^{\half{(m+1-2i)(3-2n)}-s+i-\half{1}}
dl \right)
|\det g|^{-(s-\half{m})} dg
\\&= 
\omega_{\pi'}(-1)^{(n-1)m} \left( \prod_{i=1}^m \gamma(s-m+i, \pi \times \pi', \psi) \right)
\\&\times 
\int_{P' \bs G'} 
\left( \int_{(N' \cap L') \bs L'}W(\iota(lg)) W'(lg)
\prod_{i=1}^{m}|\det l_i|^{\half{(m+1-2i)(3-2n)}+s-m+i-\half{1}}
dl \right)
|\det g|^{s-\half{m}} dg
\\&= 
\omega_{\pi'}(-1)^{(n-1)m} \left( \prod_{i=1}^m \gamma(s-m+i, \pi \times \pi', \psi) \right)
I_m(s,W,W').
\end{align*}
Here, in the third equation, 
we made the change of variables $l_i \mapsto l_{m+1-i}$ and $g \mapsto \theta(g)$. 
This completes the proof.
\end{proof}

\begin{lem}\label{=1}
For any $W' \in W^{\psi^{-1}}_\ze(\sigma')$ with $W'(\1_{(n-1)m}) \not= 0$, 
there exists $W \in W^{\psi}_\ze(\sigma)$ such that $I_m(s,W,W') = 1$ for all $s \in \C$.
\end{lem}
\begin{proof}
By \cite[Corollary 3.15]{LM}, 
the space $\{W|_D \;|\; W \in W^{\psi}_\ze(\sigma)\}$ 
contains the compact induction $\ind_{N \cap D}^{D}(\Psi)$. 
Hence the assertion follows by taking $W \in W^{\psi}_\ze(\sigma)$ 
such that $W|_D$ is supported on $(N \cap D)\Omega$ 
for a small neighborhood $\Omega$ of $\1_{nm} \in D$.
\end{proof}

\begin{prop}\label{gcd}
The $\C$-span of the integrals $I_m(s,W,W')$ 
for $W \in W^{\psi}(\sigma)$ and $W' \in W^{\psi^{-1}}(\sigma')$ 
is a fractional ideal of $\C[q^{-s},q^{s}]$,
which is generated by $P_m(q^{-s})^{-1}$ for some $P_m(X) \in \C[X]$ with $P_m(0) = 1$. 
Moreover, 
$P_1(q^{-s}) = L(s, \pi \times \pi')^{-1}$, and 
$P_m(X)$ divides $\prod_{i=1}^mP_1(q^{m-i}X)$. 
\end{prop}
\begin{proof}
Note that 
\[
I_m(s,\iota(h)W,hW') = |\det h|^{-(s-\half{m})}I_m(s,W,W')
\]
for $h \in G'$, 
where $(\iota(h)W)(g) = W(g\iota(h))$ and $(hW')(g') = W'(g'h)$.
Hence the $\C$-span of the integrals $I_m(s,W,W')$ is a fractional ideal of $\C[q^{-s},q^{s}]$. 
The other assertions follow from Lemma \ref{=1} and Theorem \ref{RS} (2). 
\end{proof}

\begin{rem}\label{rem.gcd}
One might expect that $P_m(X) = \prod_{i=1}^mP_1(q^{m-i}X)$, 
but we do not know if this holds in general. 
This is a reason why 
we cannot prove Theorem \ref{essential} below for $\sigma = \Sp(\pi,m)$ 
when $L(s,\pi) \not= 1$ by a method similar to that in \cite{JPSS2}.
However, as an application of Theorem \ref{main1}, 
we will prove the equation $P_m(X) = \prod_{i=1}^mP_1(q^{m-i}X)$
when $\pi'$ is unramified (see Theorem \ref{essential} below). 
\end{rem}

\subsection{Rankin--Selberg integrals in the Shalika models}
Now we translate the results for the Zelevinsky models obtained in the previous subsection 
to those for the Shalika models.
\par

Recall that $\sigma = \Sp(\pi,m) \in \Irr(G)$ and $\sigma' = \Sp(\pi',m) \in \Irr(G')$. 
For $W_\sh \in W^{\psi}_\sh(\sigma)$, $W'_\sh \in W^{\psi^{-1}}_\sh(\sigma')$ and $s \in \C$, 
consider the integral 
\[
Z_m(s,W_\sh,W'_\sh) = 
\int_{V' \bs G'} W_\sh(\iota(g)) W'_\sh(g) |\det g|^{s-\half{m}} dg. 
\]
We call this \emph{the Rankin--Selberg integral in the Shalika models}. 

\begin{prop}\label{I=Z}
If $W_\sh = \TT^{\psi}W_\ze$ and $W'_\sh = \TT^{\psi^{-1}}W'_\ze$, we have
\[
Z_m(s,W_\sh,W'_\sh) = I_m(s,W_\ze,W'_\ze). 
\]
\end{prop}
\begin{proof}
By Lemma \ref{group} and Proposition \ref{transition}, we have
\begin{align*}
Z_m(s,W_\sh,W'_\sh) 
&=
\int_{V' \bs G'} W_\sh(\iota(g)) 
\left( \int_{(V' \cap N') \bs V'} W'_\ze(ug) \Psi(u)du \right) |\det g|^{s-\half{m}} dg
\\&=
\int_{V' \bs G'} \int_{(V' \cap N') \bs V'}
W_\sh(\iota(ug)) W'_\ze(ug) |\det (ug)|^{s-\half{m}} dudg
\\&= 
\int_{(V' \cap N') \bs G'}
W_\sh(\iota(g)) W'_\ze(g) |\det g|^{s-\half{m}} dg
\\&=
\int_{N' \bs G'}\int_{(V' \cap N') \bs N'}
W_\sh(\iota(ug)) W'_\ze(ug) |\det (ug)|^{s-\half{m}} dudg
\\&=
\int_{N' \bs G'}
\left(\int_{(N \cap V) \bs (N \cap D)}
W_\sh(\iota(ug)) \Psi(u)^{-1} du
\right)
W'_\ze(g) |\det g|^{s-\half{m}} dg
\\&= I_m(s,W_\ze,W'_\ze).
\end{align*}
This proves the proposition.
\end{proof}

Therefore, assertions similar to those in 
Theorem \ref{RS}, Lemma \ref{=1} and Proposition \ref{gcd} 
hold for $Z_m(s,W_\sh, W'_\sh)$.
Here, we note the following. 
If $W_\ze \in \WW^\psi_\ze(\sigma)$, we define 
$\wt{W}_\ze \in \WW^{\psi^{-1}}_\ze(\wt\sigma)$, 
where $\wt\sigma$ is the contragredient representation of $\sigma$, 
by $\wt{W}_\ze(g) = W_\ze(w_{nm}{}^tg^{-1}w'_{n})$. 
One can easily check that 
\[
\TT^{\psi^{-1}}\wt{W}_\ze(g) = \TT^{\psi}W_\ze(w_{nm}{}^tg^{-1}w'_{n}). 
\]
Hence we define $\wt{W}_\sh \in \WW^{\psi^{-1}}_\sh(\wt\sigma)$ 
for $W_\sh \in \WW^\psi_\sh(\sigma)$
by $\wt{W}_\sh(g) = W_\sh(w_{nm}{}^tg^{-1}w'_{n})$.

\subsection{The case where $\pi'$ is unramified}\label{s.unram}
In the following section, 
we need sharper results when $\pi'$ is unramified.
\par

Let $\pi'$ be an irreducible unramified representation of $G_{n-1}$
with Satake parameter $(x_1, \dots, x_{n-1}) \in (\C^\times)^{n-1}/S_{n-1}$.
Hence $\pi'$ is the unique irreducible unramified constituent of 
\[
I(s_1, \dots, s_{n-1}) = |\cdot|^{s_1} \times \dots \times |\cdot|^{s_{n-1}},
\] 
where $s_j$ is a complex number such that $q^{-s_j} = x_j$.
Since the principal series $I(s_1, \dots, s_{n-1})$ is generic and unramified, 
there exists a unique Whittaker function 
$W^0(x_1, \dots, x_{n-1}) \in \WW^{\psi^{-1}}(I(s_1, \dots, s_{n-1}))$
such that $W^0(k_1; x_1, \dots, x_{n-1}) = 1$ for any $k_1 \in \GL_{n-1}(\oo)$. 
When $\pi'$ is tempered, i.e., $|x_j| = 1$ for any $1 \leq j \leq n-1$, 
then $W^0(x_1, \dots, x_{n-1}) \in \WW^{\psi^{-1}}(\pi')$. 
Note that $W^0(x_1, \dots, x_{n-1})$ is a Hecke eigenfunction 
whose Hecke eigenvalues are uniquely determined by 
$(x_1, \dots, x_{n-1}) \in (\C^\times)^{n-1}/S_{n-1}$. 
\par

Recall that $G' = G_{(n-1)m}$, $K' = G_{(n-1)m}(\oo)$ 
and that $P' = L'U'$ is the standard parabolic subgroup of $G'$ 
with $L' \cong G_{n-1} \times \dots \times G_{n-1}$ ($m$-times). 
Let $\ub{x} = (x_{i,j}) \in M_{m,n-1}(\C)$ with $x_{i,j} \in \C^\times$. 
We can define a function 
\[
W^0_\ze(\ub{x}) \colon G' \rightarrow \C
\]
by 
\[
W^0_\ze(ulk; \ub{x}) 
= \Psi^{-1}(u) 
\delta_{P'}^{\half{1}}(l)
\prod_{i=1}^m 
W^0(l_i; x_{i,1}, \dots, x_{i,n-1})
\]
for $u \in U'$, $l = \diag(l_1, \dots, l_m) \in L'$ and $k \in K'$.
(Here, we note that $\Psi(u) = 1$ for $u \in U'$.)
As in \cite[Lemma 3.8]{LM}, $W^0_\ze(\ub{x})$ is compactly supported on $(V' \cap N') \bs V'$. 
We set
\[
W^0_\sh(g; \ub{x}) = \int_{(V' \cap N') \bs V'}W^0_\ze(ug; \ub{x})\Psi(u)du. 
\]
\par

If $\ub{x} = {}^t(q^{-\half{m-1}}x_{j}, q^{-\half{m-3}}x_{j}, \dots, q^{\half{m-1}}x_{j})_{1 \leq j \leq n-1}$ 
with $|x_j| = 1$ for any $1 \leq j \leq n-1$, 
then $W^0_\ze(\ub{x}) \in \WW^{\psi^{-1}}_\ze(\Sp(\pi',m))$, 
where $\pi'$ is the irreducible unramified representation of $G_{n-1}$
with Satake parameter $(x_1, \dots, x_{n-1})$. 
In general, 
$W^0_\ze(g; \ub{x}) = l(g \cdot f^0)$
for some $l \in \Hom_{N'}(I(s_1, \dots, s_{(n-1)m}), \Psi)$, 
where $s_1, \dots, s_{(n-1)m}$ are complex numbers such that 
\[
\{q^{-s_1}, \dots, q^{-s_{(n-1)m}}\} = \{x_{i,j} \;|\; 1 \leq i \leq m, \; 1 \leq j \leq n-1\}
\]
as multi-sets, 
and $f^0 \in I(s_1, \dots, s_{(n-1)m})^{K'}$.
Note that $W^0_\ze(\ub{x})$ is a Hecke eigenfunction 
whose Hecke eigenvalues are uniquely determined by 
$(s_1, \dots, s_{(n-1)m}) \in \C^{(n-1)m}/S_{(n-1)m}$. 

\begin{lem}\label{fin_dim}
The Hecke eigenspace in $\Ind_{N'}^{G'}(\Psi)^{K'}$ 
with Hecke eigenvalues determined by $(s_1, \dots, s_{(n-1)m})$
is spanned by $W^0_\ze(\ub{x})$ for $\ub{x} = (x_{i,j}) \in M_{m,n-1}(\C)$ such that 
$\{q^{-s_1}, \dots, q^{-s_{(n-1)m}}\} = \{x_{i,j} \;|\; 1 \leq i \leq m, \; 1 \leq j \leq n-1\}$
as multi-sets. 
\end{lem}
\begin{proof}
Since $\Psi$ is trivial on $U' \subset N'$, 
we have a canonical isomorphism
\[
\Hom_{N'}(I(s_1, \dots, s_{(n-1)m}), \Psi)
\cong \Hom_{N' \cap L'}(\Jac_{P'}(I(s_1, \dots, s_{(n-1)m})), \Psi),
\]
where $\Jac_{P'}$ is the unnormalized Jacquet functor along $P' = L'U'$. 
Note that $\Psi|_{N' \cap L'}$ is a generic character. 
Moreover, by the Geometric Lemma of Bernstein--Zelevinsky \cite[2.12]{BZ}, 
the semisimplification of $\Jac_{P'}(I(s_1, \dots, s_{(n-1)m}))$ is equal to
\[
\delta_{P'}^{\half{1}} \otimes 
\left(
\bigoplus_{\ub{x}} I(s_{1,1}, \dots, s_{1,n-1}) \boxtimes \dots \boxtimes I(s_{m,1}, \dots, s_{m,n-1})
\right), 
\]
where $\ub{x} = (x_{i,j})$ runs over $M_{m,n-1}(\C)/(S_{n-1})^m$ such that 
$\{q^{-s_1}, \dots, q^{-s_{(n-1)m}}\} = \{x_{i,j} \;|\; 1 \leq i \leq m, \; 1 \leq j \leq n-1\}$, 
and $s_{i,j}$ is a complex number such that $q^{-s_{i,j}} = x_{i,j}$. 
Hence $\dim \Hom_{N'}(I(s_1, \dots, s_{(n-1)m}), \Psi)$ is less than or equal to 
the number of choices of such $\ub{x}$. 
This proves the claim.
\end{proof}

Let $\pi$ be an irreducible tempered representation of $G_n$, 
and set $\sigma = \Sp(\pi,m)$. 
For $W_\ze \in \WW^{\psi}_\ze(\sigma)$ and $W_\sh \in \WW^{\psi}_\sh(\sigma)$, 
one can consider the integrals 
$I_m(s,W_\ze,W^0_\ze(\ub{x}))$ and $Z_m(s,W_\sh,W^0_\sh(\ub{x}))$
defined by the same integrals in the previous two subsections. 
By the same arguments as in these subsections, we can prove the following theorem. 
We omit the proof of it. 

\begin{thm}\label{RS_unram}
The integrals $I_m(s,W_\ze,W^0_\ze(\ub{x}))$ and $Z_m(s,W_\sh,W^0_\sh(\ub{x}))$ 
have the following properties. 
\begin{enumerate}
\item
The integral $I_m(s,W_\ze,W^0_\ze(\ub{x}))$ is absolutely convergent for $\Re(s) \gg 0$. 

\item
The function 
\[
\left(\prod_{i=1}^m\prod_{j = 1}^{n-1} L\left(s+s_{i,j}-\half{m-1}, \pi \right)\right)^{-1} I_m(s,W_\ze,W^0_\ze(\ub{x}))
\]
is in $\C[q^{-s}, q^{s}]$, 
where $s_{i,j}$ is a complex number such that $q^{-s_{i,j}} = x_{i,j}$.  
In particular, it is entire. 

\item
The functional equation 
\begin{align*}
&I_m(m-s,\wt{W}_\ze,W^0_\ze(\ub{x}^{-1})) 
\\&= 
\left( \prod_{i=1}^m\prod_{j=1}^{n-1} \gamma\left( s+s_{i,j}-\half{m-1}, \pi, \psi \right) \right)
I_m(s,W_\ze,W^0_\ze(\ub{x})), 
\end{align*}
holds, 
where $\wt{W}_\ze(g) = W_\ze(w_{nm} {}^tg^{-1} w'_n)$ and $\ub{x}^{-1} = (x_{i,j}^{-1})$. 

\item
If $W_\sh = \TT^{\psi}W_\ze$, then 
\[
I_m(s,W_\ze,W^0_\ze(\ub{x})) = Z_m(s,W_\sh,W^0_\sh(\ub{x})).
\]

\end{enumerate}
\end{thm}
\begin{proof}
Omitted.
\end{proof}

\section{Essential vectors for Speh representations}\label{s.ess_speh}
We continue to use the notations in the previous section.
Recall that $\psi$ is unramified, i.e., $\psi$ is trivial on $\oo$ but non-trivial on $\pp^{-1}$.
Let $\pi$ be an irreducible tempered representation of $G_n$, 
and set $\sigma = \Sp(\pi,m)$. 
In this section, we define a notion of essential vectors, 
and prove Theorem \ref{main1} for Speh representations.

\subsection{Essential vectors}
The following theorem is a generalization of \cite[(4.1) Th{\'e}or{\`e}me]{JPSS2}. 
\begin{thm}\label{essential}
Let the notation be as above.   
There exists a unique function 
$W^\ess_\sh \in \WW^\psi_\sh(\sigma)$ such that 
\begin{enumerate}
\item
$W^\ess_\sh(g \cdot \iota(k)) = W^\ess_\sh(g)$ for any $g \in G$ and $k \in K'$; 
\item
for all $s \in \C$ and $\ub{x} = (x_{i,j}) \in M_{m,n-1}(\C)$ with $x_{i,j} \in \C^\times$, 
\[
Z_m(s,W^\ess_\sh, W^0_\sh(\ub{x})) 
= \prod_{i=1}^m\prod_{j = 1}^{n-1} L\left( s+s_{i,j}-\half{m-1}, \pi \right),
\]
where $s_{i,j}$ is a complex number such that $q^{-s_{i,j}} = x_{i,j}$.  
\end{enumerate}
\end{thm}

\begin{defi}
We call the unique function $W^\ess_\sh$ the \emph{essential vector} of $\WW^\psi_\sh(\sigma)$. 
\end{defi}
\par

First, we consider the existence. 
Here, we show it only when $L(s,\pi) = 1$. 
The general case will be proven in Section \ref{sec.existence} below. 
\par

\begin{proof}[Proof of the existence statement in Theorem \ref{essential} when $L(s,\pi)=1$]
Note that $Q' = S'V'$ is conjugate to a standard parabolic subgroup of $G'$ by an element of $K'$. 
Hence we have the Iwasawa decomposition $G' = Q'K'$.
Define a smooth function $\varphi$ of $D = V \iota(G')$ by $\Supp(\varphi) = V \iota(K')$ 
and $\varphi(u \cdot \iota(k)) = \Psi(u)$ for $u \in V$ and $k \in K'$.
Then $\varphi \in \ind_{V}^{D}(\Psi)$ and 
\[
\int_{V' \bs G'} \varphi(g) W^0_\sh(g; \ub{x})|\det g|^{s-\half{m}} dg = 1
\]
for all $s \in \C$ and $\ub{x} = (x_{i,j}) \in M_{m,n-1}(\C)$ with $x_{i,j} \in \C^\times$. 
By \cite[Corollary 3.15]{LM}, 
one can take $W_\sh \in \WW^\psi_\sh(\sigma)$ such that $W_\sh|_D = \varphi$. 
Then $Z_m(s,W_\sh,W^0_\sh(\ub{x})) = 1$ holds
for all $s \in \C$ and $\ub{x} = (x_{i,j}) \in M_{m,n-1}(\C)$ with $x_{i,j} \in \C^\times$. 
By replacing $W_\sh$ with 
\[
\int_{K'}W_\sh(g \cdot \iota(k))dk, 
\]
we may assume that $W_\sh$ is right $\iota(K')$-invariant. 
Then $W_\sh$ satisfies the conditions in Theorem \ref{essential}.
This completes the proof of the existence statement in Theorem \ref{essential} 
when $L(s,\pi) = 1$. 
\end{proof}

We now prove the uniqueness statement (in general). 
\begin{proof}[Proof of the uniqueness statement in Theorem \ref{essential}]
Let $L^2(V' \bs G'; \Psi)$ denote the space of functions $\varphi$ on $G'$ 
such that 
$\varphi(vg) = \Psi(v)\varphi(g)$ for $v \in V'$ and $g \in G'$, 
and $\varphi$ is square-integrable on $V' \bs G'$. 
Define $\Pi$ to be  the closure of the subspace of $L^2(V' \bs G'; \Psi)$ 
consisting of smooth functions $\varphi_\sh$ of $G'$ such that 
\[
\varphi_\sh(g) = \int_{(V' \cap N') \bs V'} \varphi_\ze(vg) \Psi(v)^{-1}dv
\]
for some smooth function $\varphi_\ze$ which satisfies
$\varphi_\ze(ug) = \Psi(u)\varphi_\ze(g)$ for $u \in N'$ and $g \in G'$. 

\begin{lem}\label{varphi=0}
Let $\varphi$ be a smooth function on $G'$ such that 
\begin{enumerate}
\item
$\varphi \in \Pi$; 
\item
$\varphi(gk) = \varphi(g)$ for $g \in G'$ and $k \in K'$; 
\item
for any $\ub{x} = (x_{i,j}) \in M_{m,n-1}(\C)$ with $x_{i,j} \in \C^\times$, 
\[
\int_{V' \bs G'}\varphi(g) W^0_\sh(g; \ub{x})dg = 0.
\]
\end{enumerate}
Then $\varphi = 0$.
\end{lem}
\begin{proof}
This is an analogue of \cite[(3.5) Lemme]{JPSS2}. 
Consider the direct integral expression 
of the unitary representation $\Pi$ of $G'$: 
\[
\Pi \cong \int^{\oplus}_{\pi' \in \Irr_\unit(G')} \pi' d\mu(\pi'),  
\]
where 
$\Irr_\unit(G')$ is the set of equivalence classes of irreducible unitary representations of $G'$, 
and $\mu$ is a certain Borel measure on it.
For almost all $\pi'$, 
there exists a $G'$-equivariant intertwining operator $A_{\pi'} \colon \Pi \rightarrow \pi'$ 
such that 
\[
(\varphi_1, \varphi_2)_{L^2(V' \bs G'; \Psi)} = 
\int_{\pi'} (A_{\pi'} \varphi_1, A_{\pi'} \varphi_2)_{\pi'} d\mu(\pi')
\]
for $\varphi_1, \varphi_2 \in \Pi \subset L^2(V' \bs G'; \Psi)$, 
where $(\cdot,\cdot)_{\pi'}$ is a $G'$-invariant inner product on $\pi'$. 
\par

Now we assume that $\varphi \not= 0$. 
Then there exists $\pi' \in \Irr(G')$ such that $A_{\pi'}\varphi \not= 0$. 
Since $\varphi$ is right $K'$-invariant, 
$A_{\pi'}\varphi$ belongs to the subspace of $\pi'$ consisting of $K'$-fixed vectors. 
Then, using Lemma \ref{fin_dim},
we see that 
$(A_{\pi'} \varphi, A_{\pi'} \varphi)_{\pi'}$ is a linear combination of integrals of the form
\[
\int_{V' \bs G'}\varphi(g) W^0_\sh(g; \ub{x})dg
\]
for some $\ub{x} = (x_{i,j}) \in M_{m,n-1}(\C)$ with $x_{i,j} \in \C^\times$. 
This contradicts Condition (3).
\end{proof}

We continue the proof of the uniqueness statement in Theorem \ref{essential}. 
Now suppose that two functions $W_1, W_2 \in \WW^\psi_\sh(\sigma)$ satisfy 
the conditions of Theorem \ref{essential}. 
Set $W = W_1-W_2$, which is square-integrable on $V' \bs G'$ by Theorem \ref{inner}.  
Note that $W = \TT^\psi W_\ze$ for some $W_\ze \in \WW^\psi_\ze(\sigma)$. 
We define $V''$ to be the subgroup of $V$ consisting of $v = (v_{i,j})$ with $v_{i,j} \in M_n(F)$ 
such that $v_{i,j}$ is of the form 
\[
v_{i,j} = \begin{pmatrix}
\delta_{i,j}\1_{n-1} & u_{i,j} \\ 0 & \delta_{i,j}
\end{pmatrix}
\]
for $u_{i,j} \in F^{n-1}$. 
Then $V'$ normalizes $V''$ and $V = V' V''$. 
Hence
\begin{align*}
W(\iota(g)) 
&= \int_{(V \cap N) \bs V} W_\ze(u \cdot \iota(g)) \Psi(u)^{-1}du 
\\&= 
\int_{(V' \cap N) \bs V'} 
\left( \int_{(V'' \cap N) \bs V''} W_\ze(u \cdot \iota(vg)) \Psi(u)^{-1} du \right) 
\Psi(v)^{-1} dv.
\end{align*}
Since $\iota(G')$ normalizes $V''$, and 
since the action of $\iota(N')$ on $V''$ does not change 
the invariant measure on $(V'' \cap N) \bs V''$, 
if we set
\[
\varphi_\ze(g') = 
\int_{(V'' \cap N) \bs V''} W_\ze(u \cdot \iota(g')) \Psi(u)^{-1} du
\]
for $g' \in G'$, then $\varphi_\ze(u'g') = \Psi(u')\varphi_\ze(g')$ for $u' \in N'$ and $g' \in G'$. 
Therefore, we have $W \circ \iota \in \Pi$. 
Hence we can apply Lemma \ref{varphi=0} to $W \circ \iota$, and we obtain that $W \circ \iota = 0$. 
Since $D = V'\iota(G')$, it follows that $W|_D = 0$. 
By Theorem \ref{inner}, we conclude that $W = 0$, as desired.
This completes the proof of the uniqueness statement in Theorem \ref{essential}. 
\end{proof}

\begin{cor}\label{W|L=0}
Let $W \in \WW^\psi_\ze(\sigma)$. 
If $W$ is right $\iota(K')$-invariant, and if $W|_L = 0$, then $W = 0$.
\end{cor}
\begin{proof}
By the assumptions, one has $I_m(s, W, W^0_\ze(\ub{x})) = 0$ 
for all $s \in \C$ and $\ub{x} = (x_{i,j}) \in M_{m,n-1}(\C)$ with $x_{i,j} \in \C^\times$. 
By the same argument as in the proof of the uniqueness statement in 
Theorem \ref{essential}, 
we have $\TT^\psi W = 0$, hence $W = 0$. 
\end{proof}

As an application, we have a part of Theorem \ref{main1} for Speh representations.
Recall from Example \ref{ex_lambda} (4) that 
\[
\lambda_\sigma 
= (\underbrace{0,\dots,0}_{(n-1)m}, \underbrace{c_\pi, \dots, c_\pi}_m) \in \Lambda_{nm},
\] 
where $c_\pi$ is the conductor of $\pi$. 

\begin{prop}\label{dim=0}
Let $\lambda \in \Lambda_{nm}$. 
If $\lambda < \lambda_\sigma$, 
then $\sigma^{\K_{nm,\lambda}} = 0$. 
\end{prop}
\begin{proof}
If $\lambda < \lambda_\sigma$, then the first $(n-1)m$ 
components of $\lambda$ are $0$. 
Hence there exists a compact subgroup $K_\lambda$ of $G$ conjugate to $\K_{nm,\lambda}$ such that 
\begin{itemize}
\item
$K_\lambda \supset \iota(K')$; 
\item
$K_\lambda \cap L \supset \K_{n,\lambda_1} \times \dots \times \K_{n,\lambda_m}$ 
with $\lambda_i \in \Lambda_n$ of the form $\lambda_i = (0,\dots,0,a_i)$ 
such that $0 \leq a_i < c_\pi$ for some $1 \leq i \leq m$.
\end{itemize}
Let $W \in \WW^\psi_\ze(\sigma)^{K_\lambda}$. 
Since $\pi^{\K_{n,\lambda_i}} = 0$ by \cite[(5.1) Th{\'e}or{\`e}me]{JPSS2}, 
we see that $W|_L = 0$. 
It follows from Corollary \ref{W|L=0} that $W = 0$. 
Hence $\sigma^{\K_{nm,\lambda}} \cong \sigma^{K_\lambda} = 0$. 
\end{proof}

\subsection{Properties of essential vectors}
Recall that $G = \GL_{nm}(F)$ and $K= \GL_{nm}(\oo)$. 
For a positive integer $a$, define $K(a) \subset K$ to be 
the subgroup consisting of
$k = (k_{i,j})_{1 \leq i,j \leq m} \in K$ with $k_{i,j} \in M_{n}(\oo)$ such that 
the last row of $k_{i,j}$ is congruent to $(0,\dots,0,\delta_{i,j}) \bmod \pp^a$ for $1 \leq i,j \leq m$. 
Put another way, 
if we denote by $D(\oo/\pp^a)$ the image of $D \cap K$ 
under $K \rightarrow \GL_{nm}(\oo/\pp^a)$, 
then $K(a)$ is the inverse image of $D(\oo/\pp^a)$.
Note that $K(a)$ is conjugate to $\K_{nm,\lambda}$ with 
\[
\lambda = (\underbrace{0,\dots,0}_{(n-1)m},\underbrace{a,\dots,a}_m) \in \Lambda_{nm}
\]
by an element of $K$. 
\par

Let $\pi$ be an irreducible tempered representation of $G_n$, 
and set $\sigma = \Sp(\pi,m)$. 
We prove the following proposition in this subsection.  
It together with Proposition \ref{dim=0}
contains Theorem \ref{main1} for $\sigma$ when $L(s,\sigma) = 1$.
\begin{prop}\label{dim=1}
Suppose that $L(s, \pi) = 1$. 
Then $\WW^\psi_\sh(\sigma)^{K(c_\pi)}$ is the one-dimensional vector space
spanned by 
the essential vector $W^\ess_\sh$. 
\end{prop}

The proof of Proposition \ref{dim=1} is analogous to that of \cite[(5.1) Th{\'e}or{\`e}me]{JPSS2}. 
Suppose that $L(s, \pi) = 1$. 
\par

For $d \in \Z$ and $W_\sh \in \WW^\psi_\sh(\sigma)$, we consider 
\[
Z_{m,d}(W_\sh; \ub{x}) 
= \int_{V' \bs \{g \in G' \;|\; |\det g| = q^{-d}\}} 
W_\sh(\iota(g)) W_\sh^0(g; \ub{x})|\det g|^{-\half{m}} dg. 
\]
Note that 
\[
Z_{m,d}(W_\sh; x\ub{x})
= x^d Z_{m,d}(W_\sh; \ub{x}),  
\]
where $x\ub{x} = (xx_{i,j})$ if $\ub{x} = (x_{i,j})$.

\begin{lem}\label{d(W)}
There is an integer $d(W_\sh)$ such that $Z_{m,d}(W_\sh; \ub{x}) = 0$ 
for any $d < d(W_\sh)$ and $\ub{x} = (x_{i,j}) \in M_{m,n-1}(\C)$ with $x_{i,j} \in \C^\times$. 
\end{lem}
\begin{proof}
By (the proof of) Proposition \ref{I=Z}, 
it is enough to show an analogous assertion for $I_m(s,W_\ze, W^0_\ze(\ub{x}))$ 
with  $W_\ze \in \WW^\psi_\ze(\sigma)$. 
Let $g \in G'$ with $|\det g| = q^{-d}$ such that $W_\ze(\iota(g)) W^0_\ze(g; \ub{x}) \not= 0$.
We take $k' \in K'$, $u' \in N'$ and $a_1, \dots, a_{(n-1)m} \in \Z$ such that
\[
g = u'
\begin{pmatrix}
\varpi^{a_1} && \\ 
& \ddots & \\
&& \varpi^{a_{(n-1)m}}
\end{pmatrix}
k'.
\]
Since 
\[
W^0_\ze(g; \ub{x}) = C \prod_{j=1}^mW^0\left(
\begin{pmatrix}
\varpi^{a_{(n-1)(j-1)+1}} && \\ 
& \ddots & \\
&& \varpi^{a_{(n-1)j}}
\end{pmatrix}; x_{i,1}, \dots, x_{i,n-1}\right)
\]
for some $C \not= 0$ (depending on $a_1, \dots, a_{(n-1)m}$), 
we must have $a_{(n-1)(j-1)+1} \geq \dots \geq a_{(n-1)j}$ for any $1 \leq j \leq m$.
In particular, $d = \sum_{i=1}^{(n-1)m}a_i \geq \sum_{j=1}^{m}(n-1)a_{(n-1)j}$.
\par

For $l \geq 0$, 
let $V''(\pp^l)$ be the subgroup of $K$ 
consisting of $(k_{i,j})_{1 \leq i,j \leq m}$ with $k_{i,j} \in M_n(\oo)$ 
such that $k_{i,j}$ is of the form 
\[
k_{i,j} = \begin{pmatrix}
\delta_{i,j}\1_{n-1} & u_{i,j} \\ 0 & \delta_{i,j}
\end{pmatrix}
\]
for $u_{i,j} \in (\pp^l)^{n-1}$. 
Since $W_\ze$ is smooth, 
one can take sufficiently large $l$ 
such that $W_\ze$ is right $V''(\pp^l)$-invariant. 
Note that $\iota(K')$ acts on $V''(\pp^l)$. 
In particular, for any $z_1, \dots, z_m \in \pp^l$, 
we can take $u \in V''(\pp^l)$ such that 
$(\iota(k') \cdot u \cdot \iota(k')^{-1})_{nj-1,nj} = z_j$ for $1 \leq j \leq m$. 
Then we have 
\begin{align*}
W_\ze(\iota(g)) &= W_\ze\left(\iota(u')
\iota\begin{pmatrix}
\varpi^{a_1} && \\ 
& \ddots & \\
&& \varpi^{a_{n-1}}
\end{pmatrix}
\iota(k')u\right)
\\&= \left(\prod_{j=1}^m\psi(\varpi^{a_{(n-1)j}}z_j)\right) W_\ze(\iota(g)).
\end{align*}
Since $z_1, \dots, z_m \in \pp^l$ are arbitrary, 
if $W_\ze(\iota(g)) \not= 0$, then we must have $a_{(n-1)j} \geq -l$ for $1 \leq j \leq m$.
In conclusion, we have $d \geq \sum_{j=1}^{m}(n-1)a_{(n-1)j} \geq -(n-1)ml$.
This completes the proof of the lemma.
\end{proof}

By the proof of this lemma, one can take $d(W_\sh) = -(n-1)m l$ 
if $W_\sh$ is right $V''(\pp^l)$-invariant. 
In particular, if $W_\sh \in \WW^\psi_\sh(\sigma)^{K(a)}$, 
then we can take $d(W_\sh) = 0$ and $d(\wt{W}_\sh) = -(n-1)ma$.
\par

Now, if we set $x = q^{-s}$, we have 
\[
Z_m(s,W_\sh,W^0_\sh(\ub{x})) 
= \sum_{d \in \Z}x^d Z_{m,d}(W_\sh; \ub{x})
= \sum_{d \geq d(W_\sh)}x^d Z_{m,d}(W_\sh; \ub{x}).
\]
If we replace $\pi$, $W_\sh$ and $\psi$ with $\wt\pi$, $\wt{W}_\sh$ and $\psi^{-1}$, respectively, 
since $\wt{W}^0_\sh(\ub{x}) = W^0_\sh(\ub{x}^{-1})$ (with respect to $\psi$),
we have
\[
Z_m(s,\wt{W}_\sh,\wt{W}^0_\sh(\ub{x})) 
= \sum_{d \geq d(\wt{W}_\sh)}x^d Z_{m,d}(\wt{W}_\sh; \ub{x}^{-1}), 
\]
hence
\[
Z_m(m-s,\wt{W}_\sh,\wt{W}^0_m(\ub{x})) 
= \sum_{d \geq d(\wt{W}_\sh)}q^{-m}x^{-d} Z_{m,d}(\wt{W}_\sh; \ub{x}^{-1}). 
\]
By the functional equation (Theorem \ref{RS_unram} (3)), 
using the assumption that $L(s,\pi) = 1$,
we have
\begin{align*}
&\sum_{d \geq d(\wt{W}_\sh)}q^{-m}x^{-d} Z_{m,d}(\wt{W}_\sh; \ub{x}^{-1})
\\&= \left(\prod_{i=1}^m\prod_{j=1}^{n-1} \ep\left( s+s_{i,j}-\half{m-1}, \pi, \psi \right) \right)
\sum_{d \geq d(W_\sh)}x^d Z_{m,d}(W_\sh; \ub{x})
\end{align*}
as a formal power series of $x$, 
where $s_{i,j}$ is a complex number such that $x_{i,j} = q^{-s_{i,j}}$. 
If we write $\ep(s,\pi, \psi) = \ep_0 q^{-c_\pi s} = \ep_0 x^{c_\pi}$, we have
\[
\prod_{i=1}^m\prod_{j=1}^{n-1} \ep \left(s+s_{i,j}-\half{m-1}, \pi, \psi \right)
= 
\ep_0^{m(n-1)} q^{\half{c_\pi m(m-1)(n-1)}} x^{c_\pi m(n-1)} 
\prod_{i=1}^m \prod_{j=1}^{n-1}x_{i,j}^{c_\pi}.
\]
In particular, if $d > d'(W_\sh) \coloneqq -c_\pi m(n-1)-d(\wt{W}_\sh)$, 
we must have $Z_{m,d}(W_\sh; \ub{x}) = 0$. 
Hence we obtain the following. 

\begin{prop}\label{poly}
Assume that $L(s,\pi) = 1$. 
Write $\ep(s,\pi, \psi) = \ep_0 q^{-c_\pi s}$. 
For $W_\sh \in \WW^\psi_\sh(\sigma)$, 
let $d(W_\sh)$ and $d(\wt{W}_\sh)$ be the constants in Lemma \ref{d(W)}, 
and set $d'(W_\sh) = -c_\pi m(n-1)-d(\wt{W}_\sh)$. 
Then 
\[
Z_m(s,W_\sh,W^0_\sh(\ub{x})) 
= \sum_{d(W_\sh) \leq d \leq d'(W_\sh)} x^d Z_{m,d}(W_\sh; \ub{x})
\]
is a finite sum.
Moreover, we have a functional equation
\begin{align*}
&\sum_{d(\wt{W}_\sh) \leq d \leq d'(\wt{W}_\sh)}q^{-m}x^{-d} Z_{m,d}(\wt{W}_\sh; \ub{x}^{-1})
\\&= 
\ep_0^{m(n-1)} q^{\half{c_\pi m(m-1)(n-1)}} x^{c_\pi m(n-1)} 
\left(\prod_{i=1}^m \prod_{j=1}^{n-1}x_{i,j}^{c_\pi}\right)
\sum_{d(W_\sh) \leq d \leq d'(W_\sh)}x^d Z_{m,d}(W_\sh; \ub{x}). 
\end{align*}
\end{prop}

Now we prove Proposition \ref{dim=1}. 
\begin{proof}[Proof of Proposition \ref{dim=1}]
First, we show that the essential vector $W^\ess_\sh$ is $K(c_\pi)$-invariant. 
By Proposition \ref{poly}, we have
\[
Z_m(m-s, \wt{W}^\ess_\sh, \wt{W}^0_\sh(\ub{x}))
= \ep_0^{m(n-1)} q^{\half{c_\pi m(m-1)(n-1)}} x^{c_\pi m(n-1)} 
\left(\prod_{i=1}^m \prod_{j=1}^{n-1}x_{i,j}^{c_\pi}\right)
\] 
where $x = q^{-s}$. 
Set $a = \varpi^{c_\pi}\1_{(n-1)m}$ which is in the center of $G'$.
We notice that $|\det a^{-1}|^{\half{m}-s} = q^{\half{c_\pi m^2(n-1)}}x^{c_\pi m(n-1)}$ and 
\[
\wt{W}^0_\sh(ga^{-1}; \ub{x}) 
= \left(\prod_{i=1}^m\prod_{j=1}^{n-1}x_{i,j}^{c_\pi}\right) \wt{W}^0_\sh(g; \ub{x}). 
\]
If we define $W'_\sh \in \WW^{\psi^{-1}}_\sh(\wt{\sigma})$ by
\[
W'_\sh(g) = W_\sh^\ess(g \cdot \iota(a)), 
\]
then it is right $\iota(K')$-invariant and 
\[
Z_m(m-s, W'_\sh, \wt{W}^0_\sh(\ub{x})) 
= q^{\half{c_\pi m^2(n-1)}}x^{c_\pi m(n-1)} \left(\prod_{i=1}^m\prod_{j=1}^{n-1}x_{i,j}^{c_\pi}\right). 
\]
By Lemma \ref{varphi=0} and Theorem \ref{inner},
we see that $\wt{W}^\ess_\sh = C W'_\sh$ for some constant $C$. 
Hence 
\[
W^\ess_\sh(g) 
= C \wt{W}'_\sh(g) 
= C W^\ess_\sh(w_{nm}{}^tg^{-1}w'_n \cdot \iota(\varpi^{c_\pi} \1_{(n-1)m})).
\]
Since $W^\ess_\sh$ is right $V''(\oo)$-invariant, 
it follows that $W^\ess_\sh$ is right ${}^tV''(\pp^{c_\pi})$-invariant.
Therefore, we conclude that $W^\ess_\sh$ is right $K(c_\pi)$-invariant.
\par

Next, we show that $\dim(\sigma^{K(c_\pi)}) = 1$. 
If $\WW_\ze \in \WW^\psi_\ze(\sigma)^{K(c_\pi)}$, 
we have
\[
W_\ze|_L \in 
\left(\bigotimes_{i=1}^m\WW^\psi(\pi|\cdot|^{\half{(m+1-2i)(n-1)}}) \right)^{K(c_\pi) \cap L}, 
\]
where the right-hand side is one-dimensional, 
and is spanned by the tensor product of essential vectors. 
Hence $Z_m(s,\TT^{\psi} W_\ze,W^0_\sh(\ub{x})) = I_m(s,W_\ze,W^0_\ze(\ub{x}))$
does not depend on $s \in \C$ and $\ub{x} = (x_{i,j}) \in M_{m,n-1}(\C)$ with $x_{i,j} \in \C^\times$. 
Using the uniqueness statement in Theorem \ref{essential}, 
we conclude that $\TT^{\psi} W_\ze$ is a constant multiple of $W^\ess_\sh$. 
\end{proof}

Since $\K_{nm,\lambda_\sigma}$ is conjugate to $K(c_\pi)$, 
by Propositions \ref{dim=0} and \ref{dim=1}, 
we complete the proof of Theorem \ref{main1} for $\sigma = \Sp(\pi,m)$ such that $L(s,\pi) = 1$.
As explained in Section \ref{sec.reduction}, 
this together with results in Sections \ref{sec.pf_ladder}, \ref{sec.pf_conj_unip}
and Lemma \ref{thm.reduction}
completes Theorem \ref{main1} in all cases. 
\par

To prove Theorem \ref{essential} in Section \ref{sec.existence}, 
we use the following special case of Theorem \ref{main1}.

\begin{cor}\label{cor.speh}
Let $\pi$ be an irreducible tempered representation of $G_n$, 
and set $\sigma = \Sp(\pi,m)$. 
Then we have
\[
\dim(\sigma^{\K_{nm,\lambda}}) = 
\left\{
\begin{aligned}
&1 \iif \lambda = \lambda_\sigma,\\
&0 \iif \lambda < \lambda_\sigma.
\end{aligned}
\right. 
\]
\end{cor}

\subsection{Proof of Theorem \ref{essential}: the case where $L(s,\pi) \neq 1$}
\label{sec.existence}
Finally, we prove the existence statement in Theorem \ref{essential} in general. 
Before doing it, we state the following consequence of Corollary \ref{cor.speh}. 

\begin{cor}\label{prop:ess}
Let $\pi$ be an irreducible tempered representation of $G_n$, 
and set $\sigma = \Sp(\pi,m)$. 
Then the restriction map $W_\ze \mapsto W_\ze|_L$ gives an isomorphism of 
one-dimensional vector spaces
\[
\WW_\ze^\psi(\sigma)^{K(c_\pi)} \xrightarrow{\cong} 
\WW^{\psi}(\pi|\cdot|^{\half{(m-1)(n-1)}})^{K(c_\pi)} 
\otimes \dots \otimes 
\WW^{\psi}(\pi|\cdot|^{-\half{(m-1)(n-1)}})^{K(c_\pi)}.
\]
\end{cor}
\begin{proof}
Since the compact open subgroup $K(c_\pi)$ is conjugate to 
$\K_{nm,\lambda_\sigma}$, 
we conclude that $\sigma^{K(c_\pi)}$ is one-dimensional. 
By Lemma \ref{W|L=0}, 
the restriction map $W_\ze \mapsto W_\ze|_L$ is injective on $\WW_\ze^\psi(\sigma)^{K(c_\pi)}$. 
Since the image is in 
\[
\left(\bigotimes_{i=1}^m \WW^{\psi}(\pi|\cdot|^{\half{(m+1-2i)(n-1)}})\right)^{K(c_\pi) \cap L}
\]
which is one-dimensional, 
we obtain the desired isomorphism. 
\end{proof}

\begin{proof}[Proof of the existence statement in Theorem \ref{essential}]
By Lemma \ref{iwasawa} and Corollary \ref{prop:ess} 
together with \cite[(4.1) Th{\'e}or{\`e}me]{JPSS2}, 
we can find $W_\ze^\ess \in \WW_\ze^\psi(\sigma)^{K(c_\pi)}$ such that
\[
I_m(s,W_\ze,W^0_\ze(\ub{x}))
= \prod_{i=1}^m\prod_{j = 1}^{n-1} L\left( s+s_{i,j}-\half{m-1}, \pi \right).
\]
Then $W_\sh^\ess = \TT^\psi \WW_\ze^\ess$ satisfies the conditions in Theorem \ref{essential}. 
\end{proof}



\begin{thebibliography}{30}
\bibitem{AF}
{F.\ W.\ Anderson and K.\ R.\ Fuller},
{\em Rings and categories of modules. Second edition.} 
Graduate Texts in Mathematics, {\bf13}. \emph{Springer-Verlag, New York}, 1992. x+376 pp.

\bibitem{SGA4}
{M.\ Artin, A.\ Grothendieck and J.-L.\ Verdier}, 
{\em Th\'eorie des topos et cohomologie \'etale des sch\'emas.
Tome 1: Th\'eorie des topos.} 
Lecture Notes in Mathematics, Vol. {\bf269}. 
\emph{Springer-Verlag, Berlin-New York}, 1972. xix+525 pp.


\bibitem{AL}
{A.\ O.\ L.\ Atkin and J.\ Lehner},
{\em Hecke operators on $\Gamma_0(m)$.}
\emph{Math.\ Ann.} {\bf 185} (1970), 134--160.

\bibitem{BZ}
{I.~N.~Bernstein and A.~V.~Zelevinsky}, 
{\em Induced representations of reductive $\mathfrak{p}$-adic groups}.  
\emph{Ann.\ Sci.\ {\'E}cole Norm.\ Sup. (4)} {\bf10} (1977), no.~4, 441--472.

\bibitem{CFGK}
{Y.\ Cai, S.\ Friedberg, D.\ Ginzburg and E.\ Kaplan}, 
{\em Doubling constructions and tensor product $L$-functions: the linear case}. \emph{Invent. Math.} {\bf217} (2019), no.~3, 985--1068.

\bibitem{C}
{W.\ Casselman},
{\em On some results of Atkin and Lehner.}
\emph{Math.\ Ann.} {\bf 201} (1973), 301--314.

\bibitem{Ga}
{P.\ Gabriel},
{\em Unzerlegbare  Darstellungen I.}
\emph{Manuscripta Math.} {\bf 6} (1972), 71--103.

\bibitem{GJ}
{R.\ Godement and H.\ Jacquet},
{\em Zeta functions of simple algebras.}
Lecture Notes in Mathematics, Vol. {\bf260}. 
\emph{Springer-Verlag, Berlin-New York}, 1972. ix+188 pp. 

\bibitem{G}
{B.\ H.\ Gross}, 
{\em On the Langlands correspondence for symplectic motives.} 
\emph{Izv. Ross. Akad. Nauk Ser. Mat.} {\bf80} (2016), no.~4, 49--64; 
translation in 
\emph{Izv. Math.} 80 (2016), no.~4, 678--692. 

\bibitem{HC} 
{Harish-Chandra},
{\em Harmonic analysis on reductive $p$-adic groups.} 
\emph{Harmonic analysis on homogeneous spaces 
(Proc. Sympos. Pure Math., Vol. XXVI, Williams Coll., Williamstown, Mass., 1972),} 
pp. 167--192. 
\emph{Amer. Math. Soc., Providence, R.I.,} 1973. 

\bibitem{J}
{H.\ Jacquet},
{\em A correction to {\it Conducteur des repr{\'e}sentations du groupe lin{\'e}aire}.}
\emph{Pacific J.\ Math.} {\bf260} (2012), no.~2, 515--525.

\bibitem{JPSS2}
{H.\ Jacquet, I.\ I.\ Piatetski-Shapiro and J.\ Shalika}, 
{\em Conducteur des repr{\'e}sentations du groupe lin{\'e}aire.} 
\emph{Math.\ Ann.} {\bf256} (1981), no.~2, 199--214.

\bibitem{JPSS}
{H.\ Jacquet, I.\ I.\ Piatetski-Shapiro and J.\ Shalika},
{\em Rankin-Selberg convolutions}. 
\emph{Amer. J. Math.} {\bf105} (1983), no.~2, 367--464.

\bibitem{HT}
{M.\ Harris and R.\ Taylor}, 
{\em The geometry and cohomology of some simple Shimura varieties.} 
With an appendix by Vladimir G. Berkovich. 
Annals of Mathematics Studies, {\bf151}. 
\emph{Princeton University Press, Princeton, NJ,} 2001. viii+276 pp.

\bibitem{H}
{G.\ Henniart}, 
{\em Une preuve simple des conjectures de Langlands pour $\GL(n)$ sur un corps $p$-adique.} 
\emph{Invent. Math.} {\bf139} (2000), no.~2, 439--455.

\bibitem{KZ}
{H.\ Knight and A.\ V.\ Zelevinsky},
{\em Representations of quivers of type $A$ and the multisegment duality.}
\emph{Adv. Math.} 117 (1996), no.~2, 273--293.

\bibitem{KY0}
{S.\ Kondo and S.\ Yasuda}, 
{\em Local $L$ and epsilon factors in Hecke eigenvalues.}
\emph{J.\ Number Theory} {\bf132} (2012), 1910--1948.

\bibitem{KY}
{S.\ Kondo and S.\ Yasuda}, 
{\em Distribution and Euler systems for the general linear group.}
Preprint (2018), arXiv:1801.04817v2. 

\bibitem{LR}
{J.\ Lansky and A.\ Raghuram}, 
{\em Conductors and newforms for $\mathrm{U}(1,1)$.} 
\emph{Proc. Indian Acad. Sci. Math. Sci.} {\bf114} (2004), no.~4, 319--343.

\bibitem{LM}
{E.\ M.\ Lapid and Z.\ Mao}, 
{\em Local Rankin--Selberg integrals for Speh representations.} 
\emph{Compos.\ Math.} {\bf156} (2020), no.~5, 908--945.

\bibitem{LMi}
{E.\ M.\ Lapid and A.\ M\'{i}nguez},
{\em On a determinantal formula of Tadi\'{c}.}
\emph{Amer.\ J.\ Math.} {\bf 136} (2014), no.~1, 111--142.

\bibitem{LMi2}
{E.\ M.\ Lapid and A.\ M\'{i}nguez},
{\em On parabolic induction on inner forms of the general
linear group over a non-archimedean local field.}
\emph{Selecta Math. (N.S.)} {\bf22} (2016), no.~4, 2347--2400. 

\bibitem{Li}
{W.\ C.\ W. Li}, 
{\em Newforms and functional equations.} 
\emph{Math. Ann.} {\bf212} (1975), 285--315.

\bibitem{MM}
{S.\ Mac Lane and I.\ Moerdijk}, 
{\em Sheaves in geometry and logic.} 
Universitext. 
\emph{Springer-Verlag, New York,} 1994. xii+629 pp.

\bibitem{M}
{I.\ G.\ Macdonald},
{\em Symmetric functions and Hall polynomials. Second edition.} 
With contributions by A. Zelevinsky. 
Oxford Mathematical Monographs. Oxford Science Publications. 
\emph{The Clarendon Press, Oxford University Press, New York,} 1995. x+475 pp.

\bibitem{Mat}
{N.\ Matringe}, 
{\em Essential Whittaker functions for $GL(n)$.} 
\emph{Doc. Math.} {\bf18} (2013), 1191--1214. 

\bibitem{M1}
{M.\ Miyauchi}, 
{\em On epsilon factors attached to supercuspidal representations of unramified $\mathrm{U}(2,1)$}. 
\emph{Trans. Amer. Math. Soc.} {\bf365} (2013), no.~6, 3355--3372.

\bibitem{M2}
{M.\ Miyauchi}, 
{\em On local newforms for unramified $\mathrm{U}(2,1)$}. 
\emph{Manuscripta Math.} {\bf141} (2013), no.~1-2, 149--169.

\bibitem{M3}
{M.\ Miyauchi}, 
{\em Conductors and newforms for non-supercuspidal representations of unramified $\mathrm{U}(2,1)$}. 
\emph{J. Ramanujan Math. Soc.} {\bf28} (2013), no.~1, 91--111.

\bibitem{M4}
{M.\ Miyauchi}, 
{\em On $L$-factors attached to generic representations of unramified $\mathrm{U}(2,1)$}. 
\emph{Math. Z.} {\bf289} (2018), no.~3-4, 1381--1408.

\bibitem{MW0}
{C.\ M{\oe}glin and J.-L.\ Waldspurger}, 
{\em Sur l'involution de Zelevinski}. 
\emph{J.\ Reine Angew.\ Math.} {\bf 372} (1986), 136--177.

\bibitem{MW}
{C.\ M{\oe}glin and J.-L.\ Waldspurger}, 
{\em Mod{\`e}les de Whittaker d{\'e}g{\'e}n{\'e}r{\'e}s pour des groupes $p$-adiques}. 
\emph{Math.\ Z.} {\bf196} (1987), no.~3, 427--452.

\bibitem{Ok}
{T.\ Okazaki}, 
{\em Local Whittaker-newforms for $GSp(4)$ matching to Langlands parameters}.
Preprint (2019), arXiv:1902.07801v2.

\bibitem{P}
{K.\ Procter}, 
{\em Parabolic induction via Hecke algebras and the Zelevinsky duality conjecture}. 
\emph{Proc. London Math. Soc. (3)} {\bf77} (1998), no.~1, 79--116. 

\bibitem{RS1}
{B.\ Roberts and R.\ Schmidt}, 
{\em Local newforms for $\mathrm{GSp}(4)$}. 
Lecture Notes in Mathematics, 1918. 
\emph{Springer, Berlin}, 2007. viii+307 pp.

\bibitem{RS2}
{B.\ Roberts and R.\ Schmidt}, 
{\em On the number of local newforms in a metaplectic representation}. 
\emph{Arithmetic geometry and automorphic forms}, 505--530, 
Adv. Lect. Math. (ALM), 19, 
\emph{Int.\ Press, Somerville, MA}, 2011.

\bibitem{Rodier}
{F.\ Rodier},
{\em Repr\'esentations de $\GL(n,k)$ o\`u $k$ est un corps $p$-adique}.
\emph{Bourbaki Seminar, Vol. 1981/1982}, 201--218,
Ast\'erisque, 92--93, 
\emph{Soc.\ Math.\ France, Paris}, 1982.

\bibitem{T}
{M.\ Tadi\'c},
{\em Induced representations of $\GL(n, A)$ for $p$-adic division algebras $A$}.\emph{J.\ Reine Angew.\ Math.} {\bf 405} (1990), 48--77.

\bibitem{Tdet}
{M.\ Tadi\'c},
{\em On characters of irreducible unitary representations of general linear groups}.
\emph{Abh.\ Math.\ Sem.\ Univ.\ Hamburg} {\bf 65} (1995), 341--363. 

\bibitem{Tate}
{J. Tate}, 
{\em Number theoretic background.} 
\emph{Automorphic forms, representations and L-functions (Proc. Sympos. Pure Math., Oregon State Univ., Corvallis, Ore., 1977), Part 2}, pp.~3--26, 
Proc. Sympos. Pure Math., XXXIII, 
\emph{Amer. Math. Soc., Providence, R.I.}, 1979. 

\bibitem{Ts}
{P.-Y.\ Tsai}, 
{\em On Newforms for Split Special Odd Orthogonal Groups}. 
PhD Thesis, Harvard University. 2013. 143 pp.

\bibitem{Z}
{A. V. Zelevinsky}, 
{\em Induced representations of reductive $\mathfrak{p}$-adic groups. II. On irreducible representations of $\GL(n)$}.  
\emph{Ann. Sci. {\'E}cole Norm. Sup. (4)} {\bf13} (1980), no.~2, 165--210.

\bibitem{Z2}
{A. V. Zelevinsky}, 
{\em $p$-adic analogue of the Kazhdan-Lusztig hypothesis}.
\emph{Funct.\ Anal.\ Appl.} {\bf 15} (1981), 83--92.

\end{thebibliography}
\end{document}